\documentclass{amsart}
\usepackage{amsmath, amssymb, amsfonts, graphicx, hyperref}

\def\Ric{\mathop{\rm Ric}}
\def\Riem{\mathop{\rm Rm}}
\def\KNP{\mathop{\makebox{$\wedge$}\hspace{-8.4pt}\bigcirc}}
\def\Ric{\mathop{\rm Ric}}
\def\cRic{\mathop{\rm R\text{\i}\makebox[0pt]{\raisebox{5pt}{\tiny$\circ$\;\,}}c}}

\makeatletter
\newcommand{\umathbb}[1]{{\mathpalette\dude@umathbb{#1}}}
\newcommand{\dude@umathbb}[2]{%
	\raisebox{\depth}{\rotatebox{180}{$\m@th#1\mathbb{#2}$}}%
}
\newcommand{\rmathbb}[1]{{\mathpalette\dude@rmathbb{#1}}}
\newcommand{\dude@rmathbb}[2]{%
	\scalebox{-1}[1]{$\m@th#1\rm{#2}$}%
}
\newcommand{\fmathbb}[1]{{\mathpalette\dude@fmathbb{#1}}}
\newcommand{\dude@fmathbb}[2]{%
	\raisebox{\depth}{\scalebox{1}[-1]{$\m@th#1\mathbb{#2}$}}%
}
\makeatother

\newtheorem{theorem}{Theorem}[section]
\newtheorem{proposition}[theorem]{Proposition}
\newtheorem{lemma}[theorem]{Lemma}
\newtheorem{corollary}[theorem]{Corollary}
\newtheorem{conjecture}[theorem]{Conjecture}

\title[Scalar-flat K\"ahler 4-manifolds with a symmetry]{Complete scalar-flat K\"ahler 4-manifolds with a continuous symmetry}
\author{Brian Weber }
\date{September 1, 2023}

\begin{document}
	
	\maketitle
	
	\begin{abstract}
		We study complete scalar-flat K\"ahler manifolds with a Killing field and a mild asymptotic condition.
		We show that topological and geometric rigidities exist that powerfully restrict the manifold's behavior at infinity.
		We create a rough classification for scalar-flat K\"ahler manifolds with a Killing field.
	\end{abstract}
	
	
	%
	%
	%
	%
	%
	%
	%
	%
	\section{Introduction}
	
	K\"ahler manifolds with zero scalar curvature and a continuous symmetry have received a great deal of attention.
	Numerous metrics of this kind are known in the compact \cite{KLP97}, \cite{ACGT04} and non-compact cases \cite{Don09} \cite{Weber23} and a variety of ans\"atze and construction methods have been built \cite{LeB89} \cite{HS02} \cite{ACG06} \cite{DF89}.
	Treatments are available from perspectives such as twistor theory, Einstein-Weyl geometry, toric K\"ahler geometry, and under asymptotic rigidity assumptions; see for example \cite{Tod95}, \cite{Dan96}, \cite{Calder00}, \cite{DunPlan11}, \cite{LV16}, \cite{Weber24a}.
	Even still, the variety of complete metrics of this type has not been understood.
	
	Without any asymptotic conditions a full classification of the complete case is probably out of reach.
	In this paper we take the asymptotic condition to be a lower bound on the norm of the Killing field.
	This condition appears naturally, for example, in the study of ALE or ALF metrics.
	Additionally, in creating singularity models for collapsing K\"ahler manifolds, an F-structure \cite{CG86} on the collapsing sequence appears.
	In the limit this leads, under the right controls, to a Killing field on a complete manifold whose norm is bounded strictly away from zero.
	
	To explain our results, let $(M^4,J,g_4,V)$ be a K\"ahler manifold where $V$ is a symplectomorphic Killing field.
	If $M^4$ is simply connected we can solve $dz=-i_V\omega$ up to a constant; if not we can always solve for $z$ locally, or pass to the universal cover where we can solve for $z$ globally.
	The function $z$ is called a \textit{Killing potential} or the \textit{momentum} for $V$.
	Level-sets $M^3_z=\{z=const\}$ are manifolds away from zeros of $V$, and the intrinsic volume form of $M^3_z$ is $dVol_3=dVol_4({\frac{\nabla{}z}{|\nabla{}z|}},\cdot,\cdot,\cdot)$.
	The Killing field $V$ passes to each $M^3_z$, where we may take a Riemannian quotient to obtain an orbifold $M^2_z$, called the \textit{K\"ahler reduction} at $z=const$.
	Each $M^2_z$ is a smooth Riemannian orbifold that inherits a metric and complex structure from $M^4$; we denote its Gaussian curvature by $K$ and its intrinsic volume form by $dVol_2$.
	By pullback, the function $K$ and $2$-form $dVol_2$ also exists on the parent manifold $M^4$.
	In the scalar-flat case, away from zeros of $V$, we prove the relations
	\begin{eqnarray}
		&&K\,dVol_4
		=
		d\left[\frac{\triangle_4z}{|V|}dVol_3\,-\,Jd\log|V|^2\wedge*dVol_2\right], \quad\text{and} \label{EqnFirstEvoEqn} \\
		&&2|\Ric{}_4|^2\,dVol_4
		=
		d\left[\mathcal{L}_{\frac{\partial}{\partial{}z}}\left(\left(\frac{\triangle_4z}{|V|}\right)^2V_\flat\wedge{}dVol_2\right)
		-Jd\left(\frac{\triangle_4z}{|V|}\right)^2\wedge*dVol_2\right]. \label{EqnSecondEvoEqn}
	\end{eqnarray}
	Eq. (\ref{EqnFirstEvoEqn}) is a global form of the ``Toda-lattice equation'' (see Section \ref{SecEvoK}) and Eq. (\ref{EqnSecondEvoEqn}) is ultimately a Chern-Gauss-Bonnet identity.
	Using an integration relation between $M^4$ and the reductions $M^2_z$ (Lemma \ref{LemmaIntegration}), the following evolution equations are more-or-less immediate consequences where $M^2_z$ is compact and $|V|\ne0$:
	\begin{eqnarray}
		&&\int_{M^2_z}K\,dVol_2
		=\frac{d}{dz}\int_{M^2_z}\triangle_4z\,dVol_2, \label{EqnFirstIntegrated} \\
		&&\int_{M^2_z}|\Ric{}_4|^2\,dVol_2
		=\frac12\frac{d^2}{dz^2}\int_{M^2_z}(\triangle_4z)^2\,dVol_2. \label{EqnSecondIntegrated}
	\end{eqnarray}
	\begin{figure}[ht]
		\centering
		\includegraphics[scale=0.3]{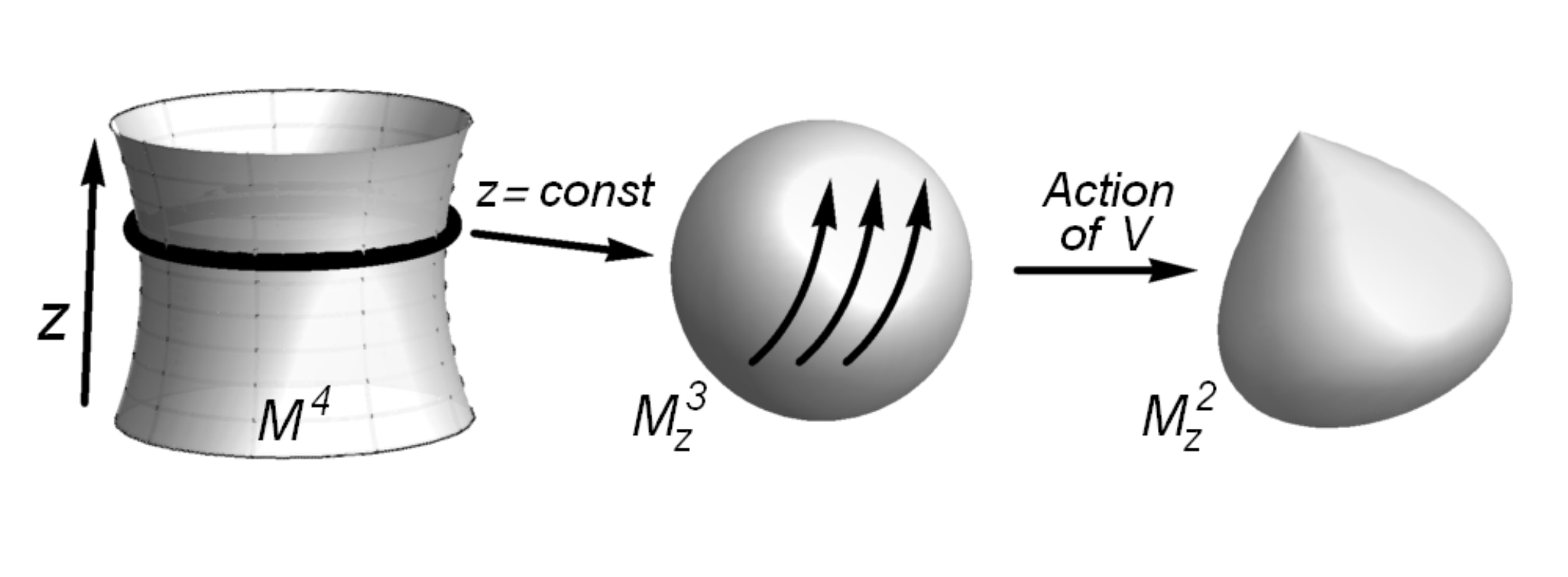}
		\caption{Schematic depiction of the K\"ahler reduction.
			Each level-set $M^3_z$ inherits a Killing field.
			When $z$ is nonsingular then $M^3_z$ is a manifold and its quotient $M^2_z$ is a Riemannian orbifold.
		} \label{FigReduction}
	\end{figure}
	
	Before stating the theorems we address the issue of finding the potential function $z$.
	A topological obstruction exists in $H^1_{dR}(M^4)$, but the function $\triangle_4z$ and field $\partial/\partial{}z$ exist regardless.
	We approach this more systematically in Section \ref{SecCoords}, but for now we state that $\triangle_4z=*(dV\wedge\omega)$ which exists globally, and we will see $\partial/\partial{}z=-|V|^{-2}(i_{V}\omega)^\sharp$ which exists everywhere except at zeros of $V$.
	The level-sets $M^3_z$ and quotients $M^2_z$ exist even if $z$ does not exist.
	This is because the distribution
	\begin{equation}
		\mathcal{D}\;=\;\{v\in{}TM^4\;\big|\;\omega(v,V)=0\}
	\end{equation}
	is always integrable, as a consequence of $d(i_V\omega)=0$.
	Supposing $V$ is Hamiltonian and a potential $z$ exists, then $\mathcal{D}=(\nabla{}z)^\perp$.
	But if $V$ is Hamiltonian or not, we may still assume every leaf of $\mathcal{D}$ is an embedded 3-manifold, except where a leaf intersects the singular locus $\{V=0\}$.
	One can, for example, pass to the universal cover of $M^4$, where $V$ is indeed Hamiltonian.
	Lemma \ref{LemmaHamiltonian} gives an explicit criterion for Hamiltonicity.
	Equations (\ref{EqnFirstEvoEqn}) and (\ref{EqnSecondEvoEqn}) hold everywhere except at zeros of $V$, even if $V$ is not Hamiltonian.
	
	We address another a possible complication in the formation of the orbifolds $M^2_z$.
	The field $V$ passes to each $M^3_z$, but if its orbits are not closed then the Riemannian quotient $M^2_z$ is non-Hausdorff; for this reason we only consider the case of Killing fields with closed orbits.
	This is not too serious a restriction however.
	If the orbits are non-closed, then the closure of its diffeomorphism action has rank at least 2, and we can find a second Killing field.
	The case of two symmetries has been dealt with elsewhere (see for example \cite{Don09}, \cite{Weber24a} and references therein), so when orbits are non-closed, one can refer to those works.
	If orbits of $V$ are closed but unbounded, we can take a discrete quotient along its diffeomorphism flow, and obtain closed orbits.
	We assume throughout this paper that the diffeomorphism flow of $V$ is parameterized by the unit circle with parameter $t\in[0,2\pi)$.
	
	\begin{theorem} \label{ThmCptNonCpt}
		Assume a compact set $K\subset{}M^4$ and number $\epsilon>0$ exist so $|V|>\epsilon$ on $M^4\setminus{}K$.
		Then every integral leaf $M^3_z$ is compact, or every integral leaf is non-compact.
	\end{theorem}
	If the asymptotic condition $|V|>\epsilon$ is replaced by only $|V|>0$, then Theorem \ref{ThmCptNonCpt} is false.
	See \S\ref{SubSecTypeIIIExamples} for an example.
	
	The geometric Euler number of $M^2_z$, also called its orbifold Euler number, is
	\begin{equation}
		\chi^{g}(M^2_z)
		\;=\;\frac{1}{2\pi}\int_{M^2_z}K\,dVol_2
	\end{equation}
	(see \cite{Orlik72}).
	If $\chi(M^2)$ is its topological Euler number, we have the well-known relation $\chi(M^2)=\chi^g(M^2)+\sum_i(1-|\Gamma_i|^{-1})$ where the sum is taken over all orbifold points $p_i$ and $|\Gamma_i|$ is the cardinality of the orbifold group at $p_i$.
	In particular $\chi(M^2)\ge\chi^g(M^2)$ with equality if and only if $M^2$ is a manifold.
	
	When $M^3_z$ is non-singular then $M^3_z\rightarrow{}M^2_z$ is an $\mathbb{S}^1$ fibration and $M^3_z$ is a Seifert fibered space.
	When compact, the fibration's \text{geometric Euler number} is
	\begin{equation}
		e^{g}(M^3_z,M^2_z)
		\;=\;\frac{1}{4\pi^2}\int_{M^3_z}|V|^{-4}V_\flat\wedge{}dV_\flat \label{EqnChernSimonsIntro}
	\end{equation}
	under the normalizing condition that orbits of $V$ are closed and parameterized by $t\in[0,2\pi)$.
	If $M^3_z\rightarrow{}M^2_z$ is not just a fibration but a fiber bundle 
	with Euler number $e$, then $e=e^g(M^3_z,M^2_z)$.
	A second interpretation of (\ref{EqnChernSimonsIntro}) is that its integrand is the natural Chern-Simons 3-form on $M^3_z$ associated to the fibration; see \cite{NR84}, and the proof of Proposition \ref{PropConstantGrowthInText} below.
	\begin{theorem} \label{ThmConstantGrowth}
		Assume the leaves $M^3_z$ are compact.
		Away from zeros of $V$ the quantities $\frac{d}{dz}\int\,dVol_2$ and $e^g(M^3_z,M^2_z)$ are constant and
		\begin{equation}
			\frac{d}{dz}\int_{M^2_z}\,dVol_2
			\;=\;2\pi{}e^{g}(M^3_z,M^2_z). \label{EqnAreaOfReduction}
		\end{equation}
		Assume in addition that $M^4$ is scalar-flat.
		Away from zeros of $V$, the quantities $\frac{d^2}{dz^2}\int|V|^2dVol_2$, $\frac{d}{dz}\int\triangle_4z\,dVol_2$, and $\chi^{g}(M^2_z)$ are constant and
		\begin{equation}
			\frac{d^2}{dz^2}\int_{M^2_z}|V|^2\,dVol_2
			\;=\;\frac{d}{dz}\int_{M^2_z}\triangle_4z\,dVol_2
			\;=\;4\pi\chi^{g}(M^2_z).
		\end{equation}
		When $M^4$ is scalar flat, we also have
		\begin{equation}
			\frac{d^2}{dz^2}\int_{M^2_z}(\triangle_4z)^2\,dVol_2
			\;=\;2\int_{M^2_z}|\Ric|^2\,dVol_2. \label{EqnIntoThmCGB}
		\end{equation}
	\end{theorem}
	Based on Theorems \ref{ThmCptNonCpt} and \ref{ThmConstantGrowth}, we create the following rough classification for scalar-flat K\"ahler manifolds $M^4$ with a symplectomorphic Killing field $V$:
	\begin{itemize}
		\item Type I: $M^4$ is compact.
		\item Type II: $M^4$ is non-compact, and $|V|>\epsilon>0$ outside a compact set.
		\begin{itemize}
			\item Type IIa: the distribution $\mathcal{D}$ has compact leaves
			\item Type IIb: the distribution $\mathcal{D}$ has non-compact leaves
		\end{itemize}
		\item Type III: a divergent sequence of points $p_i$ exist with $\lim_{i}|V(p_i)|=0$.
		\begin{itemize}
			\item Type IIIa: outside some compact set, $|V|>0$
			\item Type IIIb: a divergent sequence of points $p_i$ exist so that $|V(p_i)|=0$.
		\end{itemize}
	\end{itemize}
	In addition to Theorems \ref{ThmCptNonCpt} and \ref{ThmConstantGrowth}, another motivation for dividing Type II from Type III is that Type II metrics appear as ``deepest bubble'' blowup limits of collapsing extremal K\"ahler manifolds.
	A full classification of Type II metrics would have consequences for the study of K\"ahler 4-manifolds generally.
	About compact scalar-flat manifolds a great deal has been written; see for example \cite{LeB91b} \cite{KLP97} \cite{RS09} \cite{AP09} \cite{LeB20} \cite{CC18}.
	Hope for classifying Type III metrics seems out of reach.
	This paper focuses on the Type II case.
	\begin{lemma}[cf. Lemma \ref{LemmaUnboundedZText}] \label{LemmaOneOrTwoEnded}
		Assume $M^4$ is Type IIa.
		Then $M^4$ has at most two ends, the ends characterized by $z$ being unbounded from above, or else from below.
	\end{lemma}
	Recall that a \textit{manifold end} is an unbounded open set with compact boundary.
	In the Type IIa case $|dz|=|V|$ is asymptotically bounded away from zero, level-sets are compact, and $V$ is holomorphic so its zeros cannot accumulate; thus the set of critical values $\mathcal{C}=\{z_0,\dots,z_N\}$ of $z$ is finite.
	For $i\in\{1,\dots,N\}$ define $M^4_i=\{a\in{}M^4\;|\;z_{i-1}<z(a)<z_i\}$.
	Set $M^{4-}=\{a\in{}M^4\;|\;z(a)<z_0\}$ and $M^{4+}=\{a\in{}M^4\;|\;z_N<z(a)\}$.
	We have
	\begin{equation}
		M^4\;=\;\overline{M^{4-}}
		\,\cup\,\overline{M^4_1}
		\,\cup\,\dots\,
		\cup\,\overline{M^{4}_N}
		\cup\,\overline{M^{4+}}
	\end{equation}
	and by Lemma \ref{LemmaOneOrTwoEnded}, $M^{4+}$ and $M^{4-}$ are the two possible manifold ends.
	\begin{theorem}[cf. Corollary \ref{CorSignedEuler}] \label{ThmMfldEndTopology}
		Assume $M^4$ is Type IIa and $z$ has a non-empty set of critical points $\mathcal{C}$.
		Then on any $M^4_i$, on $M^{4+}$, or on $M^{4-}$, the level-sets $M^3_z$, reductions $M^2_z$, and fibrations $M^3_z\rightarrow{}M^2_z$ are isotopic under the diffeomorphism flow of $\frac{\partial}{\partial{}z}$.
		
		If $M^{4+}\ne\varnothing$, then $\chi^g(M^2_z)\ge0$ and $e^g(M^3_z,M^2_z)\ge0$ on $M^{4+}$.
		If $M^{4-}\ne\varnothing$, then $\chi^g(M^2_z)\ge0$ and $e^g(M^3_z,M^2_z)\le0$ on $M^{4-}$.
	\end{theorem}
	In the Type III case level-sets need not be isotopic; see \S\ref{SubSecTypeIIIExamples}.
	A classification of the topologies of Type IIa manifold ends follows from Lemma \ref{LemmaOneOrTwoEnded} and Theorem \ref{ThmMfldEndTopology}, combined with the classification of Seifert fibered spaces.
	Type IIa manifolds have one or two ends, and each end has the asymptotic topology of a ray crossed with a Seifert fibered space with $\chi^g(M^2_z)\ge0$ and either $e^g(M^3_z,M^2_z)\ge0$ or $e^g(M^3_z,M^2_z)\le0$.
	They fall into the following families.
	
	If $\chi^g(M^2_z)=0$ and $e^g(M^3_z,M^2_z)=0$, the Thurston geometry of $M^3_z$ is Euclidean, $M^3_z$ is a compact quotient of $\mathbb{R}^3$ by a crystallographic group.
	There are only three cases that are orientable with orientable base: the 3-torus, a fibration over $\mathbb{S}^2$ with four exceptional orbits each of multiplicity $2$, and a fibration over $\mathbb{S}^2$ with three exceptional orbits each of multiplicity $3$.
	
	If $\chi^g(M^2_z)=0$ and $e^g(M^3_z,M^2_z)\ne0$, the Thurston geometry of $M^3_z$ is Nil, so $M^3_z$ is a quotient of $\mathbb{R}^3$ by a discrete subgroup of the Heisenberg group (see \S4 of \cite{Scott83}).
	This is the family of nilmanifolds, which are classified by $e^g(M^3_z,M^2_z)$.
	
	If $\chi^g(M^2_z)>0$ and $e^g(M^3_z,M^2_z)=0$ then the Thurston geometry of $M^3_z$ is $\mathbb{S}^2\times\mathbb{R}$.
	By Theorem \ref{ThmTypeIIa} we must have $\int_{M^4}|\Ric|^2=\infty$.
	Only two oriented compact 3-manifolds of this type, which are $\mathbb{S}^2\times\mathbb{S}^1$ and $\mathbb{R}P^3\#\mathbb{R}P^3$ (which has $\mathbb{S}^2\times\mathbb{S}^1$ as a double cover).
	However $\mathbb{R}P^3\#\mathbb{R}P^3$ has base $\mathbb{R}P^2$ which is non-orientable, so the only possibility for $M^3_z$ is $\mathbb{S}^2\times\mathbb{S}^1$.
	
	Lastly if $\chi^g(M^2_z)>0$ and $e^g(M^3_z,M^2_z)\ne0$ then $M^3_z$ has spherical Thurston geometry, so $M^3_z$ is covered by a 3-sphere and $M^2_z$ is a 2-sphere with up to three orbifold points.
	There are infinitely many topologies in this case, which fall into finitely many families.
	
	More detailed information on any of these cases is available; see for example \cite{Raymond68}, \cite{Orlik72}, \cite{Hempel76}, or \cite{Scott83}.
	For ``Thurston geometries'' generally see \cite{Thurston82}.
	
	We prove two theorems in the case Ricci curvature has a bound.
	\begin{theorem}[cf. Prop. \ref{PropTypeIIaMainText}] \label{ThmTypeIIa}
		Assume $M^4$ is Type IIa and $\int_{M^4}|\Ric|^2<\infty$.
		If $e^g(M^3_z,M^2_z)=0$ on $M^{4+}$ (resp. $M^{4-}$) then $\chi^g(M^2_z)=0$ on $M^{4+}$ (resp. $M^{4-}$).
	\end{theorem}
	\begin{proposition}[cf. Prop. \ref{PropTypeIIbMainText}] \label{PropTypeIIb}
		Assume $M^4$ is Type IIa, non-flat, and Ricci flat.
		Then $M^4$ is one-ended, $\triangle_4z$ is a non-zero constant, and $2e^g(M^3_z,M^2_z)\cdot\triangle_4z=\chi^g(M^2_z)$ for non-critical $z$.
		If $\triangle_4z>0$ then $M^{4-}=\varnothing$.
		If $\triangle_4z<0$ then $M^{4+}=\varnothing$.
	\end{proposition}
	We have two results in the case that $V$ is globally bounded from below.
	\begin{proposition}[cf. Prop. \ref{PropMainTextCompactWithNoZeros}] \label{PropIntroCompactWithNoZeros}
		Assume $M^4$ is Type I and its Killing field is nowhere-zero.
		Then $M^4$ is flat.
	\end{proposition}
	\begin{theorem}[cf. Theorem \ref{ThmTextCptLevelsNoZeros}] \label{ThmIntroCptLevelsNoZeros}
		Assume $M^4$ is Type IIa and non-flat, $|V|$ has a global lower bound $|V|>\epsilon>0$, and $|\Ric|$ has an upper bound.
		
		Then $M^4$ is 2-ended, $z$ ranges from $-\infty$ to $\infty$, we have $\chi^g(M^2_z)\equiv2$, $e^g(M^3_z,M^2_z)\equiv0$ and $L^2(|\Ric|)=\infty$.
		All level-sets are diffeomorphic to $\mathbb{S}^2\times\mathbb{S}^1$, and $M^4$ is diffeomorphic to $\mathbb{R}\times\mathbb{S}^2\times\mathbb{S}^1$.
	\end{theorem}
	The main example of a Type IIa metric with a global lower bound on $|V|$ is the scalar-flat cross product $\mathbb{S}^2\times{}\mathbb{H}^2$ where the hyperbolic factor has a hyperbolic Killing field; see \S\ref{SubSecTypeIIExamples}.
	Theorem \ref{ThmIntroCptLevelsNoZeros} says any Type IIa manifold with $V$ bounded away from $0$ qualitatively resembles this.
	Theorem \ref{ThmIntroCptLevelsNoZeros} is false if only $|V|>0$.
	
	The following conjecture, if true, classifies the case where $|V|>\epsilon>0$ and level-sets are \textit{non}-compact.
	Compare with \cite{Weber19} which proved the Ricci-flat equivalent, and Theorem 1.6 of \cite{Weber24b} for the case with two symmetries instead of one.
	\begin{conjecture} \label{ConjIIb}
		Assume $M^4$ is Type IIb and $|V|$ has a global lower bound $|V|>\epsilon>0$.
		Then $M^4$ is flat.
	\end{conjecture}
	
	In closing, we remark that scalar-flat K\"ahler metrics with a continuous symmetry are locally determined by solutions of the LeBrun equations \cite{LeB91a}
	\begin{equation}
		u_{xx}+u_{yy}+(e^u)_{zz}=0
		\quad\quad
		w_{xx}+w_{yy}+(we^u)_{zz}=0. \label{EqnsLeBrun}
	\end{equation}
	The nonlinearity of these equations and their fundamentally local nature make global questions difficult to address, however.
	Another approach that has been explored is the momentum or Hamiltonian constructions in K\"ahler geometry \cite{HS02}, \cite{ACG06}, \cite{ACGT04}.
	New in our approach are Eqs. (\ref{EqnFirstEvoEqn}) and (\ref{EqnSecondEvoEqn}), the global nature of which provide new controls that are especially useful in the non-complete case.
	
	In Section \ref{SecCoords} we establish basic properties of the coordinate system, and prove Lemma \ref{LemmaOneOrTwoEnded} and Theorems \ref{ThmCptNonCpt} and \ref{ThmMfldEndTopology}.
	In Sections \ref{SecKillingTwoForm} and \ref{SecEvoK} we prove Eq. (\ref{EqnFirstEvoEqn}), and prove Theorem \ref{ThmConstantGrowth} except (\ref{EqnIntoThmCGB}).
	In Section \ref{SecCGB} we use the Chern-Gauss-Bonnet formalism to prove (\ref{EqnSecondEvoEqn}) and (\ref{EqnIntoThmCGB}), and prove Proposition \ref{PropIntroCompactWithNoZeros}.
	In Section \ref{SecTypeIandII} we prove Theorems \ref{ThmTypeIIa}, \ref{PropTypeIIb} and \ref{ThmIntroCptLevelsNoZeros}.
	Section \ref{SecExamples} provides examples.

	%
	%
	%
	%
	%
	%
	%
	%
	\section{Basic properties of $M^4$} \label{SecCoords}
	
	Here we lay out the basic facts about K\"ahler 4-manifolds with a Killing field, and prove Theorem \ref{ThmConstantGrowth}.
	Section \ref{SubSecMixedFrame} establishes facts about our preferred framing, the ``mixed frame'' $\{\nabla{}x,\nabla{}y,\nabla{}z,\frac{\partial}{\partial{}t}\}$.
	Section \ref{SubSecVolumeForms} studies volume forms and proves the ``Integration Lemma'' referred to in the Introduction.
	Section \ref{SubSecGlobalsOfM} establishes some global properties of $M^4$ and includes the proof of Theorem \ref{ThmConstantGrowth}.
	We first record some standard K\"ahler identities we will use frequently.
	For any function $f$ and 1-form $\alpha$ we have
	\begin{equation}
		\begin{array}{ll}
			dVol_4=\frac12\omega\wedge\omega,
			&d\alpha\wedge\omega
			=-\frac12div_4(J\alpha^\sharp)\,\omega\wedge\omega, \\
			(dJdf)^+
			=-\frac12\left(\triangle_4{}f\right)\omega, \quad
			&df\wedge{}Jdg\wedge\omega
			=-\frac12\left<\nabla{}f,\,\nabla{}g\right>\omega\wedge\omega.
		\end{array} \label{EqnBasicOmega}
	\end{equation}
	Because $V$ is symplectomorphic we have $\mathcal{L}_V\omega=0$, and because $d\omega=0$ and $\mathcal{L}_V=di_V+i_Vd$ we have $0=d(i_V\omega)=0$.
	Therefore we can solve $dz=-i_V\omega$ locally, although an obstruction to a global solution may exist in the deRham space $H^1_{dR}(M^4)$.
	We frequently use the following relations:
	\begin{equation}
		dz\;=\;-i_V\omega,
		\quad\quad
		V_\flat\;=\;i_{\nabla{}z}\omega,
		\quad\quad
		dz\;=\;JV_\flat,
		\quad\quad
		\nabla{z}\;=\;-JV. \label{EqnBasicZV}
	\end{equation}
	where $V_\flat=\left<V,\,\cdot\right>$ and we use the convention $J(\eta)=\eta\circ{}J$ when $\eta\in\bigwedge^1$.
	This causes a sign to appear sometimes, due to the fact that $J(V_\flat)=-(JV)_\flat$.

	%
	%
	%
	%
	%
	%
	\subsection{The mixed frame} \label{SubSecMixedFrame}
	We use the LeBrun coordinates $(x,y,z,t)$ of \cite{LeB91a}.
	To construct these coordinates let $M^3_0$ be a non-singular leaf of $\mathcal{D}$; on $M^3_0$ we create three local coordinates $(x,y,t)$.
	To create the function $t$, choose any transversal to $V$ within $M^3_0$, declare $t=0$ on this transversal, and push the value of $t$ along the action of $V$.
	We've assumed the flow of $V$ is parameterized by the unit circle, so we may take $t\in[0,2\pi)$.
	On the quotient $M^2_0$ of $M^3_0$, we may place local isothermal coordinates $x$ and $y$.
	Pulling back along the quotient, we obtain functions $x$, $y$ on $M^3_0$.
	Because the flow of $V$ is by isometries, we retain the relations $\left<\nabla{}x,\nabla{}y\right>=0$ and $|\nabla{}x|^2=|\nabla{}y|^2$.
	
	On any neighnorhood of $M^3_0$ that is deformation retractable onto $M^3_0$ we can solve $dz=-i_V\omega=JV_\flat$ for $z$, so we have a locally defined function $z$ with $\nabla{}z=-JV$.
	Finally we push the functions $x$, $y$, and $t$ along the diffeomorphism flow of $\nabla{}z$ to obtain functions $(x,y,t)$ on a region of $M^4$.
	We remark that because $t$ parameterizes the flow of $V$ and $\frac{\partial{}x}{\partial{}t}=\frac{\partial{}y}{\partial{}t}=\frac{\partial{}z}{\partial{}t}=0$, we have the relation $\frac{\partial}{\partial{}t}=V$.
	
	For computations we find the ``mixed frame'' $\{\nabla{}x,\nabla{}y,\nabla{}z=-JV,\frac{\partial}{\partial{}t}=V\}$ to be most convenient.
	Lemmas \ref{LemmaVZBasics} and \ref{LemmaCoordSystemFacts} establish its basic properties.
	\begin{lemma} \label{LemmaVZBasics}
		$\text{Hess}\,z$, $\nabla{}V_\flat$, and $\nabla{}V$ are $J$-invariant in the sense that
		\begin{equation}
			\nabla_{Jv}J\nabla{}z=-\nabla_{v}\nabla{}z,
			\;\;
			\nabla_{Jv}JV_\flat=-\nabla_{v}\nabla{}V_\flat,
			\;\;
			\text{Hess}\,z(J\cdot,J\cdot)
			=\text{Hess}\,z(\cdot,\cdot) \label{EqnJInvariance}
		\end{equation}
		where $v$ is any vector field.
		We have the following relation:
		\begin{equation}
			(\text{Hess}\,z)(J\cdot,\,\cdot)
			\;=\;\nabla{}V_\flat
			\;=\;\frac12dV_\flat. \label{EqnBasicFirst}
		\end{equation}
		The Lie derivatives of $g$, $\omega$, and $J$ along $\nabla{}z$ are
		\begin{equation}
			\mathcal{L}_{\nabla{}z}g
			=2\text{Hess}\,z,
			\quad
			\mathcal{L}_{\nabla{}z}\omega
			=2\nabla{}V_\flat=dV_\flat,
			\quad\text{and}\quad
			\mathcal{L}_{\nabla{}z}J\;=\;0 \label{EqnLie}
		\end{equation}
	\end{lemma}
	\begin{proof}
		The first two expression in (\ref{EqnJInvariance}) are immediate from the third.
		To prove the third, the fact that $J$ is covariant-constant and compatible with the metric gives
		\begin{equation}
			\begin{aligned}
				\left<\nabla_{v}\nabla{}z,w\right>
				&=
				\left<\nabla_{v}V,Jw\right>
				=-\left<\nabla_{Jw}V,v\right>
				=\left<\nabla_{Jw}\nabla{}z,Jv\right>
				=\left<\nabla_{Jv}\nabla{}z,Jw\right>.
			\end{aligned}
		\end{equation}
		The expression $(\text{Hess}\,z)(J\cdot,\,\cdot)=\nabla{}V_\flat$ is immediate from $J$-invariance of $\text{Hess}\,z$.
		To see $dV_\flat=2\nabla{}V_\flat$ we compute
		\begin{equation}
			\begin{aligned}
				dV_\flat(A,B)
				&=A\left<V,B\right>-A\left<V,B\right>-\left<V,[A,B]\right> \\
				&=\left<\nabla_AV,\,B\right>-\left<\nabla_AV,\,B\right>
				=2\left<\nabla_AV,\,B\right>
			\end{aligned}
		\end{equation}
		where we used $[A,B]=\nabla_AB-\nabla_BA$.
		For the Lie derivative of the metric, we use the torsion-free condition again to find
		\begin{equation}
			\begin{aligned}
				\left(\mathcal{L}_{\nabla{}z}g\right)(A,B)
				=\nabla{}z\left<A,B\right>-\left<[\nabla{}z,A],B\right>-\left<A,[\nabla{}z,B]\right>
				=2\left<\nabla_Az,\,B\right>
			\end{aligned}
		\end{equation}
		To take the Lie derivative of $\omega$, we use $\mathcal{L}_{\nabla{}z}=di_{\nabla{}z}+i_{\nabla{}z}d$, Eq. (\ref{EqnBasicZV}), and the fact that $d\omega=0$ to find
		\begin{equation}
			\mathcal{L}_{\nabla{}z}\omega
			=dV_\flat=2\nabla{}V_\flat.
		\end{equation}
		Finally for $\mathcal{L}_{\nabla{}z}J$ we use the Leibniz rule $\mathcal{L}_{\nabla{}z}\omega=\mathcal{L}_{\nabla{}z}g\circ{}J+g\circ\mathcal{L}_{\nabla{}z}J$ to obtain
		\begin{equation}
			2\nabla{}V_\flat
			=2(\text{Hess}\,z)(J\cdot,\,\cdot)
			+g\left(\left(\mathcal{L}_{\nabla{}z}J\right)\cdot, \,\cdot\right).
		\end{equation}
		Then the fact that $2\text{Hess}\,z(J\cdot,\,\cdot)=2\nabla{}V_\flat$ means $\mathcal{L}_{\nabla{}z}J=0$.
	\end{proof}
	\begin{lemma}[The coordinate system] \label{LemmaCoordSystemFacts}
		The fields $\nabla{}x$, $\nabla{}y$, $\nabla{}z$, and $\frac{\partial}{\partial{}t}$ are mutually orthogonal.
		The fields $\nabla{}z$ and $\nabla{}t$ are orthogonal.
		We have $\frac{\partial}{\partial{}t}=V$ and $\frac{\partial}{\partial{}z}=|\nabla{}z|^{-2}\nabla{}z=-|V|^{-2}JV=-|V|^{-2}(i_V\omega)^\sharp$.
		The functions $x$ and $y$ on $M^4$ are pluriharmonic conjugates (meaning $dx=Jdy$).
	\end{lemma}
	\begin{proof}
		By construction the functions $x$ and $y$ are isothermal coordinates on $M^2_0$ so $dx=Jdy$ on $M^2_0$.
		Because $J$ and $g$ are $V$-invariant, the complex structure and metric on $M^2_0$ are inherited from $M^4$.
		From the coordinate construction we have $\mathcal{L}_{\nabla{}z}x=\mathcal{L}_{\nabla{}z}y=0$ and from Lemma \ref{LemmaVZBasics} we have $\mathcal{L}_{\nabla{}z}J=0$.
		Thus $\mathcal{L}_{\nabla{}z}\left(dx-Jdy\right)=0$ and so we retain $dx=Jdy$.
		In particular $dJdx=0$ and $dJdy=0$ so $x$ and $y$ are both pluriharmonic.
		
		Because $x$, $y$, $z$ are constant along $V$ and $t$ is a unit-speed parameterization of its flow, we have $\frac{\partial}{\partial{}t}=V$.
		The fact that $\nabla{}z$ is orthogonal to $\nabla{}x$ is immediate from $0=\mathcal{L}_{\nabla{}z}x=\left<\nabla{}z,\nabla{}x\right>$.
		We similarly have $\mathcal{L}_{\nabla{}z}y=0=\mathcal{L}_{\nabla{}z}t$ so $\nabla{}z$ is also orthogonal to $\nabla{}y$ and $\nabla{}t$.
		The fact that $\frac{\partial}{\partial{}t}$ is orthogonal to $\nabla{}x$ is due to $\left<\frac{\partial}{\partial{}t},\nabla{}x\right>=\frac{\partial{}x}{\partial{}t}=0$.
		Similarly $\frac{\partial}{\partial{}t}$ is orthogonal to $\nabla{}y$ and $\nabla{}z$.
		
		In any coordinate system $\{x^1,\dots,x^n\}$ the matrices $g(\nabla{}x^i,\nabla{}x^j)$ and $g\left(\frac{\partial}{\partial{}x^i},\frac{\partial}{\partial{}x^i}\right)$ are inverses.
		Because $\nabla{}z$ is orthogonal to $\nabla{}t$, $\nabla{}x$, $\nabla{}y$, the matrix $g(\nabla{}x^i,\nabla{}x^j)$ splits into a $1\times1$ block and a $3\times{}3$ block.
		From the $1\times1$ block we get $\frac{\partial}{\partial{}z}=|\nabla{}z|^{-2}\nabla{}z$.
	\end{proof}
	
	%
	%
	%
	%
	%
	%
	\subsection{Volume forms} \label{SubSecVolumeForms}
	We define the 3-form $dVol_3$ to be the intrinsic volume form of $M^3_z\subset{}M^4$ and the 2-form $dVol_2$ to be the intrinsic volume form of $M^2_z$.
	By pull-back along the action of $V$, $dVol_2$ can be regarded as a 2-form on any of the $M^3_z$ or on $M^4$.
	The local expressions are
	\begin{equation}
		\begin{aligned}
			dVol_4
			&\;=\;|V|^{-2}|\nabla{}x|^{-2}\,
			dz\wedge{}V_\flat\wedge{}dx\wedge{}dy \\
			dVol_3&
			\;=\;i_{\frac{\nabla{}z}{|\nabla{}z|}}dVol_4
			\;=\;|V|^{-1}|\nabla{}x|^{-2}V_\flat\wedge{}dx\wedge{}dy \\
			dVol_2
			&\;=\;i_{\frac{V}{|V|}}i_{\frac{\nabla{}z}{|\nabla{}z|}}dVol_4
			\;=\;|\nabla{}x|^{-2}dx\wedge{}dy \\
			*dVol_2
			&\;=\;|\nabla{}z|^{-2}dz\wedge{}V_\flat \\
			\omega
			&\;=\;|\nabla{}z|^{-2}dz\wedge{}V_\flat
			+|\nabla{}x|^{-2}dx\wedge{}dy
			\;=\;*dVol_2\,+\,dVol_2.
		\end{aligned} \label{EqnManyVols}
	\end{equation}
	\begin{lemma} \label{LemmaExtDerivsOfTwoVols}
		We have the exterior derivatives
		\begin{equation}
			\begin{aligned}
				d\big(dVol_2\big)
				&=-\left(\frac{\partial}{\partial{}z}\log|\nabla{}x|^2\right)dz\wedge{}dVol_2, \\
				d\big(*dVol_2\big)&=\left(\frac{\partial}{\partial{}z}\log|\nabla{}x|^2\right)dz\wedge{}dVol_2.
			\end{aligned} \label{EqnVolTwoExtDeriv}
		\end{equation}
	\end{lemma}
	\begin{proof}
		From (\ref{EqnManyVols}) we have $d(dVol_2)=d(|\nabla{}x|^{-2})\wedge{}dx\wedge{}dy$ which is the first equation in (\ref{EqnVolTwoExtDeriv}).
		The second equation follows from $d\omega=0$ and $\omega=*dVol_2+dVol_2$.
	\end{proof}
	\begin{lemma}[Integration Lemma] \label{LemmaIntegration}
		In addition to (\ref{EqnManyVols}) we have
		\begin{equation}
			dVol_4=dz\wedge{}dt\wedge{}dVol_2
			\quad\text{and}\quad
			dVol_3=|V|dt\wedge{}dVol_2. \label{EqnsVolInTermsOft}
		\end{equation}
		Consequently, if $f=f(x,y,z)$ is any $V$-invariant function, then
		\begin{equation}
			\begin{aligned}
				\int_{M^3_z}f\,dVol_3
				=2\pi\int_{M^2_z}|V|\,f\,dVol_2.
			\end{aligned}
		\end{equation}
		On the domain $M^4_{z_0,z_1}=\{p\in{}M^4\,\big|\,z_0<z(p)<z_1\}$ we have
		\begin{equation}
			\int_{M^4_{z_0,z_1}}f\,dVol_4 \;=\;2\pi\int_{z_0}^{z_1}\left(\int_{M^2_z}f\,dVol_2\right)\,dz.
		\end{equation}
	\end{lemma}
	\begin{proof}
		By Lemma \ref{LemmaCoordSystemFacts} $\nabla{}z$ is perpendicular to $\nabla{}t$, $\nabla{}x$, and $\nabla{}y$.
		Thus $\frac{1}{|V|^2}V_\flat=dt-\frac{\left<\nabla{}x,\nabla{}t\right>}{|\nabla{}x|^2}dx-\frac{\left<\nabla{}y,\nabla{}t\right>}{|\nabla{}y|^2}dy$.
		The equations (\ref{EqnsVolInTermsOft}) are immediate from this and (\ref{EqnManyVols}).
		We have
		\begin{equation}
			\begin{aligned}
				\int_{M^3_z}f\,dVol_3
				&=
				\int_{M^3_z}f\,\frac{1}{|V|}V_\flat\wedge{}dVol_2
				=
				\int_{M^3_z}f\,|V|\,dt\wedge{}dVol_2.
			\end{aligned}
		\end{equation}
		Because the parameterization is $t\in[0,2\pi)$ and $|V|$, $f$ are $t$-invariant, this is
		\begin{equation}
			\begin{aligned}
				\int_{M^3_z}f\,dVol_3
				=
				\int_{M^2_z}\left(\int_0^{2\pi}f|V|\,dt\right)\,dVol_2
				=
				2\pi\int_{M^3_z}f|V|\,dVol_2
			\end{aligned}
		\end{equation}
		Finally, using $dVol_4=\frac{1}{|V|}dz\wedge{}dVol_3$, we have
		\begin{equation}
			\small
			\int_{z_0<z<z_1}f\,dVol_4
			=\int_{z_0}^{z_1}\left(\int_{M^3_z}\frac{1}{|V|}f\,dVol_3\right)dz \\
			=2\pi\int_{z_0}^{z_1}\left(\int_{M^2_z}f\,dVol_2\right)dz.
		\end{equation}
	\end{proof}
	As a final remark on the LeBrun coordinate system, from \cite{LeB91a} the metric is
	\begin{equation}
		g\;=\;we^u(dx^2\,+\,dy^2)\,+\,w\,dz^2\,+\,w^{-1}(dt+\alpha)^2 \label{EqnLeBMetric}
	\end{equation}
	where the 1-form $\alpha$ depends on the transversal and the LeBrun functions $u$, $w$ are
	\begin{equation}
		u\;=\;\log|\nabla{}z|^2\,-\,\log|\nabla{x}|^{2},
		\quad
		w\;=\;|\nabla{}z|^{-2}. \label{EqnLeBrunUW}
	\end{equation}
	The metric (\ref{EqnLeBMetric}) is scalar-flat K\"ahler if and only if $u$ and $w$ solve the system (\ref{EqnsLeBrun}).

	%
	%
	%
	%
	%
	%
	\subsection{Global properties} \label{SubSecGlobalsOfM}
	Recall that the distribution $\mathcal{D}=\left\{v\in{}TM^4\;\big|\;\omega(V,v)=0\right\}$ is integrable and its leaves are labeled by $M^3_z$.
	\begin{lemma} \label{LemmaPermutes}
		The diffeomorphism flow of $|V|^{-2}JV$ preserves the distribution $\mathcal{D}$.
		Consequently the diffeomorphism flow of $\frac{\partial}{\partial{}z}$ (not $\nabla{}z$) permutes the leaves $M^3_z$.
	\end{lemma}
	\begin{proof}
		Let $v$ be any vector field whose vectors lie in $\mathcal{D}$; this means $i_vi_V\omega=0$.
		In the LeBrun system $\frac{\partial}{\partial{z}}=-|V|^{-2}JV$.
		The identity $i_A\mathcal{L}_B-\mathcal{L}_Bi_A=i_{[A,B]}$ gives
		\begin{equation}
			\begin{aligned}
			i_v\mathcal{L}_{\frac{\partial}{\partial{}z}}i_V\omega
			&\;=\;i_{[v,\frac{\partial}{\partial{}z}]}i_V\omega
			\;=\;
			\left<JV,\,\nabla_v\frac{\partial}{\partial{}z}\right>
			-\left<JV,\,\nabla_{\frac{\partial}{\partial{}z}}v\right> \\
			&\;=\;
			-\left<\nabla{}z,\,\nabla_v\frac{\partial}{\partial{}z}\right>
			-\left<JV,\,\nabla_{\frac{\partial}{\partial{}z}}v\right>.
			\end{aligned}
		\end{equation}
		Because $\left<\nabla{}z,\partial/\partial{}z\right>=1$ and $\left<JV,v\right>=i_vi_V\omega=0$, this is
		\begin{equation}
			\begin{aligned}
				i_v\mathcal{L}_{\frac{\partial}{\partial{}z}}i_V\omega
				&\;=\;
				\left<\nabla_v\nabla{}z,\,\frac{\partial}{\partial{}z}\right>
				+\left<\nabla_{\frac{\partial}{\partial{}z}}JV,\,v\right> \\
				&\;=\;
				\left<\nabla_{\frac{\partial}{\partial{}z}}\nabla{}z,\,v\right>
				-\left<\nabla_{\frac{\partial}{\partial{}z}}\nabla{}z,\,v\right>
				\;=\;0
			\end{aligned}
		\end{equation}
		where we used the symmetry of $Hess(z)$.
		We conclude that the diffeomorphism flow of $\frac{\partial}{\partial{}z}$ leaves the distribution $\mathcal{D}=\{v\,\big|\,i_vi_V\omega=0\}$ invariant.
	\end{proof}
	\begin{lemma} \label{LemmaHamiltonian}
		If $M^4$ is non-compact and $\mathcal{D}$ has a compact leaf, $V$ is Hamiltonian.
	\end{lemma}
	\begin{proof}
		If $V$ has a zero then it is Hamiltonian (see for example \cite{MS13}), so we need only consider the case $V$ has no zeros and $-|V|^2JV$ is defined everywhere.
		
		Let $M^3_0$ be a compact integral surface of $\mathcal{D}$, and let $\psi_z:M^4\rightarrow{}M^4$ be the diffeomorphism flow of $-|V|^{-2}JV$, with flow parameter $z$.
		Because $-|V|^{-2}JV$ permutes the integral surfaces of $\mathcal{D}$, if there is a point $p\in{}M^3_0$ for which $\psi_z(p)\in{}M^3_0$ then in fact $\psi_z:M^3_0\rightarrow{}M^3_0$.
		Possibly a value $z_0\ne0$ exists where $\psi_{z_0}$ maps the integral surface $M^3_0$ to itself.
		In this case $M^4$ is a mapping torus for $\psi_{z_0}:M^3_0\rightarrow{}M^3_0$.
		But then $M^4$ is compact, contradicting the hypotheses.
		Therefore $\psi_z$ never maps $M^3_0$ to itself unless $z=0$.
		Assign to each point $p\in{}M^4$ the unique value of $z$ for which $p\in\psi_z(M^3_0)$ defines the Killing potential $z:M^4\rightarrow\mathbb{R}$.
	\end{proof}
	\begin{lemma} \label{LemmaConnectedLevelsets}
		Assume $M^4$ is connected.
		Then every leaf of $\mathcal{D}$ is connected.
	\end{lemma}
	\begin{proof}
		We show that if a leaf $M_{z_0}^3$ is \textit{locally} connected, then nearby leafs are also locally connected.
		For an argument by contradiction, assume not.
		Then there is a point $p_0\in{}M^3_{z_0}$ so that any sufficiently small ball $B_\delta$ containing $p_0$ and sufficiently small $\epsilon>0$, then $M^3_{z_0+\epsilon}\cap{}B_\delta$ is disconnected.
		
		By the isotopy criterion of Morse theory this is only possible if $p_0$ is a critical point of $z$.
		Choosing a sequence $\delta_i\searrow0$ sufficiently small and expanding the metric by $\delta_i{}^{-1}$ the ball $B_{\delta_i}$ becomes the ball $B_1$.
		For small enough $\delta_i$ local balls are diffeomorphic to Euclidean balls, so we can take a limit of the metric and complex structures in the $C^\infty$ sense on compact subsets.
		The metric converges to the flat metric on $\mathbb{R}^4$ and the complex structure converges to the complex structure on $\mathbb{C}^2$.
		The field $V$ converges to a symplectomorphic Killing field on $\mathbb{C}^2$ with a zero still at $p_0$.
		The structure of symplectomorphic Killing fields on flat $\mathbb{C}^2$ is completely straighforward: after choosing appropriate local complex coordinates $z^1=x^1+\sqrt{-1}y^1$, $z^2=x^2+\sqrt{-1}y^2$ where $p_0$ has coordinates $(0,0)$, we can write
		\begin{equation}
			\begin{aligned}
				V
				&\;=\;
				\alpha\,Im\left(z^1\frac{\partial}{\partial{}z^1}\right)
				+\beta\,Im\left(z^2\frac{\partial}{\partial{}z^2}\right) \\
				&\;=\;
				\alpha\left(-y^1\frac{\partial}{\partial{}x^1}+x^1\frac{\partial}{\partial{}y^1}\right)
				+\beta\left(-y^2\frac{\partial}{\partial{}x^2}+x^2\frac{\partial}{\partial{}y^2}\right)
			\end{aligned}
		\end{equation}
		for some $\alpha,\beta\in\mathbb{R}$, not both zero.
		Thus the level-sets of $z$ have the form $\alpha\big((x_1)^2+(y^1)^2\big)+\beta\big((x^2)^2+(y^2)^2\big)$.
		When $\alpha$ and $\beta$ have the same sign the level-sets are topologically 3-spheres.
		When $\alpha$ and $\beta$ have opposite signs, the level-sets are connected hyperboloids of revolution.
		When one of $\alpha$, $\beta$ is zero each level-set is a product of a circle with a plane.
		In all cases, the level-sets are connected.
		
		The convergence of both the metric and the potential in this blowup limit is in the $C^\infty$ sense.
		Thus, just prior to the limit for sufficiently small $\delta_i$, the critical point of the Killing potential must still be a saddle, so its level-sets near the critical point have the same structure and in particular are connected.
		This contradicts the assumption that level-sets are locally disconnected near $p$.
		
		Thus because the leafs of $\mathcal{D}$ cannot \textit{locally} separate into more than one component, all leafs must all have the same number of components.
		Because $M^4$ is connected, they must all have one component.
	\end{proof}
	
	\underline{\emph{Proof of Theorem \ref{ThmCptNonCpt}}}.
	We assume $M^4$ is complete and $|V|>\epsilon$ outside a compact set $K$.
	Assume one leaf of $\mathcal{D}$ is compact.
	We show all leafs of $\mathcal{D}$ are compact.
	
	Lemma \ref{LemmaHamiltonian} states $V$ is Hamiltonian so $z$ exists globally.
	Let $\Omega$ be a precompact domain that contains $K$ and the compact level-set $M^3_{z_0}$ which we assumed exists (it does not matter if it is singular).
	Let $\gamma$ be a trajectory of $z$ in the sense that
	\begin{equation}
		\dot\gamma
		\;=\;\frac{\partial}{\partial{}z}
		\;=\;\frac{1}{|\nabla{}z|^2}\nabla{z},
		\quad\text{and}\quad
		z(\gamma(s))\;=\;s. \label{EqnsGammaComditions}
	\end{equation}
	Thus $\frac{d}{ds}=\frac{\partial}{\partial{}z}$ along $\gamma$.
	Assume $\gamma:(a,b)\rightarrow{}M^4$ is maximally extended in the sense that the domain of $\gamma$ cannot be extended in such a way that (\ref{EqnsGammaComditions}) remains true.
	
	First we show $\gamma$ must intersect $\Omega$.
	Because $\gamma$ can always be uniquely extended except when $|\nabla{}z|=0$, we have three possibilities regarding $b$.
	The first is $b=+\infty$.
	The second is $b<\infty$ and some $\delta>0$ exists so $\gamma:[b-\delta,\,b)\rightarrow{}M^4$ is contained in a compact subset of $M^4$.
	The third is $b<\infty$ and $\gamma:[b-\delta,\,b)\rightarrow{}M^4$ is contained in no compact subset of $M^4$.
	In the second and third cases, we show $\gamma$ intersects $\Omega$.
	
	Assume the second possibility, that $b$ is finite and $\gamma|_{[b-\delta,b)}$ remains bounded.
	We show that $\gamma$ intersects $\Omega$.
	If $|(\nabla{}z)_{\gamma(s)}|>\epsilon$ for all $s\in[b-\delta,b)$ then $\gamma$ remains smooth with uniformly controlled tangent vector and therefore the path can be extended beyond $b$.
	Thus $|\nabla{}z|$ must approach $0$ near $b$, so $\gamma$ must intersect $K$ and therefore $\Omega$.
	
	Assume the third possibility, that $b$ is finite but $\gamma|_{[b-\epsilon,b)}$ is unbounded; we again show $\gamma$ intersects $\Omega$.
	Assume not; then $|\nabla{}z|\ge\epsilon$ along $\gamma$.
	Fix any $s_0\in(a,b)$.
	A sequence $\{s_i\}$ exists with $s_i\nearrow{}b$ and $\text{dist}(\gamma(s_0),\gamma(s_i))\rightarrow\infty$.
	Then
	\begin{equation}
		\begin{aligned}
			&|s_i-s_0|
			=|z(\gamma(s_i))-z(\gamma(s_0))|
			=\left|\int_{s_0}^{s_i}\frac{dz}{ds}\,ds\right|
			=\left|\int_{s_0}^{s_i}\left<\frac{d}{ds},\,\nabla{}z\right>\,ds\right| \\
			&\;\;\;=\int_{s_0}^{s_i}\left|\frac{d}{ds}\right||\nabla{}z|\,ds
			\ge\epsilon\int_{s_0}^{s_i}\left|\frac{d}{ds}\right|\,ds
			=\epsilon\;len(\gamma;s_0,s_i)
			\ge\epsilon\,\text{dist}(\gamma(s_0),\gamma(s_i))
		\end{aligned} \label{IneqLenOfGamma}
	\end{equation}
	where we used $\left<\frac{d}{ds},\,\nabla{}z\right>=\left|\frac{d}{ds}\right||\nabla{}z|$, which is due to $\frac{d}{ds}=|\nabla{}z|^{-2}\nabla{}z$ by Lemma \ref{LemmaCoordSystemFacts}.
	Because $\text{dist}(\gamma(s_0),\gamma(s_1))$ is unbounded, this contradicts the fact that $|s_i-s_0|\le|b-s_0|$.
	This forces our assumption that $|\nabla{}z|>\epsilon$ along $\gamma$ to be false.
	We conclude that $\gamma$ intersects $K$ and therefore $\Omega$.
	
	We have shown that either $\gamma$ intersects $\Omega$ or else $b=\infty$.
	The same considerations hold for the value $a$.
	Thus either $\gamma$ intersects $\Omega$ or else $\gamma:(-\infty,\infty)\rightarrow{}M^4$.
	In the latter case, because $(z\circ\gamma)(s)=s$ as $s$ ranges from $-\infty$ to $\infty$, $z\circ\gamma$ takes every value including the value $z_0$.
	Thus $\gamma$ intersects $M^3_{z_0}$, so intersects $\Omega$.
	We conclude that every trajectory intersects $\Omega$.
	
	Now we prove that every level-set $M^3_z$ is compact.
	For a contradiction, assume some $z'$ exists so $M^3_{z'}$ is not compact.
	Levelsets are closed (by continuity of $z$) so compactness is equivalent to boundedness.
	Thus $M^3_{z'}$ is unbounded.
	Because $M^3_{z'}$ is unbounded and $\overline{\Omega}$ is compact, a sequence of points $\{p_i\}\subset{}M^3_{z'}$ exists with $\text{dist}(p_i,\overline{\Omega})\rightarrow\infty$.
	Let $\gamma_i$ be the trajectory of $z$ passing through $p_i$; in particular $\gamma_i(z')=p_i\in{}M^3_{z'}$.
	Because $\gamma_i$ intersects $\Omega$, we can find some $s_i$ so $\gamma_i$ remains outside $\overline{\Omega}$ on $(s_i,z']$ and $\gamma_i(s_i)\in\overline{\Omega}$.
	Because $s_i=z(\gamma_i(s_i))$ and $z$ is uniformly bounded on $\overline{\Omega}$, the numbers $s_i$ are uniformly bounded.
	By (\ref{IneqLenOfGamma}) we have
	\begin{equation}
		|z'-s_i|
		\;\ge\;\epsilon\,\text{dist}(\gamma(s_i),\,p_i)
		\;\ge\;\epsilon\,\text{dist}(p_i,\,\overline{\Omega}). \label{EqnZbyDist}
	\end{equation}
	The left-hand side is bounded and the right-hand side is unbounded.
	This contradiction shows $M^3_{z'}$ is compact.
	\qed \\
	
	Before proving our final lemma, Lemma \ref{LemmaUnboundedZText}, we recall some basic definitions.
	An open subset $\Omega$ of a complete Riemannian manifold is called a \textit{manifold end} if it is unbounded but has a connected compact boundary.
	If $K\subset{}M^4$ is compact then the number of ends of $M^n$ relative to $K$ is the number of unbounded components of $M^n\setminus{}K$.
	We say that $M^n$ has $k$ many ends if the number of ends of $M^n$ relative to any compact subset is $\le{}k$, and there exists some compact subset $K$ so that $M^n$ has $k$ many ends relative to $K$.
	
	When all level-sets $M^3_z$ are compact, recall the decomposition
	\begin{equation}
		M^4\;=\;
		\overline{M^{4-}}\,\cup\,
		\overline{M^4_1}\,\cup\,\dots\,\cup\,
		\overline{M^4_N}\,\cup\,\overline{M^{4+}}
	\end{equation}
	where possibly $M^{4+}$ or $M^{4-}$ are empty.	
	\begin{lemma}[cf. Lemma \ref{LemmaOneOrTwoEnded}] \label{LemmaUnboundedZText}
		Assume $|V|>\epsilon>0$ outside some compact set, and at least one level-set is compact.
		Then each levelset $M^3_z$ and fibration $M^3_z\rightarrow{}M^2_z$ in $M_i^4$ (or in $M^{4+}$ or in $M^{4-}$) is isotopic along the $\frac{\partial}{\partial{}z}$ diffeomorphism flow.
		If $M^{4+}\ne\varnothing$ then $z(M^{4+})=[z_N,\infty)$, and if $M^{4-}\ne\varnothing$ then $z(M^{4-})=(-\infty,z_0]$ on $M^{4-}$.
	\end{lemma}
	\begin{proof}
		By Lemma \ref{LemmaHamiltonian} $z$ is Hamiltonian, by Lemma \ref{LemmaConnectedLevelsets} all level-sets are connected, and by Theorem \ref{ThmCptNonCpt} all level-sets are compact.
		Because levelsets of $\mathcal{D}$ are compact and connected, $M^4$ can have at most $2$ ends.
		
		Away from critical points of $z$ the diffeomorhism flow of $\frac{\partial}{\partial{}z}$ permutes the levelsets, but because $\frac{\partial}{\partial{}z}$ commutes with $V$, it also permutes the orbits of $V$.
		Because $M^3_z\rightarrow{}M^2_z$ is a quotient with leaves being the orbits of $V$, the flow of $\frac{\partial}{\partial{}z}$ permutes the fibration structure $M^3_z\rightarrow{}M^2_z$ as well.
		
		Assuming $M^{4+}\ne\varnothing$, then because $M^{4+}$ is a manifold end trajectories of $z$ in the sense of (\ref{EqnsGammaComditions}) are unbounded.
		Because $|V|>\epsilon>0$ outside of some compact set, by (\ref{IneqLenOfGamma}) along any such trajectory the value of $z$ is unbounded.
		Therefore $z$ has no upper bound on $M^{4+}$ and so $z(M^{4+})=[z_N,\infty)$.
		The same argument holds for $M^{4-}$ after replacing $z$ by $-z$.
	\end{proof}

	%
	%
	%
	%
	%
	%
	%
	%
	\section{The Killing 2-form} \label{SecKillingTwoForm}
	
	We gather required information about $dV_\flat$, the ``Killing 2-form.''
	In \S\ref{SubSecTwoFormApp} we find a global expression for the LeBrun equation $w_{xx}+w_{yy}+(we^u)_{zz}=0$ of (\ref{EqnsLeBrun}), and use this with information from \S\ref{SubSecTheTwoForm} to prove the first part of Theorem \ref{ThmConstantGrowth}.
	We frequently we use the decomposition $\bigwedge{}^2=\bigwedge^+\oplus\bigwedge^-$ where the projectors are $\pi^\pm:\bigwedge^2\rightarrow\bigwedge^\pm$, $\pi^\pm=\frac12(Id\pm*)$.
	This leads to identities such as $\eta=-*\eta+2\eta^+$ and $\eta^-=-*\eta+\eta^+$ whenever $\eta$ is a 2-form.
	
	%
	%
	%
	%
	%
	\subsection{The 2-form $dV_\flat$} \label{SubSecTheTwoForm}
	We gather some necessary information about $dV_\flat$.	
	\begin{lemma} \label{LemmaCompOfdV}
		We have
		\begin{equation}
			\begin{aligned}
				&dV_\flat
				\;=\;
				-\left(\frac{\partial}{\partial{}z}\log|\nabla{}z|^2\right)dz\wedge{}V_\flat
				-dz\wedge{}Jd\log|\nabla{}z|^2 \\
				&\hspace{0.425in}
				-V_\flat\wedge{}d\log|\nabla{}z|^2
				-|\nabla{}z|^2\left(\frac{\partial}{\partial{}z}\log|\nabla{}x|^{2}\right)dVol_2
			\end{aligned} \label{EqndV}
		\end{equation}
	\end{lemma}
	\begin{proof}
		We use $dV_\flat=2\left<\nabla_{\cdot}V,\,\cdot\right>$ from Lemma \ref{LemmaVZBasics}.
		Letting $v$ be any vector field,
		\begin{equation}
			\small
			\begin{aligned}
				&dV_\flat\left(\frac{\partial}{\partial{}z},\,v\right)
				=
				2\left<\nabla_{\frac{\partial}{\partial{}z}}V,\,v\right>
				=
				2|\nabla{}z|^{-2}\left<\nabla_{\nabla{}z}J\nabla{}z,\,v\right> \\
				&\quad\quad=
				-2|\nabla{}z|^{-2}\left<\nabla_{\nabla{}z}\nabla{}z,\,Jv\right>
				=
				-2|\nabla{}z|^{-2}\left<\nabla_{Jv}\nabla{}z,\,\nabla{}z\right>
				=
				-\left<Jv,\nabla\log|\nabla{z}|^2\right>.
			\end{aligned}
		\end{equation}
		Therefore $i_{\frac{\partial}{\partial{z}}}dV_\flat=-Jd\log|\nabla{}z|^2$.
		
		To compute $i_VdV_\flat$, we use the fact that $\mathcal{L}_VV_\flat=0$ which is $i_VdV_\flat+d(i_VV_\flat)=0$.
		Thus $i_{V}dV_\flat=-di_VV_\flat=-d|V|^2=-|V|^2d\log|V|^2$.
		From these expressions for $i_{\frac{\partial}{\partial{}z}}dV_\flat$ and $i_VdV_\flat$, the first three terms of (\ref{EqndV}) follow.
		
		Lastly we must check the action of $dV_\flat$ on pairs of vectors that are perpendicular to both $\partial/\partial{}z$ and $V_\flat$, meaning we check the action of $dV_\flat$ on $\nabla{}x$ and $\nabla{}y$.
		Using $\left<V,\,\nabla{}x\right>=\left<V,\,\nabla{}y\right>=0$, $\nabla{}y=J\nabla{}x$, and $V=J\nabla{}z$ we have
		\begin{equation}
			\begin{aligned}
				dV_\flat(\nabla{}x,\,\nabla{}y)
				&\;=\;2\left<\nabla_{\nabla{}x}V,\,\nabla{}y\right>
				\;=\;-2\left<V,\,\nabla_{\nabla{}x}J\nabla{}x\right>
				\;=\;-2\left<\nabla{}z,\,\nabla_{\nabla{}x}\nabla{}x\right> \\
				&\;=\;-2\left<\nabla{}x,\,\nabla_{\nabla{}z}\nabla{}x\right>
				\;=\;-\nabla{}z|\nabla{}x|^2 \\
			\end{aligned}
		\end{equation}
		Because $(dx\wedge{}dy)(\nabla{}x,\,\nabla{}y)=|\nabla{}x|^4$ we therefore have
		\begin{equation}
			\begin{aligned}
				dV_\flat\Big|_{span\{dx\wedge{}dy\}}
				&=-|\nabla{}z|^2|\nabla{}x|^{-4}\left(\frac{\partial}{\partial{}z}|\nabla{}x|^2\right)\,dx\wedge{}dy \\
				&=-|\nabla{}z|^2\left(\frac{\partial}{\partial{}z}\log|\nabla{}x|^2\right)dVol_2.
			\end{aligned}
		\end{equation}
	\end{proof}
	\begin{lemma} \label{LemmaVdV}
		We have the expressions
		\begin{equation}
			\begin{aligned}
				V_\flat\wedge{}dV_\flat
				&\;=\;
				dz\wedge{}V_\flat\wedge{}Jd\log|\nabla{}z|^2
				\,-\,|\nabla{}z|^2\left(\frac{\partial}{\partial{}z}\log|\nabla{}x|^2\right)\,V_\flat\wedge{}dVol_2 \\
				&\;=\;
				*\bigg(-d|V|^2 \,+\, (\triangle_4z)\,dz\bigg). \label{EqnVdVFirst}
			\end{aligned}
		\end{equation}
	\end{lemma}
	\begin{proof}
		The first equality in (\ref{EqnVdVFirst}) is immediate from (\ref{EqndV}).
		To compute the second equality, we use the identity $dV_\flat=-*dV_\flat+2(dV_\flat)^+$.
		Because $(dV_\flat)^+=-(dJdz)^+=\frac12(\triangle_4z)\omega$, from (\ref{EqnBasicOmega}), and because $**:\bigwedge^3\rightarrow\bigwedge^3$ is multiplication by $(-1)$,
		\begin{equation}
			\begin{aligned}
				&V_\flat\wedge{}dV_\flat
				=
				-V_\flat\wedge*dV_\flat
				+2V_\flat\wedge(dV_\flat)^+ \\
				&\hspace{0.2in}=
				**(V_\flat\wedge*dV_\flat)
				-(\triangle_4z)**(V_\flat\wedge\omega)
				=
				*\left(i_VdV_\flat
				-(\triangle_4z)\,i_V\omega\right).
			\end{aligned} \label{EqnNewCompVwDV}
		\end{equation}
		Then $\mathcal{L}_V=di_V+i_Vd$ and $\mathcal{L}_VV_\flat=0$ gives $i_VdV_\flat=-d|V|^2$, and (\ref{EqnBasicZV}) gives $i_V\omega=-dz$.
		Entering these into (\ref{EqnNewCompVwDV}) provides the conclusion.
	\end{proof}
	\begin{lemma} \label{LemmaTriangleZ}
		We have the following expressions for $\triangle_4z$:
		\begin{equation}
			\triangle_4z
			\;=\;|\nabla{}z|^2\frac{\partial}{\partial{}z}\left(\log|\nabla{}z|^2-\log|\nabla{}x|^2\right)
			\;=\;\left<\nabla{}z,\,\nabla{}u\right>
			\;=\;wu_z.
		\end{equation}
	\end{lemma}
	\begin{proof}
		We have $V_\flat=-Jdz$ so that $dV_\flat\wedge\omega=(\triangle_4z)dVol_4$.
		Using $dV_\flat$ from Lemma \ref{LemmaCompOfdV} and $\omega$ from (\ref{EqnManyVols}), we have
		\begin{equation}
			\begin{aligned}
				(\triangle_4z)\,dVol_4
				=dV_\flat\wedge\omega
				&=|\nabla{}z|^2\left(\frac{\partial}{\partial{}z}\log|\nabla{}z|^2-\frac{\partial}{\partial{}z}\log|\nabla{}x|^2\right)dVol_4 \\
				&=\left<\nabla{}z,\,\nabla\log|\nabla{}z|^2 -\nabla\log|\nabla{}x|^2\right>\,dVol_4.
			\end{aligned}
		\end{equation}
		Using (\ref{EqnLeBrunUW}), this is $\triangle_4z=\left<\nabla{}z,\nabla{}u\right>=|\nabla{}z|^2\frac{\partial{}u}{\partial{}z}=wu_z$.
	\end{proof}
	\begin{lemma} \label{LemmaLieOfV}
		We have the two Lie derivatives
		\begin{equation}
			\begin{aligned}
				&\mathcal{L}_{\frac{\partial}{\partial{}z}}V_\flat
				\;=\;-Jd\log|\nabla{}z|^2
				\quad\text{and}\quad
				\mathcal{L}_{\frac{\partial}{\partial{}z}}dV_\flat
				\;=\;-dJd\log|\nabla{}z|^2.
			\end{aligned}
		\end{equation}
	\end{lemma}
	\begin{proof}
		Using $\mathcal{L}_{\frac{\partial}{\partial{}z}}=di_{\frac{\partial}{\partial{}z}}+i_{\frac{\partial}{\partial{}z}}d$ and $i_{\frac{\partial}{\partial{}z}}V_\flat=0$, we have
		\begin{equation}
			\mathcal{L}_{\frac{\partial}{\partial{}z}}V_\flat
			\;=\;i_{\frac{\partial}{\partial{}z}}dV_\flat
			\;=\;-|V|^{-2}Jd|\nabla{}z|^2
			\;=\;-Jd\log|\nabla{}z|^2.
		\end{equation}
		The second equation follows from the commutator identity $[\mathcal{L}_{\frac{\partial}{\partial{}z}},d]=0$.
	\end{proof}
	\begin{lemma} \label{LemmaLieDerivsOfVols}
		We have Lie derivatives
		\begin{equation}
			\begin{aligned}
				\mathcal{L}_{\frac{\partial}{\partial{}z}}dVol_4
				&\;=\;\left(-\frac{\partial}{\partial{}z}\log|\nabla{}x|^2\right)\,dVol_4 \\
				\mathcal{L}_{\frac{\partial}{\partial{}z}}dVol_3
				&\;=\;\left(\frac12\frac{\partial}{\partial{}z}\log|\nabla{}z|^2\,-\,\frac{\partial}{\partial{}z}\log|\nabla{}x|^2\right)\,dVol_3
			\end{aligned} \label{EqnLieDeriveOfVolumes}
		\end{equation}
		Especially important are Lie derivatives involving $dVol_2$.
		These are
		\begin{equation}
			\begin{aligned}
				&\mathcal{L}_{\frac{\partial}{\partial{}z}}dVol_2
				\;=\;\left(-\frac{\partial}{\partial{}z}\log|\nabla{}x|^2\right)\,dVol_2 \\
				&\mathcal{L}_{\frac{\partial}{\partial{}z}}\big(|V|^2\,dVol_2\big)
				\;=\;
				(\triangle_4z)\,dVol_2 \\
				&\mathcal{L}_{\frac{\partial}{\partial{}z}}\left(V_\flat\wedge{}dVol_2\right)
				\;=\;\left(\frac{1}{|V|}\triangle_4z\right)\,dVol_3 \\
				&\mathcal{L}_{\frac{\partial}{\partial{}z}}*dVol_2
				\;=\;-\left(\frac{\partial}{\partial{}z}\log|V|^2\right)*dVol_2+dz\wedge{}Jd|V|^{-2}.
			\end{aligned} \label{EqnLieDeriveOfVSdVol}
		\end{equation}
	\end{lemma}
	\begin{proof}
		Because $\mathcal{L}_{\frac{\partial}{\partial{}z}}=di_{\frac{\partial}{\partial{}z}}$ on any 4-form, and because $i_{\nabla{}z}\omega=V_\flat$, we have
		\begin{equation}
			\begin{aligned}
				\mathcal{L}_{\frac{\partial}{\partial{z}}}dVol_4
				&=di_{\frac{\partial}{\partial{z}}}dVol_4
				=d\left(|\nabla{}z|^{-2}(i_{\nabla{}z}\omega)\wedge\omega\right)
				=
				d\left(|\nabla{}z|^{-2}V_\flat\wedge\omega\right) \\
				&=
				d\left(|\nabla{}z|^{-2}V_\flat\right)\wedge\omega
				=d(|\nabla{}z|^{-2})\wedge{}V_\flat\wedge\omega
				+|\nabla{}z|^{-2}dV_\flat\wedge\omega \\
				&=-d(|\nabla{}z|^{-2})\wedge{}Jdz\wedge\omega
				-|\nabla{}z|^{-2}dJdz\wedge\omega.
			\end{aligned}
		\end{equation}
		Then using (\ref{EqnBasicOmega}) and Lemma \ref{LemmaTriangleZ} this is
		\begin{equation}
			\begin{aligned}
				\mathcal{L}_{\frac{\partial}{\partial{z}}}dVol_4
				&\;=\;
				-\left(\frac{\partial}{\partial{}z}\log|\nabla{}z|^2\right)\,dVol_4
				+|\nabla{}z|^{-2}(\triangle_4z)dVol_4 \\
				&\;=\;-\left(\frac{\partial}{\partial{}z}\log|\nabla{}x|^2\right)dVol_4
			\end{aligned}
		\end{equation}
		The second equation in (\ref{EqnLieDeriveOfVolumes}) follows from the identity $[\mathcal{L}_A,i_B]=i_{[A,B]}$ along with $dVol_3=i_{|\nabla{}z|\frac{\partial}{\partial{}z}}dVol_4=i_{\frac{\nabla{}z}{|\nabla{}z|}}dVol_4$.
		
		In (\ref{EqnLieDeriveOfVSdVol}), the first equation follows from $[\mathcal{L}_A,i_B]=i_{[A,B]}$ and $dVol_2=i_{|V|^{-1}V}dVol_3$.
		The second equation follows from the Leibniz rule and Lemma \ref{LemmaTriangleZ}.
		The third equation uses Lemma \ref{LemmaLieOfV} to obtain
		\begin{equation}
			\begin{aligned}
				\mathcal{L}_{\frac{\partial}{\partial{}z}}(V_\flat\wedge{}dVol_2)
				&\;=\;
				\left(\mathcal{L}_{\frac{\partial}{\partial{}z}}V_\flat\right)\wedge{}dVol_2
				+V_\flat\wedge\left(\mathcal{L}_{\frac{\partial}{\partial{}z}}dVol_2\right) \\
				&\;=\;
				-Jd\log|V|^{2}\wedge{}dVol_2
				-\left(\frac{\partial}{\partial{}z}\log|\nabla{}x|^2\right)V_\flat\wedge{}dVol_2 \\
				&\;=\;
				|\nabla{}z|^{-2}\triangle_4z\,V_\flat\wedge{}dVol_2
			\end{aligned}
		\end{equation}
		Then $dVol_3=|\nabla{}z|^{-1}V_\flat\wedge{}dVol_2$ gives the desired equation.
		For the final equation,
		\begin{equation}
			\begin{aligned}
				\mathcal{L}_{\frac{\partial}{\partial{}z}}\big(*dVol_2\big)
				&=
				\left(di_{\frac{\partial}{\partial{}z}}
				+i_{\frac{\partial}{\partial{}z}}d\right)*dVol_2 \\
				&=d\left(|V|^{-2}V_\flat\right)
				+\left(\frac{\partial}{\partial{}z}\log|\nabla{}x|^2\right)dVol_2 \\
				&=
				d|V|^{-2}\wedge{}V_\flat
				+|V|^{-2}dV_\flat
				+\left(\frac{\partial}{\partial{}z}\log|\nabla{}x|^2\right)dVol_2.
			\end{aligned}
		\end{equation}
		Using Lemma \ref{LemmaCompOfdV}, this is
		\begin{equation}
			\begin{aligned}
				\mathcal{L}_{\frac{\partial}{\partial{}z}}\big(*dVol_2\big)
				&=
				-\left(\frac{\partial}{\partial{}z}\log|V|^2\right)\,*dVol_2
				+dz\wedge{}Jd|V|^{-2}.
			\end{aligned}
		\end{equation}
	\end{proof}

	%
	%
	%
	%
	\subsection{Volume of $M^2_z$} \label{SubSecTwoFormApp}
	
	The LeBrun equations (\ref{EqnsLeBrun}) are local, as they rely on a choice of isothermal coordinates $x$, $y$ on $M^2_z$.
	However our first lemma shows the second of the LeBrun equations has a more global expression.	
	\begin{lemma} \label{LemmaPartialElliptic}
		The LeBrun equation $w_{xx}+w_{yy}+(we^u)_{zz}=0$ is the same as
		\begin{equation}
			\mathcal{L}_{\frac{\partial}{\partial{}z}}\mathcal{L}_{\frac{\partial}{\partial{}z}}dVol_2\,+\,\left(\triangle_2|V|^{-2}\right)dVol_2
			\quad\text{when $V\ne0$.}
		\end{equation}
		Therefore when $M^2_z$ is compact, $\frac{d^2}{dz^2}\int_{M^2_z}dVol_2=0$ at regular values of $z$.
	\end{lemma}
	\begin{proof}
		Because $x,y$ are isothermal coordinates on $M^2_z$, $\triangle_2=|\nabla{}x|^{-2}\left(\frac{\partial^2}{\partial{}x^2}+\frac{\partial^2}{\partial{}y^2}\right)$ and $dVol_2=|\nabla{}x|^{-2}dx\wedge{}dy$.
		From (\ref{EqnLeBrunUW}), $w=|\nabla{}z|^{-2}$ and $we^u=|\nabla{}x|^{-2}$.
		Thus
		\begin{equation}
			\begin{aligned}
				0&\;=\;w_{xx}+w_{yy}+(we^u)_{zz} \\
				&\;=\;
				\left(\frac{\partial^2}{\partial{}x^2}
				+\frac{\partial^2}{\partial{}y^2}\right)|\nabla{}z|^{-2}
				+\frac{\partial^2}{\partial{}z^2}|\nabla{}x|^{-2} \\
				&\;=\;
				|\nabla{}x|^{-2}\triangle_2|\nabla{}z|^{-2}
				\,+\,\frac{\partial^2}{\partial{}z^2}|\nabla{}x|^{-2}
			\end{aligned}
		\end{equation}
		Because $\mathcal{L}_{\frac{\partial}{\partial{}z}}dx=d\frac{\partial{}x}{\partial{}z}=0$ and $\mathcal{L}_{\frac{\partial}{\partial{}z}}dy=d\frac{\partial{}y}{\partial{}z}=0$ we have $\mathcal{L}_{\frac{\partial}{\partial{}z}}(dx\wedge{}dy)=0$, so
		\begin{equation}
			\begin{aligned}
				\mathcal{L}_{\frac{\partial}{\partial{}z}}\mathcal{L}_{\frac{\partial}{\partial{}z}}dVol_2
				&=\left(\frac{\partial^2}{\partial{}z^2}|\nabla{}x|^{-2}\right)dx\wedge{}dy
				=-\left(\triangle_2|V|^{-2}\right)\;dVol_2.
			\end{aligned} \label{EqnDoubleLieVol}
		\end{equation}
		When $M^2_z$ is compact, integrating both sides of (\ref{EqnDoubleLieVol}) gives
		\begin{equation}
			\begin{aligned}
				&\int_{M^2_z}\mathcal{L}_{\frac{\partial}{\partial{}z}}\mathcal{L}_{\frac{\partial}{\partial{}z}}dVol_2
				\;=\;-\int_{M^2_z}\triangle_2|V|^{-2}\,dVol_2\;=\;0.
			\end{aligned}
		\end{equation}
		(Because $M^2_z$ is smooth in the orbifold sense and because $|V|^{-2}$ is smooth, there is no issue in using integration-by-parts despite orbifold points; see for example \cite{CW11} \cite{TV05} \cite{Uhl82} \cite{DK81}.)
		By Lemma \ref{LemmaPermutes} the diffeomorphism flow of $\frac{\partial}{\partial{}z}$ fixes $\mathcal{D}$ and therefore permutes the leaves $M^3_z$ and their reductions $M^2_z$.
		We conclude that
		\begin{equation}
			\begin{aligned}
				&\frac{d^2}{dz^2}\int_{M^2_z}dVol_2
				\;=\;
				\int_{M^2_z}\mathcal{L}_{\frac{\partial}{\partial{}z}}\mathcal{L}_{\frac{\partial}{\partial{}z}}dVol_2
				\;=\;0.
			\end{aligned}
		\end{equation}		
	\end{proof}
	\begin{proposition}[cf. Theorem \ref{ThmConstantGrowth}] \label{PropConstantGrowthInText}
		When $z$ is non-singular and $M^2_z$ is compact,
		\begin{equation}
			\frac{d}{dz}\int_{M^2_z}\,dVol_2
			\;=\;
			\frac{1}{2\pi}\int_{M^3_z}\frac{1}{|V|^4}V_\flat\wedge{}dV_\flat
			\;=\;
			2\pi{}e^g(M^3_z,M^2_z) \label{EqnVolGrowthEquality}
		\end{equation}
		where $e^g(M^3_z,M^2_z)$ is the Euler number of the fibration $M^3_z\rightarrow{}M^2_z$.
	\end{proposition}
	\begin{proof}
		The Lie group $\mathbb{S}^1$ acts on $M^3_z$, generated by the field $\frac{\partial}{\partial{}t}=V$.
		Therefore
		\begin{equation}
			\theta
			\;=\;\frac{1}{|V|^2}V_\flat\;\otimes\;\frac{\partial}{\partial{}t}
			\;\in\;\Gamma\left(\bigwedge{}^1\otimes\mathfrak{g}\right) \label{EqnBundleConnection}
		\end{equation}
		is a connection 1-form on the fibration $M^3_z\rightarrow{}M^2_z$.
		Because the Lie group is abelian its curvature 2-form is $F=d\theta$.
		Thus
		\begin{equation}
			F\;=\;
			d\left(\frac{1}{|V|^2}V_\flat\right)\Big|_{M^3_z}\;\otimes\;\frac{\partial}{\partial{}t}
			\;\in\;
			\Gamma\left(\bigwedge{}^2\otimes\mathfrak{g}\right).
		\end{equation}
		Certainly then $Tr(F)=d(|V|^{-2}V_\flat)$.
		Because the Lie algebra is Abelian, the Chern-Simons form is simply
		\begin{equation}
			Tr(\theta\wedge{}d\theta)
			\;=\;
			\frac{1}{|V|^2}V_\flat\wedge{}d\left(\frac{1}{|V|^2}V_\flat\right)
			\;=\;
			\frac{1}{|V|^2}V_\flat\wedge{}Tr(F).
		\end{equation}
		By definition the geometric Euler number is $e^g(M^3_z,M^2_z)=\frac{1}{2\pi}\int_{M^2_z}Tr(F)$.
		Thus
		\begin{equation}
			\chi(M^3_z,M^2_z)
			=\frac{1}{2\pi}\int_{M^2_z}Tr(F)
			=\frac{1}{4\pi^2}\int_{M^3_z}\frac{V}{|V|^2}\wedge{}Tr(F)
			=\frac{1}{4\pi^2}\int_{M^3_z}Tr(\theta\wedge{}d\theta)
		\end{equation}
		(the second equality is from the Integration Lemma, Lemma \ref{LemmaIntegration}).
		When $M^3_z$ is compact it is known that this is a rational number, and when $M^3_z\rightarrow{}M^2_z$ is a compact fiber bundle it is an integer (the usual Euler number of the bundle)---see, for example, Lemma 1 of \cite{NR84}.
		
		From Lemma \ref{LemmaVdV} and Lemma \ref{LemmaLieDerivsOfVols} we compute
		\begin{equation}
			Tr(\theta\wedge{}d\theta)
			\;=\;
			\frac{1}{|V|^{-2}}\left(\frac{\partial}{\partial{}z}\log|\nabla{}x|^2\right)V_\flat\wedge{}dVol_2
			\;=\;
			\frac{1}{|V|^2}V_\flat\wedge{}\mathcal{L}_{\frac{\partial}{\partial{}z}}dVol_2. \label{EqnLieAndChernSimons}
		\end{equation}
		so using (\ref{EqnLieAndChernSimons}) and Lemma \ref{LemmaIntegration} we obtain
		\begin{equation}
			\begin{aligned}
				\chi(M^3_z,M^2_z)
				&=\frac{1}{4\pi^2}\int_{M^3_z}\frac{1}{|V|^2}V_\flat\wedge\mathcal{L}_{\frac{\partial}{\partial{}z}}dVol_2
				=\frac{1}{2\pi}\int_{M^2_z}\mathcal{L}_{\frac{\partial}{\partial{}z}}dVol_2 \\
				&=\frac{1}{2\pi}\frac{d}{dz}\int_{M^2_z}dVol_2.
			\end{aligned}
		\end{equation}
		By Lemma \ref{LemmaPartialElliptic}, $\frac{d}{dz}\int_{M^2_z}dVol_2$ is constant when the level-sets $M^3_z$ are compact and $z$ is non-singular.
	\end{proof}

	%
	%
	%
	%
	%
	%
	%
	%
	\section{The Euler number of the base} \label{SecEvoK}
	
	We prove the evolution equation $\frac{d^2}{dz^2}\int_{M^2_z}|V|^2dVol_2=4\pi\chi^g(M^2_z)$
	when $M^4$ is scalar-flat and the integral leaves $M^3_z$ are compact.
	In \S\ref{SubsecEuclExample} we perform an example.
	
	%
	%
	%
	%
	%
	%
	\subsection{Interpretation of the Toda-lattice equation} \label{SubsecEulerNumber}
	The equation $u_{xx}+u_{yy}+(e^u)_{zz}=0$---the ``Toda-lattice equation'' or ``$SU(\infty)$ Toda-lattice equation''---has long been studied in physics, K\"ahler geometry, and conformal geometry.
	Because it relies on the formation of $x$, $y$ coordinates, this equation is fundamentally local.
	We prove the Toda-lattice equation has the more global expression
	\begin{equation}
		\begin{aligned}
			2K\,dVol_4
			&\;=\;
			d\left[
			-(*dVol_2)\wedge{}Jd\log|\nabla{}z|^2
			+\frac{1}{|V|}\triangle_4z\,dVol_3
			\right]
		\end{aligned} \label{EqnsPointwiseEvoForK}
	\end{equation}
	where $K$ is the Gaussian curvature of $M^2_z$, considered as a function on $M^4$ by pulling back along the reduction.
	\begin{proposition}
		In the scalar-flat case the equation $u_{xx}+u_{yy}+(e^u)_{zz}=0$ is
		\begin{equation}
			\begin{aligned}
				2K\,dVol_2
				&\;=\;
				\Big(\triangle_2\log|V|^2\Big)\;dVol_2
				+\mathcal{L}_{\frac{\partial}{\partial{}z}}\mathcal{L}_{\frac{\partial}{\partial{}z}}\Big(|V|^2\,dVol_2\Big) \\
				&\;=\;
				\Big(\triangle_2\log|V|^2\Big)\;dVol_2
				+\mathcal{L}_{\frac{\partial}{\partial{}z}}\Big(\triangle_4z\;dVol_2\Big).
			\end{aligned} \label{EqnPropTodaSubs}
		\end{equation}
	\end{proposition}
	\begin{proof}
		Recall $u=\log\frac{|\nabla{}z|^2}{|\nabla{}x|^2}$, $w=|\nabla{}z|^{-2}$ from (\ref{EqnLeBrunUW}).
		Thus $u_{xx}+u_{yy}+(e^u)_{zz}=0$ is
		\begin{equation}
			0
			\;=\;
			\left(\frac{\partial^2}{\partial{}x^2}+\frac{\partial^2}{\partial{}y^2}\right)\log|\nabla{}z|
			-\left(\frac{\partial^2}{\partial{}x^2}+\frac{\partial^2}{\partial{}y^2}\right)\log|\nabla{}x|
			+\frac12\left(\frac{|\nabla{}z|^2}{|\nabla{}x|^2}\right)_{zz}. \label{EqnTodaSubs}
		\end{equation}
		The classic Gaussian curvature expression is $K=-h^{-2}\left(\partial_{xx}+\partial_{yy}\right)(\log{}h)$ where $h=|\partial/\partial{}x|$ and $(x,y)$ is an isothermal coordinate system.
		Using $|\nabla{}x|=|\partial/\partial{}x|^{-1}$, $dVol_2=|\nabla{}x|^{-2}dx\wedge{}dy$, and the fact that $\mathcal{L}_{\frac{\partial}{\partial{}z}}(dx\wedge{}dy)=0$, then Eq. (\ref{EqnTodaSubs}) becomes
		\begin{equation}
			\small
			\begin{aligned}
				2K\;dVol_2
				&\;=\;
				2\left(|\nabla{}x|^{2}\left(\frac{\partial^2}{\partial{}x^2}+\frac{\partial^2}{\partial{}y^2}\right)\log|\nabla{}x|\right)\left(|\nabla{}x|^{-2}dx\wedge{}dy\right) \\
				&\;=\;
				\left(2\left(\frac{\partial^2}{\partial{}x^2}+\frac{\partial^2}{\partial{}y^2}\right)\log|\nabla{}z|
				+\frac{\partial^2}{\partial{}z^2}\frac{|\nabla{}z|^2}{|\nabla{}x|^2}\right)\;dx\wedge{}dy \\
				&\;=\;
				\left(\left(\frac{\partial^2}{\partial{}x^2}+\frac{\partial^2}{\partial{}y^2}\right)\log|\nabla{}z|^2\right)\;dx\wedge{}dy
				+\mathcal{L}_{\frac{\partial}{\partial{}z}}\mathcal{L}_{\frac{\partial}{\partial{}z}}\left(\frac{|\nabla{}z|^2}{|\nabla{}x|^2}\;dx\wedge{}dy\right).
			\end{aligned}
		\end{equation}
		Then using $(\partial^2_{xx}+\partial^2_{yy})=h^2\triangle_2$ and $dx\wedge{}dy=h^{-2}dVol_2$, we arrive at
		\begin{equation}
			\begin{aligned}
				2K\,dVol_2
				&\;=\;
				\Big(\triangle_2\log|V|^2\Big)\;dVol_2
				+\mathcal{L}_{\frac{\partial}{\partial{}z}}\mathcal{L}_{\frac{\partial}{\partial{}z}}\Big(|V|^2\,dVol_2\Big). \label{EqnGlobalSUToda}
			\end{aligned}
		\end{equation}
		For the second equation in (\ref{EqnPropTodaSubs}), use $\mathcal{L}_{\frac{\partial}{\partial{}z}}(|V|^2\,dVol_2)=(\triangle_4z)dVol_2$ from (\ref{EqnLieDeriveOfVSdVol}).
	\end{proof}
	\begin{proposition}
		In the scalar-flat case the equation $u_{xx}+u_{yy}+(e^u)_{zz}=0$ is
		\begin{equation}
			\begin{aligned}
				2K\,dVol_4
				&\;=\;
				\Big(\triangle_2\log|V|^2\Big)\;dVol_4
				+\mathcal{L}_{\frac{\partial}{\partial{}z}}\Big(\triangle_4z\,dVol_4\Big), \quad \text{or} \\
				2K\,dVol_4&\;=\;
				d\left[
				-Jd\log|V|^2\wedge*dVol_2
				+\frac{1}{|V|}\triangle_4z\,dVol_3
				\right].
			\end{aligned} \label{EqnKOnFourMfld}
		\end{equation}
	\end{proposition}
	\begin{proof}
		Taking the exterior product with $*dVol_2$ on both sides of (\ref{EqnPropTodaSubs}), then using the Leibniz rule, we have
		\begin{equation}
			\begin{aligned}
				2K\,dVol_4
				&\;=\;
				\Big(\triangle_2\log|V|^2\Big)\;dVol_4
				+\mathcal{L}_{\frac{\partial}{\partial{}z}}\Big(\triangle_4z\,dVol_2\Big)\wedge*dVol_2 \\
				&\;=\;
				\Big(\triangle_2\log|V|^2\Big)\;dVol_4
				+\mathcal{L}_{\frac{\partial}{\partial{}z}}\Big(\triangle_4z\,dVol_2\wedge*dVol_2\Big) \\
				&\hspace{1.4in}
				-(\triangle_4z)\,dVol_2\wedge\mathcal{L}_{\frac{\partial}{\partial{}z}}\Big(*dVol_2\Big).
			\end{aligned} \label{EqnCompKFourVolInterm}
		\end{equation}
		We show the third term on the right vanishes.
		Using the fourth equation from (\ref{EqnLieDeriveOfVSdVol}) of Lemma \ref{LemmaLieDerivsOfVols} we have
		\begin{equation}
			\begin{aligned}
				dVol_2\wedge\mathcal{L}_{\frac{\partial}{\partial{}z}}\Big(*dVol_2\Big)
				&=-\left(\frac{\partial}{\partial{}z}\log|V|^2\right)dVol_4
				+dz\wedge{}Jd|V|^{-2}\wedge{}dVol_2 \\
				&=-\left(\frac{\partial}{\partial{}z}\log|V|^2\right)dVol_4
				+dz\wedge{}Jd|V|^{-2}\wedge\omega
			\end{aligned}
		\end{equation}
		where we used $\omega=dVol_2+*dVol_2$ with $dz\wedge*dVol_2=0$.
		Then using (\ref{EqnBasicOmega}) to simplify $dz\wedge{}Jd|V|^{-2}\wedge\omega$ we obtain
		\begin{equation}
			\begin{aligned}
				&dVol_2\wedge\mathcal{L}_{\frac{\partial}{\partial{}z}}\Big(*dVol_2\Big)
				=
				-\left(\frac{\partial}{\partial{}z}\log|V|^2\right)dVol_4
				-\left<\nabla{}z,\,\nabla|V|^{-2}\right>\,dVol_4 \\
				&\hspace{1in}
				=-\left(\frac{\partial}{\partial{}z}\log|V|^2\right)dVol_4
				-\left(\frac{\partial}{\partial{}z}\log|V|^{-2}\right)\,dVol_4
				\;=\;0.
			\end{aligned}
		\end{equation}
		We conclude that (\ref{EqnCompKFourVolInterm}) is
		\begin{equation}
			\begin{aligned}
				2K\,dVol_4
				&\;=\;
				\Big(\triangle_2\log|V|^2\Big)\;dVol_4
				+\mathcal{L}_{\frac{\partial}{\partial{}z}}\Big(\triangle_4z\,dVol_4\Big)
			\end{aligned} \label{EqnKOnFourMfldFirstDerived}
		\end{equation}
		which is the first equation of (\ref{EqnKOnFourMfld}).
		For the second equation, note that $\mathcal{L}_{\frac{\partial}{\partial{}z}}=di_{\frac{\partial}{\partial{}z}}$ on 4-forms, and that $i_{\frac{\partial}{\partial{}z}}(dVol_4)=|V|^{-1}dVol_3$ by (\ref{EqnManyVols}).
		Thus
		\begin{equation}
			\small
			\begin{aligned}
				\mathcal{L}_{\frac{\partial}{\partial{}z}}\Big(\triangle_4z\,dVol_4\Big)
				\;=\;d\left(\frac{1}{|V|}\triangle_4z\,dVol_3\right).
			\end{aligned} \label{EqnPropFirstExtDerivSimpl}
		\end{equation}
		For the other term on the right side of (\ref{EqnKOnFourMfldFirstDerived}) we compute
		\begin{equation}
			\begin{aligned}
				&\left(\triangle_2\log|V|^2\right)\;dVol_4
				\;=\;
				\left(\triangle_2\log|V|^2\right)\;dVol_2\wedge*dVol_2 \\
				&\quad\quad
				\;=\;d\left(i_{\nabla{}\log|V|^2}dVol_2\right)\wedge*dVol_2 \\
				&\quad\quad
				\;=\;d\left(i_{\nabla{}\log|V|^2}dVol_2\wedge*dVol_2\right)
				+(i_{\nabla{}\log|V|^2}dVol_2)\wedge{}d(*dVol_2).
			\end{aligned} \label{EqnPropKEqnInter}
		\end{equation}
		We show the second term on the right vanishes.
		From Lemma \ref{EqnVolTwoExtDeriv} we have $d(*dVol)=\left(\frac{\partial}{\partial{}z}\log|\nabla{}x|^2\right)dz\wedge{}dVol_2$.
		Then
		\begin{equation}
			\begin{aligned}
				&(i_{\nabla{}\log|V|^2}dVol_2)\wedge{}d(*dVol_2) \\
				&\quad=
				\left(\frac{\partial}{\partial{}z}\log|\nabla{}x|^2\right)(i_{\nabla{}\log|V|^2}dVol_2)\wedge{}dVol_2\wedge{}dz \\
				&\quad=
				\frac12\left(\frac{\partial}{\partial{}z}\log|\nabla{}x|^2\right)i_{\nabla{}\log|V|^2}(dVol_2\wedge{}dVol_2)\wedge{}dz
				\;=\;0
			\end{aligned}
		\end{equation}
		where we used $i_v(\eta\wedge\eta)=2(i_v\eta)\wedge\eta$ whenever $\eta$ is a 2-form, and that $dVol_2\wedge{}dVol_2=0$ due to the fact that $dVol_2$ is decomposable (it is a multiple of $dx\wedge{}dy$).
		Therefore
		\begin{equation}
			\begin{aligned}
				&\triangle_2\log|V|^2\,dVol_4
				\;=\;
				d\Big((i_{\nabla{}\log|V|^2}dVol_2)\wedge*dVol_2\Big).
			\end{aligned}
		\end{equation}
		Because we have the two expressions
		\begin{equation}
			\begin{aligned}
				i_{\nabla{}\log|V|^2}dVol_2
				&=-\left(\frac{\partial}{\partial{}y}\log|V|^2\right)dx
				+\left(\frac{\partial}{\partial{}x}\log|V|^2\right)dy \\
				Jd\log|V|^2
				&=\left(\frac{\partial}{\partial{}y}\log|V|^2\right)dx
				-\left(\frac{\partial}{\partial{}x}\log|V|^2\right)dy
				-\left(\frac{\partial}{\partial{}z}\log|V|^2\right)V_\flat,
			\end{aligned}
		\end{equation}
		then using $*dVol_2=|V|^{-2}dz\wedge{}V_\flat$ we obtain
		\begin{equation}
			\begin{aligned}
				\left(i_{\nabla{}\log|V|^2}dVol_2\right)\wedge*dVol_2
				\;=\;
				-Jd\log|V|^2\wedge*dVol_2.
			\end{aligned}
		\end{equation}
		Inserting these back into (\ref{EqnKOnFourMfldFirstDerived}), we obtain as claimed
		\begin{equation}
			\begin{aligned}
				2K\,dVol_4
				&\;=\;
				d\left(-Jd\log|V|^2\wedge*dVol_2
				+\frac{1}{|V|}\triangle_4z\,dVol_3\right).
			\end{aligned}
		\end{equation}
	\end{proof}
	\begin{proposition}[cf. Theorem \ref{ThmConstantGrowth}] \label{PropGrowthOfVSqVol}
		Assume $M^4$ is scalar-flat and the level-sets $M^3_z$ are compact.
		Then at non-critical values of $z$ the value $\chi^g(M^2_z)$ is constant and
		\begin{equation}
			\frac{d^2}{dz^2}\int_{M^2_z}|V|^2\,dVol_2
			\;=\;\frac{d}{dz}\int_{M^2_z}\triangle_4z\,dVol_2
			\;=\;4\pi\chi^g(M^2_z). \label{EqnSecondOfTri}
		\end{equation}
	\end{proposition}
	\begin{proof}
		We use the convenient notation $M^4_{z_0,z_1}=\{p\in{}M^4\,\big|\,z_0<z(p)<z_1\}$.
		If $z$ is non-critical at some value, by smoothness and the fact that levelsets are compact, $z$ is non-critical on an interval $z_0<z<z_1$ around that value.
		Integrating both sides of (\ref{EqnKOnFourMfld}) gives
		\begin{equation}
			\small
			\begin{aligned}
				&\int_{M^4_{z_0,z_1}}2K\,dVol_4 \\
				&\hspace{0.5in}=
				\int_{M^4_{z_0,z_1}}
				d\left[
				\left(\frac{1}{|\nabla{}z|}\triangle_4{}z\right)dVol_3
				+dz\wedge{}V_\flat\wedge{}Jd|\nabla{}z|^{-2}
				\right].
			\end{aligned} \label{EqnFirstintForNow}
		\end{equation}
		For the left-hand side of (\ref{EqnFirstintForNow}), Lemma \ref{LemmaIntegration} gives
		\begin{equation}
			\begin{aligned}
				\int_{M^4_{z_0,z_1}}2K\,dVol_4
				=4\pi\int_{z_0}^{z_1} \left(\int_{M^2_z}K\,dVol_2\right)\,dz
				=8\pi^2\int_{z_0}^{z_1}\chi^g(M^2_z)\,dz.
			\end{aligned} \label{EqnDealWithK}
		\end{equation}
		due to the Gauss-Bonnet theorem on orbifolds which states $\chi^g(M^2_z)=\frac{1}{2\pi}\int_{M^2_z}K\,dVol_2$ (as discussed in the introduction).
		For the right-hand side of (\ref{EqnFirstintForNow}), by Stokes theorem and the Integration Lemma \ref{LemmaIntegration} we have
		\begin{equation}
			\begin{aligned}
				&\int_{M^4_{z_0,z_1}}
				d\left[
				\left(\frac{1}{|\nabla{}z|}\triangle_4{}z\right)dVol_3
				+dz\wedge{}V_\flat\wedge{}Jd|\nabla{}z|^{-2}
				\right] \\
				&\quad\quad=
				\int_{M^3_{z_1}}\left(\frac{1}{|\nabla{}z|}\triangle_4{}z\right)dVol_3
				-\int_{M^3_{z_0}}\left(\frac{1}{|\nabla{}z|}\triangle_4{}z\right)dVol_3 \\
				&\quad\quad=
				2\pi\int_{M^2_{z_1}}\triangle_4{}z\,dVol_2
				-2\pi\int_{M^2_{z_0}}\triangle_4{}z\,dVol_2
			\end{aligned} \label{EqnDealWithd}
		\end{equation}
		The terms of the form $\int_{M^3_z}dz\wedge{}V_\flat\wedge{}J|\nabla{}z|^{-2}$ disappear because $M^3_z$ is an integral leaf of $(\nabla{}z)^{\perp}$, so $dz$ vanishes on $M^3_z$.
		Combining (\ref{EqnDealWithK}) and (\ref{EqnDealWithd}), multiplying by $\frac{1}{z_1-z_0}$, and taking a limit, the Fundamental Theorem of Calculus gives
		\begin{equation}
			\frac{d}{dz}\int_{M^2_z}\triangle_4{}z\;dVol_2
			\;=\;
			4\pi\chi^g(M^2_z). \label{EqnDoubleDeriv}
		\end{equation}
		Lemma \ref{LemmaLieDerivsOfVols} gives $(\triangle_4z)dVol_2=\mathcal{L}_{\frac{\partial}{\partial{}z}}(|V|^2dVol_2)$.
		Therefore
		\begin{equation}
			\frac{d^2}{dz^2}\int_{M^2_z}|V|^2\,dVol_2
			=
			\frac{d}{dz}\int_{M^2_z}(\triangle_4z)\,dVol_2
			=
			2\int{}K\,dVol_2
			=
			4\pi\chi^g(M^2). \label{EqnInProofEulerOfBase}
		\end{equation}
		Lastly we show the constancy of $\chi^g(M^2_z)$.
		The left-hand side of (\ref{EqnDoubleDeriv}) is
		\begin{equation}
			\small
			\frac{d}{dz}\int_{M^2_z}(\triangle_4z)dVol_2
			=
			\frac{1}{2\pi}\frac{d}{dz}\int_{M^3_z}\frac{\triangle_4z}{|V|}\,dVol_3
			=
			\frac{1}{2\pi}\int_{M^3_z}\mathcal{L}_{\frac{\partial}{\partial{}z}}\left(\frac{1}{|V|}\triangle_4z\,dVol_3\right). \label{EqnBackToM3}
		\end{equation}
		Away from $|V|=0$, $M^3_z$ is a manifold and $\mathcal{L}_{\frac{\partial}{\partial{}z}}\left(\frac{1}{|V|}\triangle_4z\,dVol_3\right)$ is a smooth 3-form.
		Thus the quantities on both sides of (\ref{EqnDoubleDeriv}) are smoothly varying with $z$.
		But $\chi^g(M^2_z)$ is rational (see, for example, \cite{Scott83}), so therefore is a constant.
	\end{proof}
	
	\begin{corollary}[cf. Theorem \ref{ThmMfldEndTopology}] \label{CorSignedEuler}
		Assume $M^4$ has compact level-sets and $|V|>\epsilon>0$ for sufficiently large $z$ (or large $-z$).
		Then for large $z$ (or $-z$) the numbers $\chi^g(M^2_z)$ and $e^g(M^3_z,M^3_z)$ are constant.
		Further, we have $\chi^g(M^2_z)\ge0$ for large $z$ or $-z$.
		We have $e^g(M^3_z,M^2_z)\ge0$ for large $z$ and $e^g(M^3_z,M^2_z)\le0$ for large $-z$.
	\end{corollary}
	\begin{proof}
		By Lemma \ref{LemmaUnboundedZText} we have that the $M^3_z$ and fibrations $M^3_z\rightarrow{}M^2_z$ are isotopic for large $z$; this implies the constancy of $\chi^g(M^2_z)$ and $e^g(M^3_z,M^2_z)$.
		To see that $\chi^g(M^2_z)\ge0$ for large $z$, note that the range of $z$ is infinite, so $\frac{d^2}{dz^2}\int_{M^2_z}|V|^2dVol_2=4\pi\chi^g(M^2_z)$ forces the constant $\chi^g(M^2_z)$ to be non-negative, or else $\int_{M^2_z}|V|^2dVol_2$ would become negative.
		
		To see that $e^g(M^3_z,M^2_z)\ge0$ on $M^{4+}$, note that $\frac{d}{dz}\int_{M^2_z}dVol_2=2\pi{}e^g(M^3_z,M^2_z)$ for all large $z$ forces $e^g(M^3_z,M^2_z)\ge0$ or else $Vol(M^2_z)$ would become negative.
	\end{proof}
	
	\indent\textbf{Remark.}
	We have no need to prove it here, but when the scalar curvature $s$ of $M^4$ is non-zero, then (\ref{EqnsPointwiseEvoForK}) becomes
	\begin{equation}
		\begin{aligned}
			\left(2K-s\right)\,dVol_4
			&\;=\;
			d\left[
			-Jd\log|V|^2\wedge*dVol_2
			+\frac{1}{|V|}\triangle_4z\,dVol_3
			\right].
		\end{aligned} \label{EqnKsTransgressed}
	\end{equation}
	This resembles a Chern-Gauss-Bonnet transgression.
	Despite this resemblance, note that $(2K-s)dVol_4$ is not scale-invariant and this expression has no obvious interpretation as a product of symmetric invariant polynomials.

	%
	%
	%
	%
	%
	%
	\subsection{The Euclidean example} \label{SubsecEuclExample}
	We illustrate the results of the previous two sections with the example that $M^4$ is flat $\mathbb{C}^2$.
	The choices of Killing field with compact level-sets are given by the potential
	\begin{equation}
		z\;=\;-\frac12\alpha|z^1|^2-\frac12\beta|z^2|^2
	\end{equation}
	where $\alpha$, $\beta$ are both positive---if $\alpha$, $\beta$ have different signs then the $M^3_z$ are non-compact.
	For orbits to close up after a diffeomorphism flow of $2\pi$, $\alpha$ and $\beta$ must be coprime integers.
	The easiest example is $\alpha=\beta=1$, where the Killing field is
	\begin{equation}
		V
		=y^1\frac{\partial}{\partial{}x^1}
		-x^1\frac{\partial}{\partial{}y^1}
		+y^2\frac{\partial}{\partial{}x^2}
		-x^2\frac{\partial}{\partial{}y^2}
	\end{equation}
	and the orbits are Hopf circles.
	In this case $M^3_z$ is the sphere $\mathbb{S}^3(\sqrt{-2z})\subset\mathbb{C}^2$ of radius $\sqrt{-2z}$ and $M^2_z$ is the 2-sphere of radius $\sqrt{-z/2}$ and area $-2\pi{}z$.
	We have $|V|^2=2z$ and $\triangle_4{}z=-4$.
	We compute
	\begin{equation}
		\begin{aligned}
			&\int_{M^2_z}|V|^2dVol_2\;=\;4\pi(z)^2,
			\quad
			\frac{d^2}{dz^2}\int_{M^2_z}|V|^2dVol_2
			\;=\;8\pi
			\;=\;4\pi\chi^g(M^2_z) \\
			&\int_{M^2_z}\triangle_4z\,dVol_2
			\;=\;8\pi{}z,
			\quad
			\frac{d}{dz}\int_{M^2_z}\triangle_4z\,dVol_2
			\;=\;8\pi
			\;=\;4\pi\chi^g(M^2_z) \\
			&\int_{M^2_z}dVol_2
			\;=\;-2\pi{}z,
			\quad
			\frac{d}{dz}\int_{M^2_z}dVol_2
			\;=\;-2\pi
			\;=\;2\pi{}e^g(M^3_z,M^2_z),
		\end{aligned}
	\end{equation}
	recalling that the Hopf fibration has $\chi^g(\mathbb{S}^3,\mathbb{S}^2)=-1$.
	Thus we explicitly recover the conclusions of Propositions \ref{PropGrowthOfVSqVol} and \ref{PropConstantGrowthInText}.
	
	Referring to the construction of Proposition \ref{PropConstantGrowthInText}, we have $V_\flat=rJdr$ where $r=(|z^1|^2+|z^2|^2)^{1/2}$ and $dV_\flat=-2\omega$.
	Then the integrand of (\ref{EqnVolGrowthEquality}) is the quantity $|V|^{-4}V_\flat\wedge{}dV_\flat=-2r^{-3}dVol_3$ (the ``$-$'' sign comes from orientation).
	Using $r=\sqrt{-2z}$, we have
	\begin{equation}
		\small
		e^g(M^3_z,M^2_z)
		=\frac{1}{4\pi^2}\int_{M^3_z}|V|^{-4}V_\flat\wedge{}dV_\flat
		=-\frac{1}{4\pi^2}\int_{\mathbb{S}^3(\sqrt{-2z})}\frac{2}{\big(\sqrt{-2z}\big)^3}\,dVol_3
		=-1.
	\end{equation}
	This explicitly recovers Eq. (\ref{EqnVolGrowthEquality}).

	%
	%
	%
	%
	%
	%
	%
	%
	\section{The Chern-Gauss-Bonnet Formalism} \label{SecCGB}
	
	Here we compute the Chern-Gaus-Bonnet integrals based on the existence the symplectomorphic Killing field $V$.
	We show that one of these integrals reduces to
	\begin{equation}
		\frac{d^2}{dz^2}\int_{M^2_z}(\triangle_4z)^2\,dVol_2
		=2\int_{M^2_z}|\Ric|^2\,dVol_2
	\end{equation}
	when level-sets are compact and scalar curvature is zero.
	The computation requires precise manipulation of the details of the theory; for example sign choices are crucial and not always consistent in the literature.
	We take the time to start from the beginning and be careful.
	We conclude in \S\ref{SubSectCGBExamples} by giving explicit demonstrations of our results on the scalar flat LeBrun instantons \cite{LeB89}.
	
	First we record a well-known Bochner identity.
	\begin{lemma} \label{LemmaRicciBochner}
		$d(\triangle_4z)+2\Ric(\nabla{}z)=0$.
		In particular $\frac{\partial}{\partial{}z}(\triangle_4z)=-2\Ric\left(\frac{\nabla{}z}{|\nabla{}z|},\,\frac{\nabla{}z}{|\nabla{}z|}\right)$.
	\end{lemma}
	\begin{proof}
		First we compute $\triangle{}V_\flat$.
		The Killing condition $V_{i,j}=-V_{j,i}$ forces $V_{j,j}=0$ so the usual commutation identity gives
		\begin{equation}
			\begin{aligned}
				V_{i,jj}
				&\;=\;-V_{j,ij}
				\;=\;-V_{j,ji}-\Riem{}_{ijjs}V_s\;=\;-\Ric{}_{is}V_s.
			\end{aligned}
		\end{equation}
		Because $J$ is covariant constant and $\Ric$ is $J$-invariant, this is
		\begin{equation}
			\triangle_4{}V+\Ric(V)\;=\;0
			\quad\quad\text{or}\quad\quad
			\triangle_4{}dz\;+\;\Ric(\nabla{}z)\;=\;0.
		\end{equation}
		Then we compute
		\begin{equation}
			z_{,ijj}
			\;=\;z_{,jij}
			\;=\;z_{,jji}+\Riem{}_{ijjs}z_{,s}
			\;=\;z_{,jji}+\Ric{}_{is}z_{,s}
		\end{equation}
		which is $\triangle_4dz=d\triangle_4z+\Ric(\nabla{}z)$.
		Thus $d\triangle_4z+2\Ric(\nabla{}z)=0$.
	\end{proof}

	%
	%
	%
	%
	%
	%
	\subsection{The transgression construction} \label{SubSecCGBBasics}
	We review the basic theory and set notation; this section's objective is to record Eq. \ref{EqnFirstTrans}.
	References are \cite{GH78} \cite{KN96} \cite{CS74}.
	Despite our K\"ahler setting it will be easiest to assume the structure group is $SO(4)$, due to the ease of representing the Hodge star on the Lie algebra $\mathfrak{o}(4)$.
	
	Recall the bracket $[A,B]$ when $A\in\bigwedge{}^k\otimes\mathfrak{o}(4)$ and $B\in\bigwedge{}^l\otimes\mathfrak{o}(4)$.
	This is defined by extending the rule $[\eta\otimes\mathfrak{a},\mu\otimes\mathfrak{b}]=\eta\wedge\mu\otimes[\mathfrak{a},\mathfrak{b}]$ by linearity.
	We have the commutation identity $[A,B]=(-1)^{kl+1}[B,A]$.
	
	Let $D$ be a connection $D:\bigwedge^k\otimes\mathfrak{o}(4)\rightarrow\bigwedge^{k+1}\otimes\mathfrak{o}(4)$; then $D$ obeys the Leibniz rule $D(\eta\wedge{}A)=d\eta\wedge{}A+(-1)^{k}\eta\wedge{}DA$ when $\mu\in\bigwedge{}^k$ and $A\in\bigwedge{}^k\otimes\mathfrak{g}$.
	$D$ can be expressed locally as a $\mathfrak{g}$-valued $1$-form $\theta$, where we have $DA=dA+[\theta,A]$.
	The curvature of $D$ is $F=D\circ{}D=d\theta+\frac12[\theta,\theta]\in\bigwedge{}^2\otimes\mathfrak{g}$.
	
	We use two different, but related, Hodge-star operators: one that operates on the vector space $\bigwedge^2$ and the other that operates on the Lie algebra $\mathfrak{o}(4)$.
	They are abstractly the same but it is important for practical reasons to distinguish them, so we use $*:\bigwedge^2\rightarrow\bigwedge^2$ and write $*\eta$ when $\eta$ is a 2-form, and $\star:\mathfrak{o}(4)\rightarrow\mathfrak{o}(4)$ and write $\mathfrak{a}\star$ when $\mathfrak{a}\in\mathfrak{o}(4)$.
	If $A=\eta\otimes\mathfrak{a}$ is a $\mathfrak{o}(4)$-valued 2-form, we write
	\begin{equation}
		*A=(*\eta)\otimes\mathfrak{a}, \quad\quad
		A\star=\eta\otimes(\mathfrak{a}\star), \quad\quad
		*A\star=(*\eta)\otimes(\mathfrak{a}\star).
	\end{equation}
	In a basis, the operators $*$ and $\star$ have expressions in terms of the Levi-Civita antisymmetric symbol $\epsilon$.
	If $\eta=\eta_{ij}$ is a 2-form then $(*\eta)_{ij}=\frac12\epsilon^{kl}{}_{ij}\eta_{kl}$ and if $\mathfrak{a}\in\mathfrak{o}(4)$ is an antisymmetric matrix $\mathfrak{a}=\mathfrak{a}^i_j$ then $(\mathfrak{a}\star)^i_j=-\frac12\mathfrak{a}^k_l\epsilon^{il}{}_{jk}$.
	The operators $*$, $\star$ create projectors $\pi^\pm:\bigwedge^2\rightarrow\bigwedge^\pm$, $\pi^\pm:\mathfrak{o}(4)\rightarrow\mathfrak{o}(3)$ where $\pi^\pm=\frac12(1\pm*)$, $\pi^\pm=\frac12(1\pm\star)$, respectively.
	
	It is well-known that the vector space of quadratic symmetric invariant polynomials on $\mathfrak{o}(4)$ is spanned by just two polynomials, which in polar form are $\mathcal{P}_1(\mathfrak{a},\mathfrak{b})=-\frac12Tr(\mathfrak{a}\mathfrak{b})$ (which 
	is half the polarization of the matrix norm) and $\mathcal{P}_2(\mathfrak{a},\mathfrak{b})=-\frac12Tr(\mathfrak{a}\mathfrak{b}\star)$ (which is twice the polarization of the Pfaffian).
	If $\mathcal{P}$ is a quadratic polynomial on $\mathfrak{o}(4)$, it extends to the algebra $\bigoplus_k\bigwedge{}^k\otimes\mathfrak{o}(4)$ by setting
	\begin{equation}
		\mathcal{P}(\eta\otimes\mathfrak{a},\,\mu\otimes\mathfrak{b})
		\;=\;\eta\wedge\mu\cdot\mathcal{P}(\mathfrak{a},\mathfrak{b})
	\end{equation}
	and extending by bilinearlity.
	For $A\in\bigwedge^k\otimes\mathfrak{o}(4)$, $A\in\bigwedge^l\otimes\mathfrak{o}(4)$ we have the commutator and Leibniz identities
	\begin{equation}
		\begin{aligned}
			&\mathcal{P}(A,\,B)\;=\;(-1)^{kl}\mathcal{P}(B,\,A), \\
			&d\mathcal{P}(A,\,B)\;=\;\mathcal{P}(DA,\,B)+(-1)^k\mathcal{P}(A,\,DB).
		\end{aligned} \label{EqnLeibniz}
	\end{equation}
	Below, we will make use of the identity
	\begin{equation}
		\mathcal{P}(\eta\otimes\mathfrak{a},\;\eta\wedge\mu\otimes\mathfrak{b})\;=\;0 \label{EqnAntiSymOnOneForms}
	\end{equation}
	whenever $\eta\in\bigwedge^1$ and $\mu$ is any other form.
	
	Now assume $\bar\theta\in\bigwedge^1\otimes\mathfrak{g}$ has the particular form $\bar\theta=\eta\otimes{}\mathfrak{a}$ where $\eta$ is a 1-form and $\mathfrak{a}$ is a section $\mathfrak{a}\in\mathfrak{g}$.
	Then $D_1=d+[\theta-\bar\theta,\,\cdot]$ is another connection.
	Interpolating between the original connection $D=d+\theta$ with curvature $F$ and the new connection $D_1=d+\theta-\bar\theta$ with curvature $F_1$, for each $t$ we have a connection $D_t=d+\theta-t\bar\theta$ with curvature
	\begin{equation}
		\begin{aligned}
			F_t
			&=
			d(\theta-t\bar\theta)+\frac12[\theta-t\bar\theta,\,\theta-t\bar\theta]
			=
			d\theta-td\bar\theta
			+\frac12[\theta,\,\theta]
			-[\bar\theta,\,\theta]+t^2[\bar\theta,\,\bar\theta].
		\end{aligned} \label{EqnSpecialF}
	\end{equation}
	Because we assumed $\bar\theta=\eta\otimes{}A$ we have $[\bar\theta,\bar\theta]=\eta\wedge\eta\otimes[A,A]=0$.
	Therefore
	\begin{equation}
		\begin{aligned}
			F_t
			&\;=\;
			d\theta+\frac12[\theta,\,\theta]-t\left(d\bar\theta+[\theta,\bar\theta]\right)
			\;=\;F\,-\,tD\bar\theta.
		\end{aligned}
	\end{equation}
	Then when $\mathcal{P}$ a symmetric invariant polynomial we have
	\begin{equation}
		\mathcal{P}(F_1,F_1)\,-\,\mathcal{P}(F,F)
		\;=\;\int_0^1\frac{d}{dt}\mathcal{P}(F_t,F_t)\,dt
		\;=\;2\int_0^1\mathcal{P}\left(\frac{d}{dt}F_t,F_t\right)\,dt.
	\end{equation}
	By (\ref{EqnSpecialF}) we have $F_t=F-tD\bar\theta$ so $\frac{d}{dt}F_t=-D\bar\theta$.
	Thus
	\begin{equation}
		\begin{aligned}
			\mathcal{P}(F_1,F_1)\,-\,\mathcal{P}(F,F)
			&\;=\;-2\int_0^1\mathcal{P}\left(D\bar\theta,F-tD\bar\theta\right)\,dt \\
			&\;=\;
			-2\mathcal{P}\left(D\bar\theta,F\right)
			+\mathcal{P}\left(D\bar\theta,D\bar\theta\right).
		\end{aligned} \label{EqnFirstTransgressionComp}
	\end{equation}
	To simplify this, we use the Leibniz rule (\ref{EqnLeibniz}) and the fact that $DF=0$ to obtain
	\begin{equation}
		\mathcal{P}\left(D\bar\theta,F\right)
		=d\left(\mathcal{P}\left(\bar\theta,F\right)\right)
	\end{equation}
	Then by the Leibniz rule we have
	\begin{equation}
		\begin{aligned}
			\mathcal{P}\left(D\bar\theta,D\bar\theta\right)
			&\;=\;d\left(\mathcal{P}\left(\bar\theta,D\bar\theta\right)\right)
			+\mathcal{P}\left(\bar\theta,D\circ{}D\bar\theta\right) \\
			&\;=\;d\left(\mathcal{P}\left(\bar\theta,D\bar\theta\right)\right)
			+\mathcal{P}\left(\bar\theta,F\wedge\bar\theta\right) \\
			&\;=\;d\left(\mathcal{P}\left(\bar\theta,D\bar\theta\right)\right)
			+\mathcal{P}\left(\eta\otimes{A},F\wedge\eta\otimes{A}\right)
			\;=\;d\left(\mathcal{P}\left(\bar\theta,D\bar\theta\right)\right)
		\end{aligned}
	\end{equation}
	where we also used (\ref{EqnAntiSymOnOneForms}).
	Equation (\ref{EqnFirstTransgressionComp}) is therefore
	\begin{equation}
		\mathcal{P}(F,F)
		\;=\;
		\mathcal{P}(F_1,F_1)
		+d\Big(2\mathcal{P}(\bar\theta,\,F)
		-\mathcal{P}(\bar\theta,D\bar\theta)\Big). \label{EqnFirstTrans}
	\end{equation}
	
	The curvature tensor $F\in\bigwedge^2\otimes\mathfrak{o}(4)$ is usually expressed as a matrix of 2-forms $F=(F^i_j)$.
	For the Levi-Civita connection, $F$ relates to the Riemann tensor $\Riem$ by
	\begin{equation}
		F^i_j\;=\;-\frac12\Riem{}_{stj}{}^i\,\eta^s\wedge\eta^t
	\end{equation}
	where we use the sign convention $sec(e_i,e_j)=\Riem_{ijji}$.
	The decomposition $\bigwedge^2=\bigwedge^+\oplus\bigwedge^-$ breaks $\Riem\in\bigwedge^2\otimes\bigwedge^2$ into the four components $\Riem^{\pm\pm}\in\bigwedge^\pm\otimes\bigwedge^\pm$.
	For example $\Riem^{+-}=\frac14(1+*)F(1-\star)$.
	It is well-known \cite{SiTh69}, \cite{Besse} that
	\begin{equation}
		\begin{aligned}
			&\Riem{}^{++}=W^++\frac{s}{12}Id_{\bigwedge^+}, \quad
			\Riem{}^{--}=W^-+\frac{s}{12}Id_{\bigwedge^-}, \\
			&\hspace{0.8in}
			\Riem{}^{+-}+\Riem{}^{-+}=\frac12\cRic\KNP{}g.
		\end{aligned} \label{EqnRmDecomp}
	\end{equation}
	It is typical \cite{DK97} to use $F^+$ to indicate $\frac12(F+F\star)$---as opposed to $\frac12(F+*F)$---so for example $F^+=\Riem{}^{++}+\Riem^{-+}$.
	For K\"ahler metrics, Proposition 2 of \cite{Derd83} states $W^+=-\frac{s}{12}\left(Id_{\bigwedge^+}+\frac32\omega\otimes\omega\right)$.
	(To resolve a sign issue with Derdzhinski's Eq. (36) of \cite{Derd83}, because our convention on the Riemann tensor is the opposite to the one used there, a tensor $T_{ijkl}:\bigwedge^+\rightarrow\bigwedge^+$ operates by $T(\zeta)_{ij}=\frac12T_{ijst}\zeta^{ts}$---notice the index reversal---so that $Id_{\bigwedge^+}=-\frac12\omega\otimes\omega-\frac12\eta\otimes\eta-\frac12\mu\otimes\mu$ where $\eta,\mu\in\bigwedge^+$ are orthogonal and of length $\sqrt{2}$.)
	We have the well known facts for K\"ahler metrics
	\begin{equation}
		\Riem{}^{++}\;=\;-\frac{s}{8}\omega\otimes\omega,
		\quad\quad
		|\Riem{}^{++}|^2\;=\;\frac14s^2. \label{EqnDerdRiemPlPl}
	\end{equation}

	%
	%
	%
	%
	%
	%
	\subsection{Construction of the Chern-Gauss-Bonnet integrands}
	This section's objective is to record Eq. (\ref{EqnTransgressionOfRic}).
	Recall the decomposition of $dJdz\in\bigwedge^2$ into its $\bigwedge^\pm$ components gives
	\begin{equation}
		dJdz\;=\;-\frac12(\triangle_4{}z)\,\omega\,+\,(dJdz)^-. \label{EqnDJDDecomp}
	\end{equation}
	Because $Jdz=-V_\flat$ we have $(dV_\flat)^+=\frac12(\triangle_4z)\omega$.
	Let $\{e_1,e_2,e_3,e_4\}$ be any orthonormal frame and $\{\eta^1,\eta^2,\eta^3,\eta^4\}$ its coframe.
	Because $\nabla{}V_\flat$ is antisymmetric, then $\nabla{}V$ is antisymmetric in the sense that $V^i{}_{,j}=\eta^i(\nabla_{e_j}V)$ is an anti-symmetric matrix, so it is a section with values in $\mathfrak{o}(4)$.
	By (\ref{EqnBasicFirst}) we have $(dJdz)_{ij}=-2V^s{}_{,j}g_{si}$.
	Using $\nabla{}V^+=\frac12(\nabla{}V+\nabla{}V\star)$ and (\ref{EqnDJDDecomp}) we have
	\begin{equation}
		(\nabla{}V)^+\;=\;-\frac14(\triangle{}z)\,J
		\quad\text{or}\quad
		\left[(\nabla{}V)^+\right]^i_j
		\;=\;-\frac14(\triangle{}z)\,J_j{}^i \label{EqnPosDerivOfV}
	\end{equation}
	where $J=J_s{}^t\eta^s\otimes{}e_t$ is the complex structure.
	(The ``$-$'' sign in (\ref{EqnPosDerivOfV}) comes from $\omega_{sj}g^{si}=-J_j{}^i$.)
	
	Let $\{e_1,e_2,e_3,e_4\}$ be an orthonormal frame adapted to $V$ in the sense that $[V,e_i]=0$ and let $\{\eta^1,\eta^2,\eta^3,\eta^4\}$ be the associated coframe.
	Let $\nabla=d+\theta$ be Levi-Civita differentiation, so $\theta$ is an $\mathfrak{o}(4)$-valued 1-form.
	Letting $w$ be any vector, $i_w\theta$ is a matrix with components $i_w\theta^i_j=\eta^i(\nabla_we_j)$.
	Recalling that $\nabla{}V$ is antisymmetric, we create a new $\mathfrak{o}(4)$-valued 1-form $\bar\theta$ by
	\begin{equation}
		\bar\theta\;=\;\frac{1}{|V|^2}V_\flat\otimes\nabla{}V. \label{EqnThetaBar}
	\end{equation}
	Then $\bar\theta\in\bigwedge{}^1\otimes\mathfrak{o}(4)$ is a tensor away from $V=0$.
	Therefore $\theta-\bar\theta$ is a connection 1-form.
	It has connection $D_1=d+[\theta-\bar\theta,\,\cdot]$ and curvature $F_1$.
	We show that $V$ is a null vector of $F_1$.
	First we verify that the connection $\theta-\bar\theta$ has a null vector:
	\begin{equation}
		\begin{aligned}
			i_V(\theta-\bar\theta)^i_j
			&\;=\;i_V\theta^i_j-i_V\left(\eta^1\otimes\frac{1}{|V|}\eta^i(\nabla_{e_j}V)\right) \\
			&\;=\;\eta^i(\nabla_Ve_j)-\eta^i(\nabla_{e_j}V)
			\;=\;\eta^i([V,\,e_j])
			\;=\;0.
		\end{aligned}
	\end{equation}
	Next we verify that $\theta$ and $\bar\theta$ are $V$-invariant in the sense that $\mathcal{L}_V\theta=0$.
	Very simply, $\mathcal{L}_V\theta^i_j=\mathcal{L}_V(\eta^j(\nabla{}e_j))$; then adaptedness of the frame gives $\mathcal{L}_V(\eta^i)=0$ and $\mathcal{L}_V(e_j)=0$, and we have $\mathcal{L}_V(\nabla)=0$ because $V$ is Killing.
	Thus the Leibniz rule gives $\mathcal{L}_V\theta=0$.
	Essentially the same argument gives $\mathcal{L}_V\bar\theta=0$.
	
	Now we establish that $F_1$ has a null vector.
	Because $i_V(\theta-\bar\theta)=0$ we have $i_V[\theta-\bar\theta,\theta-\bar\theta]=0$.
	Because $i_V(\theta-\bar\theta)=0$ and $\mathcal{L}_V(\theta-\bar\theta)=0$, then the fact that $\mathcal{L}_V=di_V+i_Vd$ on $\mathfrak{o}(4)$-valued 1-forms gives $i_Vd(\theta-\bar\theta)=0$.
	Therefore
	\begin{equation}
		\begin{aligned}
			i_VF_1
			&=i_Vd(\theta-\bar\theta)\,-\,\frac12i_V[\theta-\bar\theta,\,\theta-\bar\theta]
			=0.
		\end{aligned}
	\end{equation}
	Because $F_1$ has a null vector, the 4-form $\mathcal{P}(F_1,F_1)$ has a null vector so certainly $\mathcal{P}(F_1,F_1)=0$.
	Therefore (\ref{EqnFirstTrans}) gives
	\begin{equation}
		\begin{aligned}
			\mathcal{P}(F,F)
			&=d\,\bigg(2\mathcal{P}(\bar\theta,\,F)-\mathcal{P}(\bar\theta,\,D\bar\theta)\bigg).
		\end{aligned} \label{EqnComputationOfTRansgressions}
	\end{equation}
	We apply this to the usual Chern-Gauss-Bonnet formulas for the signature $\tau(M^4)$ and Euler invariant $\chi(M^4)$.
	The 4-forms $\mathcal{P}_\tau=\frac{1}{12\pi^2}\mathcal{P}_1(F,F)$ and $\mathcal{P}_\chi=\frac{1}{8\pi^2}\mathcal{P}_2(F,F)$ have the expressions\footnote{Some presentations have a factor of 1 rather than $\frac14$ on the Weyl components.
		This depends on whether the Weyl tensors are normed as operators $\bigwedge^\pm\rightarrow\bigwedge^\pm$ as in \cite{Derd83}, or as tensors as $|W|^2=W_{ijkl}W^{ijkl}$.
		The latter is 4 times the former.
		We use the latter, hence the factor of $\frac14$.}
	\begin{equation}
		\begin{aligned}
			&\mathcal{P}_{\tau}\;=\;\frac{1}{12\pi^2}\left(\frac14|W^+|^2-\frac14|W^-|^2\right)\,dVol, \\
			&\mathcal{P}_{\chi}\;=\;\frac{1}{8\pi^2}\left(\frac{1}{24}s^2\,-\,\frac12|\cRic|^2\,+\,\frac14|W^+|^2\,+\,\frac14|W^-|^2\right)\,dVol.
		\end{aligned}
	\end{equation}
	By (\ref{EqnComputationOfTRansgressions}) we have transgressions
	\begin{equation}
		\begin{aligned}
			&\mathcal{P}_{\tau}\;=\;d\mathcal{TP}_\tau \quad\text{where}\quad
			\mathcal{TP}_\tau\;=\;-\frac{1}{24\pi^2}\left(2\bar\theta^i_j\wedge{}F^j_i - \bar\theta^i_j\wedge{}D\bar\theta^j_i\right), \quad\text{and} \\
			&\mathcal{P}_{\chi}\;=\;d\mathcal{TP}_\chi \quad\text{where}\quad
			\mathcal{TP}_\chi\;=\;-\frac{1}{16\pi^2}\left(2\bar\theta^i_j\wedge(F\star)^j_i - \bar\theta^i_j\wedge(D\bar\theta\star)^j_i\right)
		\end{aligned} \label{EqnTwoTransgressionComp}
	\end{equation}
	which hold whenever $V\ne0$.
	Using $|\Riem|^2=\frac16s^2+2|\cRic|^2+|W^+|^2+|W^-|^2$ and $|\Riem{}^{++}|=\big|\frac{1}{12}s\,Id_{\bigwedge^+}+W^+\big|^2=\frac{1}{12}s^2+|W^+|^2$, we obtain the useful combination
	\begin{equation}
		\begin{aligned}
			-8\pi^2(3\mathcal{P}_\tau+2\mathcal{P}_\chi)
			\;=\;
			\big(|\cRic|^2\,-\,|\Riem{}^{++}|^2\big)\,dVol.
		\end{aligned}
	\end{equation}
	Other combinations of $\mathcal{P}_{\chi}$ and $\mathcal{P}_{\tau}$ are possible, but this combination is particularly useful because of the simplicity of its transgression:
	\begin{equation}
		\begin{aligned}
			-8\pi^2(3\mathcal{TP}_\tau+2\mathcal{TP}_\chi),
			=2\bar\theta^i_j\wedge\left(2(F^+)^j_i-(D\bar\theta^+)^j_i\right)
		\end{aligned} \label{EqnsMainTransgression}
	\end{equation}
	where the $D\bar\theta^+$ term is easy to work with due to (\ref{EqnPosDerivOfV}).
	We abbreviate
	\begin{equation}
		\begin{aligned}
			\mathcal{P}_{\Ric}
			&\;\triangleq\;-8\pi^2(3\mathcal{P}_\tau+2\mathcal{P}_\chi),
			\quad\text{which is transgressed by} \\
			\mathcal{TP}_{\Ric}
			&\;\triangleq\;2\bar\theta^i_j\wedge\left(2(F^+)^j_i-(D\bar\theta^+)^j_i\right)
		\end{aligned}
	\end{equation}
	so $\mathcal{P}_{\Ric}=d\mathcal{TP}_{\Ric}$ when $V\ne0$.
	Lastly we consider the K\"ahler case.
	In that case $|\Riem{}^{++}|^2=\frac14s^2$ and $|\cRic|^2=|\Ric|^2-\frac14s^2$ and we have the key expression
	\begin{equation}
		\small
		\begin{aligned}
			\mathcal{P}_{\Ric}
			&=\left(|\Ric|^2\,-\,\frac12s^2\right)\,dVol_4
			=
			d\left[
			2\bar\theta^i_j\wedge\left(2(F^+)^j_i
			-(D\bar\theta^+)^j_i\right)
			\right]
			=d\mathcal{TP}_{\Ric}.
		\end{aligned} \label{EqnTransgressionOfRic}
	\end{equation}
	\begin{proposition}[cf. Theorem \ref{PropIntroCompactWithNoZeros}] \label{PropMainTextCompactWithNoZeros}
		Let $M^4$ be a compact, scalar-flat K\"ahler manifold with a Killing field $V$.
		If $V$ is nowhere zero, then $M^4$ is flat.
	\end{proposition}
	\begin{proof}
		The scalar-flat K\"ahler condition implies $s$, $W^+$, and $\Riem{}^{++}$ are all zero.
		The fact that $V\ne0$ implies the transgression on the right-hand side (\ref{EqnTransgressionOfRic}) is smooth.
		Integration by parts gives
		\begin{equation}
			\int_{M^4}|\Ric|^2\,dVol_4
			\;=\;
			\int_{M^4}d\left[
			2\bar\theta^i_j\wedge\left(2(F^+)^j_i-(D\bar\theta^+)^j_i\right)
			\right]
			\;=\;0
		\end{equation}
		so that $\Ric\equiv0$.
		Similarly, $\mathcal{P}_\tau=-\frac{1}{96\pi^2}|W^-|^2$ is smoothly transgressed by (\ref{EqnTwoTransgressionComp}), so integrating gives $W^-\equiv0$.
		We conclude that $\Riem\equiv0$.
	\end{proof}

	%
	%
	%
	%
	%
	%
	\subsection{Simplifying the Transgression}
	We simplify $\mathcal{TP}_{\Ric}$ in three steps, at the end arriving at a very usable expression.
	The aim is to record Eq. (\ref{EqnSecondDerivOfTriSq}).
	
	Given a K\"ahler metric, recall the Ricci tensor is $J$-invariant and $\rho=\Ric(J\cdot,\cdot)$ is a $(1,1)$-form called the Ricci form.
	We use $\rho_0=\cRic(J\cdot,\cdot)$ for the primitive Ricci form, and note that $\rho_0=\rho-\frac{1}{4}s\omega$.
	For convenience we set $\eta=\frac{1}{|V|}V_\flat$.
	\begin{lemma}[Step 1] \label{LemmaTPRicComp1}
		If the metric is K\"ahler and $V$ is Hamiltonian then
		\begin{equation}
			\begin{aligned}
				\mathcal{TP}_{\Ric}
				&\;=\;
				2\left(\frac{1}{|V|}\triangle_4z\right)\,\eta\wedge\rho
				+\frac12\left(\frac{1}{|V|}\triangle_4z\right)^2\eta\wedge{}d\eta.
			\end{aligned} \label{EqnTPRicFirst}
		\end{equation}
		If $V$ is not Hamiltonian, (\ref{EqnTPRicFirst}) remains true after replacing $\triangle_4z$ by $*\left[(dV_\flat)^+\wedge\omega\right]$.
	\end{lemma}
	\begin{proof}
		In the Hamiltonian case, (\ref{EqnPosDerivOfV}) gives $(dV_\flat)^+=\frac12(\triangle_4z)\omega$, which gives $\triangle_4z=*\left[(dV_\flat)^+\wedge\omega\right]$.
		In the non-Hamiltonian case, we still have that $(dV_\flat)^+$ is proportional to $\omega$ and that $(dV_\flat)^+=\frac12*\left(dV_\flat\wedge\omega\right)\cdot\omega$.
		In everything that follows we begin our computations with $dV_\flat$ or $\nabla{}V$, and we can simply take $\triangle_4z$ to be shorthand for the quantity $*\left(dV_\flat\wedge\omega\right)$.
		
		We first evaluate $Tr(\bar\theta\wedge{}F^+)$.
		Using (\ref{EqnPosDerivOfV}) we have
		\begin{equation}
			\begin{aligned}
				Tr(\bar\theta\wedge{}F^+)
				&\;=\;\frac{1}{|V|}\eta\wedge(\nabla{}V)^s_t(F^+)_s^t \\
				&\;=\;\frac{1}{|V|}\eta\wedge(\nabla{}V^+)^s_t(F^+)_s^t
				\;=\;-\frac14(\triangle_4z)\frac{1}{|V|}\eta\wedge(F^+)^t_sJ_t{}^s.
			\end{aligned} \label{EqnFirstToughFplusComp}
		\end{equation}
		We have $F^+=-\frac12(\Riem^{++}+\Riem^{-+})$.
		Because $(\Riem^{+-})^j_i(\nabla{}V^+)^i_j=0$ we can write
		\begin{equation}
			\begin{aligned}
				(F^+)^t_sJ_t{}^s
				&\;=\;
				-\frac12\left(\Riem{}^{++}+\Riem{}^{-+}\right)_{ijs}{}^tJ_t{}^s\,\cdot\,\eta^i\wedge\eta^j \\
				&\;=\;
				-\frac12\left(\Riem{}^{++}+\Riem{}^{-+}+\Riem{}^{+-}\right)_{ijst}\omega^{ts}\,\cdot\,\eta^i\wedge\eta^j
			\end{aligned}
		\end{equation}
		By (\ref{EqnDerdRiemPlPl}) we have $\Riem^{++}=-\frac{s}{8}\omega\otimes\omega$ and using $\omega_{st}\omega^{st}=4$ we have
		\begin{equation}
			-\frac12(\Riem{}^{++})_{ijst}\omega^{ts}
			=-\frac{s}{4}\omega_{ij} \label{EqnRmPlPlJ}
		\end{equation}
		By (\ref{EqnRmDecomp}) we have $(\Riem{}^{-+}+\Riem{}^{+-})=\frac12\cRic\,\KNP{}g$, so
		\begin{equation}
			\begin{aligned}
				&-\frac12(\Riem{}^{+-}+\Riem{}^{-+})_{ijst}\omega^{ts}\cdot\eta^i\wedge\eta^j \\
				&\hspace{0.4in}=
				-\frac14\left(
				\cRic{}_{it}g_{js}+\cRic{}_{js}g_{it}
				-\cRic{}_{is}g_{jt}-\cRic{}_{jt}g_{is}
				\right)\omega^{ts}\cdot\eta^i\wedge\eta^j \\
				&\hspace{0.4in}=
				-(\rho_0)_{ij}\cdot\eta^i\wedge\eta^j
				\;=\;-2\rho_0.
			\end{aligned}
		\end{equation}
		Combining these, we have $(F^+)_s^tJ_t{}^s=-2\rho_0-\frac{s}{2}\omega=-2\rho$.
		Therefore
		\begin{equation}
			Tr(\bar\theta\wedge{}F^+)
			\;=\;\frac12\left(\frac{1}{|V|}\triangle_4z\right)\,\eta\wedge\rho.
		\end{equation}
		Next we compute $Tr(\bar\theta\wedge{}D\bar\theta^+)$.
		Using the Leibniz rule,
		\begin{equation}
			\begin{aligned}
				D\bar\theta^+
				&\;=\;D\left(\eta\otimes\frac{1}{|V|}DV^+\right)
				\;=\;d\eta\otimes\frac{1}{|V|}\nabla{}V^+
				\,-\,\eta\wedge{}D\left(\frac{1}{|V|}\nabla{}V^+\right).
			\end{aligned} \label{EqnLeibnizExtCompI}
		\end{equation}
		For the last term, we use $|V|^{-1}\nabla{}V^+=-\frac14|V|^{-1}(\triangle_4z)J$ and the fact that $J$ is covariant-constant to obtain
		\begin{equation}
			\begin{aligned}
				D\bar\theta^+
				&\;=\;d\eta\otimes\frac{1}{|V|}\nabla{}V^+
				\,+\,\frac14\eta\wedge{}d\left(\frac{1}{|V|}\triangle_4z\right)\otimes{}J.
			\end{aligned} \label{EqnLeibnizExtCompII}
		\end{equation}
		Therefore
		\begin{equation}
			\begin{aligned}
				Tr(\bar\theta\wedge{}D\bar\theta^+)
				&\;=\;
				\eta\wedge{}d\eta\,\cdot\,|V|^{-2}Tr\left((\nabla{}V)(\nabla{}V^+)\right) \\
				&\quad\hspace{0.4in}
				+\eta\wedge\eta\wedge{}d\left(\frac{1}{|V|}\triangle_4z\right)\,\cdot\,Tr\left(\frac{1}{|V|}(\nabla{}V)(J)\right) \\
				&\;=\;
				\eta\wedge{}d\eta\,\cdot\,|V|^{-2}Tr\left((\nabla{}V)(\nabla{}V^+)\right)
			\end{aligned}
		\end{equation}
		because $\eta\wedge\eta=0$.
		Then using (\ref{EqnPosDerivOfV}) and $Tr(JJ)=J_i{}^jJ_j{}^i=-|J|^2=-4$, we have
		\begin{equation}
			Tr\left((\nabla{}V)(\nabla{}V^+)\right)
			\;=\;
			Tr\left((\nabla{}V^+)(\nabla{}V^+)\right)
			\;=\;-|\nabla{}V^+|^2
			\;=\;-\frac14\left(\triangle_4z\right)^2
		\end{equation}
		and we conclude that
		\begin{equation}
			Tr(\bar\theta\wedge{}D\bar\theta)
			\;=\;-\frac14\left(\frac{1}{|V|}\triangle_4z\right)\,\eta\wedge{}d\eta.
		\end{equation}
		Combining the two pieces, we obtain as promised
		\begin{equation}
			\begin{aligned}
				TP_{\Ric}
				&\;=\;4\bar\theta^i_j\wedge(F^+)^j_i
				-2\bar\theta^i_j\wedge(D\bar\theta^+)^j_i \\
				&\;=\;
				2\left(\frac{1}{|V|}\triangle_4z\right)\,\eta\wedge\rho
				+\frac12\left(\frac{1}{|V|}\triangle_4z\right)^2\eta\wedge{}d\eta.
			\end{aligned} \label{EqnComputedFirst}
		\end{equation}
	\end{proof}
	\begin{lemma}[Step 2] \label{LemmaTPRicComp2}
		If the metric is K\"ahler and $V$ is Hamiltonian we have
		\begin{equation}
			\begin{aligned}
				\mathcal{TP}_{\Ric}
				&\;=\;
				\frac12*d\left(\frac{1}{|V|}\triangle_4z\right)^2
				+\left(\frac{1}{|V|}\triangle_4z\right)\left[s+\frac12\left(\frac{1}{|V|}\triangle_4z\right)^2\right]\,dVol_3.
			\end{aligned} \label{EqnTPSecond}
		\end{equation}
		If $V$ is not Hamiltonian, (\ref{EqnTPSecond}) remains true after replacing $\triangle_4z$ with $*\left[(dV_\flat)^+\wedge\omega\right]$.
	\end{lemma}
	\begin{proof}
		We evaluate $\eta\wedge\rho$ using the Bochner identity $d(\triangle_4z)+2\Ric(\nabla{}z,\cdot)=0$ of Lemma \ref{LemmaRicciBochner}.
		This with $*\rho=\rho^+-\rho^-=-\rho+2\rho^+$ and $\rho^+=\frac{s}{4}\omega$ gives
		\begin{equation}
			\rho\;=\;-*\rho\,+\,\frac{s}{2}\omega.
		\end{equation}
		We have the usual identity $*(\mu\wedge*\gamma)=\gamma(\mu^\sharp,\cdot)$ when $\mu$ is a $1$-form and $\gamma$ is a 2-form.
		Using $\eta=|V|^{-1}V_\flat$ and the fact that $**:\bigwedge{}^3\rightarrow\bigwedge^3$ is multiplication by $-1$ gives
		\begin{equation}
			\begin{aligned}
				\eta\wedge\rho
				&\;=\;-\eta\wedge*\rho\,+\,\frac{s}{2}\,\eta\wedge\omega
				\;=\;\frac{1}{|V|}*\left(i_V\rho\,-\,\frac{s}{2}\,i_V\omega\right).
			\end{aligned}
		\end{equation}
		Then $\rho(\cdot,\cdot)=\Ric(J\cdot,\cdot)$ along with (\ref{EqnBasicZV}) and Lemma \ref{LemmaRicciBochner} gives
		\begin{equation}
			\begin{aligned}
				\eta\wedge\rho
				&\;=\;\frac{1}{|V|}*\left(-\Ric(\nabla{}z)\,+\,\frac{s}{2}\,dz\right)
				\;=\;\frac12\frac{1}{|V|}*\Big(d(\triangle_4{}z)\,+\,s\,dz\Big).
			\end{aligned}
		\end{equation}
		To compute $\eta\wedge{}d\eta$, note $\eta\wedge{}d\eta=|V|^{-2}V_\flat\wedge{}dV_\flat$.
		Then Eq. (\ref{EqnVdVFirst}) of Lemma \ref{LemmaVdV} is
		\begin{equation}
			\eta\wedge{}d\eta
			\;=\;
			-\frac{1}{|V|^2}*d|V|^2
			+\frac{1}{|V|^2}(\triangle_4z)*dz
		\end{equation}
		Placing these expressions for $\eta\wedge\rho$ and $\eta\wedge{}d\eta$ into (\ref{EqnComputedFirst}), we have
		\begin{equation}
			\begin{aligned}
				2\mathcal{TP}_{\Ric}
				&\;=\;
				4\left(\frac{1}{|V|}\triangle_4z\right)\left(\frac12\frac{1}{|V|}*\Big(d(\triangle_4{}z)\,+\,s\,dz\Big)\right) \\
				&\quad\quad
				+\left(\frac{1}{|V|}\triangle_4z\right)^2\left(-\frac{1}{|V|^2}*d|V|^2
				+\frac{1}{|V|^2}(\triangle_4z)*dz \right).
			\end{aligned}
		\end{equation}
		Collecting terms,
		\begin{equation}
			\begin{aligned}
				2\mathcal{TP}_{\Ric}&\;=\;
				2\left(\frac{1}{|V|}\triangle_4z\right)
				*\left[
				\frac{1}{|V|}d(\triangle_4z)
				-(\triangle_4{}z)\frac{1}{|V|^2}d|V|
				\right] \\
				&\quad\quad
				+\left(\frac{1}{|V|}\triangle_4z\right)
				\left[
				2s
				+\left(\frac{1}{|V|}\triangle_4z\right)^2
				\right]\frac{1}{|V|}\,*dz \\
				&\;=\;
				*d\left(\frac{1}{|V|}\triangle_4z\right)^2
				+\left(\frac{1}{|V|}\triangle_4z\right)
				\left[
				2s
				+\left(\frac{1}{|V|}\triangle_4z\right)^2
				\right]\,dVol_3.
			\end{aligned} \label{EqnComputedSecond}
		\end{equation}
	\end{proof}
	\begin{lemma}[Step 3] \label{LemmaUsableCGB}
		If the metric is K\"ahler and $V$ is Hamiltonian we have
		\begin{equation}
			\begin{aligned}
				\mathcal{TP}_{\Ric}
				&\;=\;
				\frac12\mathcal{L}_{\frac{\partial}{\partial{}z}}\left((\triangle_4z)^2\frac{1}{|V|}\,dVol_3\right)
				\,-\,\frac12(*dVol_2)\wedge{}Jd\left(\frac{1}{|V|}\triangle_4z\right)^2 \\
				&\hspace{0.8in}
				+s\left(\frac{1}{|V|}\triangle_4z\right)\,dVol_3.
			\end{aligned} \label{EqnTPThird}
		\end{equation}
		If $V$ is not Hamiltonian, replace $\frac{\partial}{\partial{}z}$ with $-|V|^{-2}JV_\flat$ and $\triangle_4z$ with $*(dV^+\wedge\omega)$.
	\end{lemma}
	\begin{proof}
		We begin with
		\begin{equation}
			\begin{aligned}
				*d\left(\frac{\triangle_4z}{|V|}\right)^2
				=\frac{\partial}{\partial{}z}\left(\frac{\triangle_4z}{|V|}\right)^2*dz
				+\frac{\partial}{\partial{}x}\left(\frac{\triangle_4z}{|V|}\right)^2*dx
				+\frac{\partial}{\partial{}y}\left(\frac{\triangle_4z}{|V|}\right)^2*dy.
			\end{aligned} \label{EqnWeStartWith}
		\end{equation}
		The last two terms on the right simplify: whenever $f=f(z,x,y)$ then
		\begin{equation}
			f_x*dx-f_y*dy
			=-(*dVol_2)\wedge{}Jdf
		\end{equation}
		The first term also simplifies using $*dz=V_\flat\wedge{}dVol_2$.
		Then (\ref{EqnWeStartWith}) becomes
		\begin{equation}
			*d\left(\frac{\triangle_4z}{|V|}\right)^2
			=\frac{\partial}{\partial{}z}\left(\frac{\triangle_4z}{|V|}\right)^2\,V_\flat\wedge{}dVol_2
			-(*dVol_2)\wedge{}Jd\left(\frac{\triangle_4z}{|V|}\right)^2.
		\end{equation}
		Use the Leibniz rule on the first term and Lemma \ref{LemmaLieDerivsOfVols}, we obtain
		\begin{equation}
			\begin{aligned}
				&\frac{\partial}{\partial{}z}\left(\frac{\triangle_4z}{|V|}\right)^2\;V_\flat\wedge{}dVol_2\\
				&\quad\quad\;=\;\mathcal{L}_{\frac{\partial}{\partial{}z}}\left(\left(\frac{1}{|V|}\triangle_4z\right)^2V_\flat\wedge{}dVol_2\right)
				-\left(\frac{1}{|V|}\triangle_4z\right)^2\mathcal{L}_{\frac{\partial}{\partial{}z}}\left(V_\flat\wedge{}dVol_2\right) \\
				&\quad\quad\;=\;\mathcal{L}_{\frac{\partial}{\partial{}z}}\left(\frac{1}{|V|}(\triangle_4z)^2\,dVol_3\right)
				-\left(\frac{1}{|V|}\triangle_4z\right)^3\,dVol_3.
			\end{aligned}
		\end{equation}
		Combining, we obtain
		\begin{equation}
			\begin{aligned}
				&*d\left(\frac{1}{|V|}\triangle_4z\right)^2
				+\left(\frac{1}{|V|}\triangle_4z\right)^3dVol_3 \\
				&\hspace{0.2in}=\mathcal{L}_{\frac{\partial}{\partial{}z}}\left(\frac{1}{|V|}(\triangle_4z)^2\,dVol_3\right)
				-*dVol_2\wedge{}Jd\left(\frac{1}{|V|}\triangle_4z\right)^2.
			\end{aligned}
		\end{equation}
		Substituting this into (\ref{EqnTPSecond}) gives the result.
	\end{proof}
	
	\begin{proposition}[cf. Theorem \ref{ThmConstantGrowth}] \label{PropTwoDerivsOfLaplSq}
		Assume $M^4$ is scalar-flat and assume the leaf $M^3_z$ is compact and non-singular.
		Then
		\begin{equation}
			\frac{d^2}{dz^2}\int_{M^2_z}(\triangle_4{}z)^2\,dVol_2
			\;=\;
			2\int_{M^2_z}|\Ric|^2\,dVol_2 \label{EqnSecondDerivOfTriSq}
		\end{equation}
	\end{proposition}
	\begin{proof}
		With $s=0$, inserting (\ref{EqnTPThird}) into (\ref{EqnTransgressionOfRic}) gives
		\begin{equation}
			\small
			2|\Ric|^2\,dVol_4
			=
			d\left[\mathcal{L}_{\frac{\partial}{\partial{}z}}\left((\triangle_4z)^2\frac{1}{|V|}\,dVol_3\right)
			-\frac{dz}{|V|}\wedge\frac{V_\flat}{|V|}\wedge{}Jd\left(\frac{1}{|V|}\triangle_4z\right)^2
			\right].
		\end{equation}
		Choose $z_0<z_1$, and denote $M^4_{z_0,z_1}=\{p\in{}M^4\big|z_0<z(p)<z_1\}$.
		Integrating,
		\begin{equation}
			\small
			\begin{aligned}
				&\int_{M^4_{z_0,z_1}}|\Ric|^2\,dVol_4
				\;=\;
				\frac12\int_{M^3_{z_1}}\mathcal{L}_{\frac{\partial}{\partial{}z}}\left((\triangle_4z)^2\frac{1}{|V|}\,dVol_3\right) \\
				&\hspace{1.7in}-\frac12\int_{M^3_{z_0}}\mathcal{L}_{\frac{\partial}{\partial{}z}}\left((\triangle_4z)^2\frac{1}{|V|}\,dVol_3\right) \\
				&\hspace{0.6in}
				\;=\;
				\frac12\frac{\partial}{\partial{}z}\Big|_{z=z_1}\int_{M^3_{z}}(\triangle_4z)^2\frac{1}{|V|}\,dVol_3
				-\frac12\frac{\partial}{\partial{}z}\Big|_{z=z_0}\int_{M^3_{z}}(\triangle_4z)^2\frac{1}{|V|}\,dVol_3
			\end{aligned} \label{EqnPreDeriv}
		\end{equation}
		where the term $|V|^{-2}dz\wedge{}V_\flat\wedge{}Jd(|V|^{-1}\triangle_4z)^2$ vanished because $dz\equiv0$ restricted to $M^3_z=\{z=const\}$.
		The Integration Lemma, Lemma \ref{LemmaIntegration}, gives
		\begin{equation}
			\int_{M^3_z}(\triangle_4z)^2\frac{1}{|V|}\,dVol_3
			\;=\;
			2\pi\int_{M^2_z}(\triangle_4z)^2\,dVol_2
		\end{equation}
		Then we consider the left-hand side.
		From Lemma \ref{LemmaIntegration}
		\begin{equation}
			\int_{M^4_{z_0,z_1}}|\Ric|^2\,dVol_4
			\;=\;
			2\pi\int_{z_0}^{z_1}\left(\int_{M^2_z}|\Ric|^2\,dVol_2\right)dz
		\end{equation}
		From this, we multiply both sides of (\ref{EqnPreDeriv}) by $\frac{1}{z_1-z_0}$ and take the limit $z_1-z_0\rightarrow0$.
		The Fundamental Theorem of Calculus then gives
		\begin{equation}
			\begin{aligned}
				2\int_{M^2_z}|\Ric|^2\,dVol_2
				&\;=\;
				\frac{d^2}{dz^2}\int_{M^2_{z}}(\triangle_4z)^2\,dVol_2
			\end{aligned} \label{EqnDeriv}
		\end{equation}
	\end{proof}

	%
	%
	%
	%
	%
	%
	\subsection{The LeBrun instanton examples.} \label{SubSectCGBExamples}
	We produce nontrivial examples demonstrating Lemma \ref{LemmaTPRicComp2} and Proposition \ref{PropTwoDerivsOfLaplSq}.
	To be able to easily compute Ricci curvatures, we use the framework of \cite{NaffWeber}.
	In that framework, $C=C(w)$ and $F=F(w)$ and the metric is
	\begin{equation}
		\begin{aligned}
			&g\;=\;C\left(\frac{1}{2F}dw^2\,+\,F(\eta^1)^2+(\eta^2)^2+(\eta^3)^2\right), \\
			&J(dw)\;=\;-2F\eta^1, \quad J(\eta^2)\;=\;-\eta^3, \\
			&\omega\;=\;\frac12Cdw\wedge\eta^1+C\eta^2\wedge\eta^3
		\end{aligned} \label{EqnsUTwo}
	\end{equation}
	where $\eta^1$, $\eta^2$, $\eta^3$ form the left-invariant coframe on $\mathbb{S}^2$ with $d\eta^i=-\epsilon^i{}_{jk}\eta^j\wedge\eta^k$.
	This framework is very flexible: when $C=C_0e^{-w}$, the metric is K\"ahler for any $F$.
	
	From \cite{LeB89}, the LeBrun instanton metric of mass $m$ on the surface $O(-k)$ is
	\begin{equation}
		\begin{aligned}
		&g
		=\frac{1}{f(r)}(dr)^2
		+r^2\left(f(r)(\eta^1)^2+(\eta^2)^2+(\eta^3)^2\right), \text{ where} \\
		&f(r)
		=\left(1-\frac{m^2}{r^2}\right)\left(1+\frac{(k-1)m^2}{r^2}\right)
		\end{aligned}
	\end{equation}
	Changing coordinates by $r=e^{-w/2}$ to express this in the form (\ref{EqnsUTwo}), we find $C=e^{-w}$ and $F=1+m^2(k-2)e^{w}-m^4(k-1)e^{2w}$ with coordinate range $w\in(-\infty,-\log\,m^2]$.
	For example $k=1$ is the Burns metric and $k=2$ is the Eguchi-Hanson metric.
	For any $w<-\log(m^2)$ the level-set $M^3_w$ is the lens space $L(k,1)$ and the symplectic reduction $M^2_w$ is a smooth 2-sphere.
	For $w=const$ we have $Vol_3(M^3_w)=\frac2k\pi^2e^{-w}$ and $Vol_2(M^2_w)=\pi{}e^{-w}.$
	From (13) of \cite{NaffWeber}
	\begin{equation}
		|\Ric|^2\;=\;16m^4(k-2)^2e^{4w}.
	\end{equation}
	Letting $\psi,\theta,\varphi$ be the Euler coordinates on $\mathbb{S}^3$, the Killing field is $V=\frac{2}{k}\frac{\partial}{\partial\psi}$ (see \S6 of \cite{Weber24a}) and using the fact that $i_{V}\eta^1=\frac1k$ we obtain the momentum function $z=-\frac{1}{2k}(e^{-w}-m^2)$ with range $z\in(-\infty,0]$.
	We compute for general $F=F(w)$
	\begin{equation}
		\begin{aligned}
			&dJdz=-\frac1k\frac{\partial}{\partial{}w}(e^{-w}F)dw\wedge\eta^1
			+\frac2k{}e^{-w}F\eta^1\wedge\eta^2, \\
			&dJdz\wedge\omega=-\frac1ke^{w}\left(\frac{\partial}{\partial{}w}(e^{-w}F)-e^{-w}F\right)\omega\wedge\omega
		\end{aligned}
	\end{equation}
	so therefore $\triangle_4z=-\frac4k-\frac{2(k-2)}{k}m^2e^{w}$.
	Changing variables from $w$ to $z$, we find
	\begin{equation}
		\begin{aligned}
			&\triangle_4z\;=\;-2\frac{-4z+m^2}{-2kz+m^2}, \\
			&Vol_2(M^2_z)\;=\;\pi(-2kz+m^2), \\
			&\int_{M^2_z}\triangle_4z\,dVol_2
			\;=\;2\pi(4z-m^2), \\
			&\int_{M^2_z}(\triangle_4z)^2\,dVol_2
			\;=\;4\pi\frac{(-4z+m^2)^2}{-2kz+m^2}, \\
			&\int_{M^2_z}|\Ric|^2\,dVol_2
			\;=\;16\pi{}m^4(k-2)^2\frac{1}{(-2kz+m^2)^3}.
		\end{aligned} \label{EqnLeBrunEquations}
	\end{equation}
	We have $M^2_z\approx\mathbb{S}^2$ for each $z$ so $\chi^g(M^2_z)=2$.
	Each $M^3_z$ is the lense space $L(k,1)$ and $M^3_z\rightarrow{}M^2_z$ has the usual fiber bundle structure (the standard projection along a $k$-to-1 quotient of the Hopf fibration), so the Euler number of this $\mathbb{S}^1$ fiber bundle is $e^g(M^3_z,M^2_z)=-k$.
	Direct computation from (\ref{EqnLeBrunEquations}) gives
	\begin{equation}
		\begin{aligned}
			&\frac{d}{dz}Vol_2(M^2_z)
			\;=\;-2k\pi\;=\;2\pi{}e^g(M^3_z,M^2_z) \\
			&\frac{d}{dz}\int_{M^2_z}\triangle_4z\,dVol_2
			\;=\;8\pi{}\;=\;4\pi\chi^g(M^2_z)
		\end{aligned}
	\end{equation}
	which recovers (\ref{EqnSecondOfTri}) of \S\ref{SecEvoK} and (\ref{EqnVolGrowthEquality}) from \S\ref{SecKillingTwoForm}.
	Further, we have
	\begin{equation}
		\small
		\begin{aligned}
			&\frac{d}{dz}\int_{M^2_z}(\triangle_4z)^2\,dVol_2
			\;=\;8\pi\frac{(-4z+m^2)(4kz-(k-4)m^2)}{(-2kz+m^2)^2} \\
			&\frac{d^2}{dz^2}\int_{M^2_z}(\triangle_4z)^2\,dVol_2
			\;=\;32\pi{}m^4(k-2)^2\frac{1}{(-2kz+m^2)^3}
			\;=\;2\int_{M^2_z}|\Ric|^2\,dVol_2
		\end{aligned}
	\end{equation}
	which recovers the Chern-Gauss-Bonnet relation Eq. (\ref{EqnSecondDerivOfTriSq}) of Proposition \ref{PropTwoDerivsOfLaplSq}.

	%
	%
	%
	%
	%
	%
	%
	%
	\section{Type II metrics} \label{SecTypeIandII}
	
	We prove Theorem \ref{ThmTypeIIa}, Proposition \ref{PropTypeIIb}, and Theorem \ref{ThmIntroCptLevelsNoZeros}.

	%
	%
	%
	%
	%
	%
	\subsection{The basic inequalities} \label{SubSecBasicIneq}
	
	Recall that Type IIa metrics have compact level-sets and have a lower bound $|V|>\epsilon>0$ outside a compact set.
	Theorems \ref{ThmTypeIIa} and \ref{PropTypeIIb} both follow from the basic relationship between $\int_{M^2_z}dVol_2$, $\int_{M^2_z}(\triangle_4z)dVol_2$, and $\int_{M^2_z}(\triangle_4z)^2dVol_2$ from H\"older's inequality that
	\begin{equation}
		\left(\int_{M^2_z}\triangle_4z\,dVol_2\right)^2
		\le
		\int_{M^2_z}dVol_2
		\;
		\int_{M^2_z}(\triangle_4z)^2dVol_2. \label{EqnHolder}
	\end{equation}
	The quantities in (\ref{EqnHolder}) are controlled by other factors: $\int_{M^2_z}dVol_2$ is controlled by $e^g(M^3_z,M^2_z)$ by Proposition \ref{PropConstantGrowthInText}, $\int_{M^2_z}\triangle_4z\,dVol_2$ is controlled by $\chi^g(M^2_z)$ by Proposition \ref{PropGrowthOfVSqVol}, and $\int_{M^2_z}(\triangle_4z)^2dVol_2$ is controlled by $\int_{M^4}|\Ric|^2dVol_4$ by Proposition \ref{PropTwoDerivsOfLaplSq}.
	It will be convenient to abbreviate $\chi^g=\chi^g(M^2_z)$ and $e^g=e^g(M^3_z,M^2_z)$.
	\begin{lemma}[The Basic Inequalities] \label{LemmaBasicIneqs}
		Assume $M^4$ is Type IIa, $\mathcal{C}=\{z_0,\dots,z_N\}$ are the critical values of $z$, and either $M^{4+}$ or $M^{4-}$ is non-empty.
		
		If $M^{4+}$ and/or $M^{4-}$ is non-empty, we have
		\begin{equation}
			\begin{aligned}
				&\chi^g\ge0, \; e^g\ge0 \quad \text{on $M^{4+}$}, \\
				&\chi^g\ge0, \; e^g\le0 \quad \text{on $M^{4-}$}.
			\end{aligned} \label{EqnEulerNumberSigns}
		\end{equation}
		Assuming the set of critical values $\mathcal{C}$ is non-empty, set
		\begin{equation}
			\begin{aligned}
				&C_0=\lim_{z\searrow{}z_N}\int_{M^2_z}|V|^2\,dVol_2, \quad
				C_0'=\lim_{z\nearrow{}z_0}\int_{M^2_z}|V|^2\,dVol_2, \\
				&C_1=\lim_{z\searrow{}z_N}\int_{M^2_z}\triangle_4z\,dVol_2, \quad
				C_1'=\lim_{z\nearrow{}z_0}\int_{M^2_z}\triangle_4z\,dVol_2, \\
				&C_2=\lim_{z\searrow{}z_N}\int_{M^2_z}dVol_2, \hspace{0.425in}
				C_2'=\lim_{z\nearrow{}z_0}\int_{M^2_z}dVol_2.
			\end{aligned}
		\end{equation}
		Then we have
		\begin{equation}
			\begin{aligned}
				\int_{M^2_z}|V|^2dVol_2
				&\;\le\;\left(2\pi\chi^g\cdot{}(z-z_N)^2+C_1\cdot(z-z_N)+C_0\right) \quad \text{on $M^{4+}$} \\
				\int_{M^2_z}|V|^2\,dVol_2
				&\;\le\;\left(2\pi\chi^g\cdot{}(z-z_0)^2+C_1'\cdot(z-z_0)+C_0'\right) \quad \text{on $M^{4-}$}
			\end{aligned} \label{IneqVSqIntegrated}
		\end{equation}
		and
		\begin{equation}
			\small
			\begin{aligned}
				&|4\pi\chi^g\cdot(z-z_N)+C_1|
				\le\sqrt{\big(2\pi{}e^g\cdot(z-z_N)+C_2\big)\int_{M^2_z}(\triangle_4z)^2\,dVol_2} \quad \text{on $M^{4+}$} \\
				&|4\pi\chi^g\cdot(z-z_0)+C_1'|
				\le\sqrt{\big(2\pi{}e^g\cdot(z-z_0)+C_2'\big)\int_{M^2_z}(\triangle_4z)^2\,dVol_2} \quad \text{on $M^{4-}$}.
			\end{aligned} \label{IneqHolderBasic}
		\end{equation}
		If $|z|$ is sufficiently large then $|V|>\epsilon$; for these $z$ we have
		\begin{equation}
			\begin{aligned}
				2\pi{}e^g\cdot(z-z_N)+C_1
				&\;\le\;\epsilon^{-2}\int_{M^2_z}|V|^2\,dVol_2 \quad\text{on $M^{4+}$} \\
				2\pi{}e^g\cdot(z-z_0)+C_1'
				&\;\le\;\epsilon^{-2}\int_{M^2_z}|V|^2\,dVol_2 \quad\text{on $M^{4-}$}.
			\end{aligned} \label{IneqVolAndVSq}
		\end{equation}
	\end{lemma}
	\begin{proof}
		Proposition \ref{PropConstantGrowthInText} gives
		\begin{equation}
			\begin{aligned}
				&Vol(M^2_z)\;=\;2\pi{}e^g\cdot(z-z_N)+C_0{}\quad\text{on $M^{4+}$}, \\
				&Vol(M^2_z)\;=\;2\pi{}e^g\cdot(z-z_0)+C_0'{}\quad\text{on $M^{4-}$}.
			\end{aligned}
		\end{equation}
		Because $Vol(M^2_z)>0$, this forces $e^g\ge0$ on $M^{4+}$ and $e_g\le0$ on $M^{4-}$.
		Proposition \ref{PropGrowthOfVSqVol} gives on $M^{4+}$
		\begin{equation}
			\begin{aligned}
				&\int_{M^2_z}|V|^2=2\pi\chi^g\cdot(z-z_N)^2+C_1\cdot(z-z_N)+C_2\quad\text{and} \\
				&\hspace{0.4in}\;\int_{M^2_z}\triangle_4z\,dVol_2=4\pi\chi^g\cdot(z-z_N)+C_1,
			\end{aligned}
		\end{equation}
		and on $M^{4-}$
		\begin{equation}
			\begin{aligned}
				&\int_{M^2_z}|V|^2=2\pi\chi^g\cdot(z-z_0)^2+C_1'\cdot(z-z_0)+C_2'\quad\text{and} \\
				&\hspace{0.4in}\int_{M^2_z}\triangle_4z\,dVol_2=4\pi\chi^g\cdot(z-z_0)+C_1'.
			\end{aligned}
		\end{equation}
		Because $\int_{M^2_z}|V|^2dVol_2\ge0$, this forces $\chi^g\ge0$ on both $M^{4+}$ and $M^{4-}$.
		This proves (\ref{EqnEulerNumberSigns}) and (\ref{IneqVSqIntegrated}).
		Also (\ref{IneqHolderBasic}) follows from H\"older's inequality (\ref{EqnHolder}).
		
		Because the metric is Type II, the set $\{|V|<\epsilon\}$ has compact closure, and by Lemma \ref{LemmaUnboundedZText} the sets $M^4_{z_0,z_1}=\{a\in{}M^4\;|\;z_0<z(a)<z_1\}$ are compact, but exhaust $M^4$ as $z_1\rightarrow{}\infty$ and $z_0\rightarrow-\infty$.
		Therefore sufficiently small $z_0$ and sufficiently large $z_1$ exist so that $|V|>\epsilon$ on $M^4\setminus{}M^4_{z_0,z_1}$.
		Therefore when $|z|>\text{max}\{|z_0|,|z_1|\}$
		\begin{equation}
			Vol(M^2_z)
			=\int_{M^2_z}|V|^{-2}|V|^2\,dVol_2
			\le\epsilon^{-2}\int_{M^2_z}|V|^2\,dVol_2
		\end{equation}
		which verifies (\ref{IneqVolAndVSq}).
	\end{proof}

	%
	%
	%
	%
	%
	%
	\subsection{Proofs of Theorem \ref{ThmTypeIIa} and Proposition \ref{PropTypeIIb}} \label{SubSecProofsOfThreeThms}
	
	\begin{proposition}[cf. Theorem \ref{ThmTypeIIa}] \label{PropTypeIIaMainText}
		Assume $M^4$ is Type IIa and $\int_{M^4}|\Ric|^2<\infty$.
		If $e^g(M^3_z,M^2_z)=0$ on $M^{4+}$ (resp. $M^{4-}$), then $\chi^g(M^2_z)=0$ on $M^{4+}$ (resp. $M^{4-}$).
	\end{proposition}
	\begin{proof}
		Assume $M^{4+}$ is non-empty and $e^g=0$ (the argument for $M^{4-}$ is entirely similar).
		Eq. (\ref{IneqHolderBasic}) gives
		\begin{equation}
			\left(4\pi\chi^g\cdot(z-z_N)+C_1\right)^2
			\;\le\;C_2\,\int_{M^2_z}(\triangle_4z)^2\,dVol_2. \label{EqnQuadraticLeft}
		\end{equation}
		Set
		\begin{equation}
			\begin{aligned}
				&C_3=\lim_{z\rightarrow{}z_N{}^+}\frac{d}{dz}\int_{M^3_z}(\triangle_4z)^2\,dVol_2, \quad\text{and} \\
				&C_4=\lim_{z\rightarrow{}z_N{}^+}\int_{M^3_z}(\triangle_4z)^2\,dVol_2.
			\end{aligned}
		\end{equation}
		Because $z$ has no critical points on $M^{4+}$, Proposition \ref{PropTwoDerivsOfLaplSq} along with the Integration Lemma gives
		\begin{equation}
			\frac{d}{dz}\int_{M^3_z}(\triangle_4z)^2\,dVol_2
			\;=\;
			C_3
			\,+\,\int_{M^4_{z_N,z}}|\Ric|^2\,dVol_4. \label{EqnDerivOfDzSq}
		\end{equation}
		Because $\frac{d}{dz}\int_{M^3_z}(\triangle_4z)^2\,dVol_2$ is monotonically increasing, sending $z\rightarrow\infty$ on the right side of (\ref{EqnDerivOfDzSq}), that $\frac{d}{dz}\int_{M^3_z}(\triangle_4z)^2dVol_2\le{}C_3+\int_{M^{4+}}|\Ric|^2dVol_4$.
		Thus
		\begin{equation}
			\int_{M^2_z}(\triangle_4z)^2\,dVol_2
			\;\le\;\left(C_3+\int_{M^{4+}}|\Ric|^2\right)z\,+\,C_4
		\end{equation}
		for large $z$.
		The fact that the left side of (\ref{EqnQuadraticLeft}) is quadratic when $\chi^g\ne0$ but the right side has linear bounds forces $\chi^g=0.$
	\end{proof}
	
	\begin{proposition}[cf. Prop \ref{PropTypeIIb}] \label{PropTypeIIbMainText}
		Assume $M^4$ is Type IIa and Ricci-flat.
		Then $M^4$ is one-ended so either $M^{4-}=\varnothing$ of $M^{4+}=\varnothing$, and $\triangle_4z$ is constant on $M^4$.
		For non-critical $z$ we have $2e^g(M^3_z,M^2_z)\cdot\triangle_4z=\chi^g(M^2_z)$.
		Further, if $M^4$ is non-flat then $\triangle_4z\ne0$.
		If $\triangle_4z>0$ then $M^{4-}=\varnothing$ and if $\triangle_4z<0$ then $M^{4+}=\varnothing$.
	\end{proposition}
	\begin{proof}
		Because $\Ric=0$, Lemma \ref{LemmaRicciBochner} shows $d(\triangle_4z)=0$ so $\triangle_4z$ is constant.
		Then at any non-critical value of $z$, by Proposition \ref{PropGrowthOfVSqVol} we have $(\triangle_4z)\frac{d}{dz}Vol(M^2_z)=2\pi\chi^g(M^2_z)$ and by Proposition \ref{PropConstantGrowthInText} we have $\frac{d}{dz}Vol(M^2_z)=4\pi{}e^g(M^3_z,M^2_z)$.
		Therefore $2e^g(M^3_z,M^2_z)\triangle_4{}z=\chi^g(M^2_z)$ as claimed.
		
		To see that $M^4$ must be one-ended, by Proposition \ref{PropGrowthOfVSqVol} we have
		\begin{equation}
			\begin{aligned}
				\frac{d}{dz}\int_{M^2_z}|V|^2dVol_2
				&\;=\;\int_{M^2_z}\triangle_4z\,dVol_2 \\
				&\;=\;4\pi\chi^g\cdot(z-z_N)\,+\,\int_{M^2_{z_N}}\triangle_4z\,dVol_2
				\quad\text{on $M^{4+}$}\\
				\frac{d}{dz}\int_{M^2_z}|V|^2dVol_2
				&\;=\;\int_{M^2_z}\triangle_4z\,dVol_2 \\
				&\;=\;4\pi\chi^g\cdot(z-z_0)\,+\,\int_{M^2_{z_0}}\triangle_4z\,dVol_2
				\quad\text{on $M^{4-}$.}
			\end{aligned}
		\end{equation}
		If $\triangle_4z<0$ then the first equation forces $\chi^g=0$ on $M^{4+}$, but then $\frac{d}{dz}\int|V|^2dVol_2$ is a negative constant, an impossibility due to the fact that $z$ has no upper bound by Lemma \ref{LemmaUnboundedZText} and $\int|V|^2dVol_2>0$; therefore $M^{4+}=\varnothing$.
		Likewise if $\triangle_4z>0$ then the second equation forces $\chi^g=0$ on $M^{4-}$, but then $\frac{d}{dz}\int|V|^2dVol_2$ is a positive constant, an impossibility because $z$ has no lower bound; therefore $M^{4-}=\varnothing$.
		
		Lastly we consider the possibility that $\triangle_4z=0$.
		If $M^4$ has two ends, then a line exists in $M^4$ so $M^4$ is flat by the Cheeger-Gromoll splitting theorem \cite{CG71}.
		Therefore $M^4$ has only one end.
		By Lemma \ref{LemmaOneOrTwoEnded}, if $M^{4-}=\varnothing$ then $z$ obtains a minimum, and if $M^{4+}=\varnothing$ then $z$ obtains a maximum.
		But because $\triangle_4z=0$, this contradicts the maximum principle.
		We conclude that a non-flat, Ricci-flat Type IIa metric with $\triangle_4z=0$ cannot exist.
	\end{proof}

	%
	%
	%
	%
	%
	%
	\subsection{A global bound $|V|>\epsilon$} \label{SubSecCompactAndLower}
	
	We study the case that $M^4$ has compact level-sets and a global bound $|V|>\epsilon$.
	By Lemma \ref{LemmaUnboundedZText}, every level-set $M^3_z$ is isotopic so in particular are diffeomorphic, and the manifold $M^4$ a product $M^3\times\mathbb{R}$ where $z:M^4\rightarrow\mathbb{R}$ is projection onto the second factor.
	Any one-variable function $\varphi:\mathbb{R}\rightarrow\mathbb{R}$ can be regarded as a function on $M^4$ via $\varphi\circ{}z:M^4\rightarrow\mathbb{R}$.
	If the level-sets are compact and $M^4$ is Type IIa, then $\varphi\circ{}z$ has compact support if and only if $\varphi$ has compact support.
	Before proving our main results we make two $L^2$-style integral estimates.
	\begin{lemma} \label{LemmaPrelimEllipticEst}
		Assume $M^4$ has compact level-sets, and $\varphi=\varphi(z)$ is a smooth 1-variable function with compact support on $M^4$.
		Then
		\begin{equation}
			\small
			\begin{aligned}
				\int\varphi(\triangle_4z)^2\,dVol_4
				=
				-\int\varphi'|\nabla{}z|^2(\triangle_4z)\,dVol_4
				+2\int\varphi\Ric(\nabla{}z,\nabla{}z)\,dVol_4.
			\end{aligned} \label{EqnFirstElliptic}
		\end{equation}
	\end{lemma}
	\begin{proof}
		When $\varphi$ has compact support, integration by parts gives
		\begin{equation}
			\small
			\begin{aligned}
			\int\varphi^2(\triangle_4z)^2\,dVol_4
			&=
			-2\int\left<\nabla\varphi,\nabla{}z\right>\varphi\triangle_4z\,dVol_4
			-\int\varphi^2\left<\nabla{}z,\,\nabla(\triangle_4z)\right>\,dVol_4.
			\end{aligned} \label{EqnIntByPartsEllipticStart}
		\end{equation}
		To simplify this , we use $\nabla(\triangle_4z)=-2\Ric(\nabla{}z,\cdot)$ from Lemma \ref{LemmaRicciBochner}, and we use $\nabla\varphi=\varphi'\nabla{}z$ to obtain $\left<\nabla\varphi,\nabla{}z\right>=\left<\varphi'\nabla{}z,\nabla{}z\right>=\varphi'|\nabla{}z|^2$.
		Substituting into (\ref{EqnIntByPartsEllipticStart}) gives(\ref{EqnFirstElliptic}).
	\end{proof}
	\begin{lemma} \label{LemmaRicEllipticEst}
		Assume $M^4$ has compact levelsets, a global lower bound $|V|>\epsilon>0$, and a global upper bound $|\Ric|\le\Lambda^2$.
		If $\varphi$ is a 1-variable function with compact support in $[z',z'']$ so that $M^3_{z'}$ is non-empty.
		Then
		\begin{equation}
			\begin{aligned}
			\int\varphi(\triangle_4z)^2\,dVol_4
			&\le
			\int_{z'}^{z''}|\varphi'|\left(2\Lambda^2|z-z'|+C'\right)
			\left(2\pi\chi^g(z-z_0)^2+C_0\right)\,dz \\
			&+2\Lambda^2\int_{z'}^{z''}|\varphi|
			\left(2\pi\chi^g(z-z_1)^2+C_1\right)\,dz
			\end{aligned}
		\end{equation}
		where the constants are defined as follows: $\chi^g$ is the constant $\chi^g(M^2_z)$, $z_0$ may be chosen arbitrarily, $C_0=\sup_{M^3_{z_0}}|\triangle_4z|$, $z_1$ is any value where $\int_{M^2_{z_1}}|\nabla{}z|^2dVol_2=\inf_{z}\int_{M^2_z}|\nabla{}z|^2dVol_2$ reaches a minimum, and $C_1=\int_{M^2_{z_1}}|\nabla{}z|^2dVol_2$.
	\end{lemma}
	\begin{proof}
		Pick any $z_0$.
		By Lemma \ref{LemmaRicciBochner} $\frac{\partial\triangle_4z}{\partial{}z}=-2\Ric\left(\frac{\nabla{}z}{|\nabla{}z|},\frac{\nabla{}z}{|\nabla{}z|}\right)$, so
		\begin{equation}
			\sup_{M^3_z}|\triangle_4z|
			\;\le\;
			2\Lambda^2|z-z_0|+
			\sup_{M^3_{z_0}}|\triangle_4z|.
		\end{equation}
		Abbreviating $C_0=\sup_{M^3_{z_0}}|\triangle_4z|$, Lemma \ref{LemmaPrelimEllipticEst} gives
		\begin{equation}
			\begin{aligned}
			\int\varphi(\triangle_4z)^2\,dVol_4
			&\le
			\int|\varphi'||\nabla{}z|^2\left(2\Lambda^2|z-z_0|+C_0\right)\,dVol_4 \\
			&\quad\quad+2\Lambda^2\int|\varphi||\nabla{}z|^2\,dVol_4
			\end{aligned}
		\end{equation}
		Then the Integration Lemma (Lemma \ref{LemmaIntegration}) gives
		\begin{equation}
			\begin{aligned}
			\int\varphi(\triangle_4z)^2\,dVol_4
			&\le
			\int_{z'}^{z''}|\varphi'|\left(2\Lambda^2|z-z'|+C'\right)
				\left(\int_{M^2_z}|\nabla{}z|^2\,dVol_2\right)dz \\
			&+2\Lambda^2\int_{z'}^{z''}|\varphi|\left(\int_{M^2_z}|\nabla{}z|^2\,dVol_2\right)dz
			\end{aligned} \label{EqnEllipticReducedAlmostToZ}
		\end{equation}
		Proposition \ref{PropGrowthOfVSqVol} gives $\frac{d^2}{dz}\int_{M^2_z}|\nabla{}z|^2dVol_2=4\pi\chi^g(M^2_z)$.
		Because the $M^2_z$ are all isotopic we can use the constant $\chi^g$ for $\chi^g=\chi^g(M^2_z)$.
		
		Because $z\mapsto\int_{M^2_z}|\nabla{}z|^2dVol_2$ is an upward opening parabola or constant, there exists some number $z_1$ so that $\int_{M^2_z}|\nabla{}z|^2dVol_2$ reaches a minimum of $C_0>0$ at $z=z_1$; therefore $\int_{M^2_z}|\nabla{}z|^2dVol_2=2\pi\chi^g(z-z_1)^2+C_1$.
		Therefore (\ref{EqnEllipticReducedAlmostToZ}) gives
		\begin{equation}
			\begin{aligned}
			\int\varphi(\triangle_4z)^2\,dVol_4
			&\le
			\int_{z'}^{z''}|\varphi'|\left(2\Lambda^2|z-z_0|+C_0\right)
				\left(2\pi\chi^g(z-z_1)^2+C_1\right)\,dz \\
			&+2\Lambda^2\int_{z'}^{z''}|\varphi|
				\left(2\pi\chi^g(z-z_1)^2+C_1\right)\,dz.
			\end{aligned}
		\end{equation}
	\end{proof}
	
	\begin{theorem}[cf. Theorem \ref{ThmIntroCptLevelsNoZeros}] \label{ThmTextCptLevelsNoZeros}
		Assume $M^4$ is non-flat, but scalar-flat and has a pointwise bound $|\Ric|\le\Lambda^2$.
		Assume $M^4$ has compact level-sets, and a global bound $|V|>\epsilon>0$.
		
		Then $M^4$ is 2-ended, $L^2(|\Ric|)=\infty$, the level-sets $M^3_z$ and fibrations $M^3_z\rightarrow{}M^2_z$ are all mutually isotopic, $e^g(M^3_z,M^2_z)=0$, and $\chi^g(M^2_z)>0$.
		Topologically, each $M^2_z$ is $\mathbb{S}^2$, each $M^3_z$ is $\mathbb{S}^2\times{}\mathbb{S}^1$, the fibration $M^3_z\rightarrow{}M^2_z$ is the projection $\mathbb{S}^2\times\mathbb{S}^1\rightarrow\mathbb{S}^2$, and $M^4$ is $\mathbb{S}^2\times\mathbb{S}^1\times\mathbb{R}$.
	\end{theorem}
	\begin{proof}
		By Lemma \ref{LemmaUnboundedZText} $M^4$ is 2-ended, $z$ ranges from $-\infty$ to $\infty$, and the values $e^g=e^g(M^3_z,M^2_z)$ and $\chi^g=\chi^g(M^2_z)$ are constant.
		To see that $e^g=0$, by Proposition \ref{PropConstantGrowthInText} we have
		\begin{equation}
			Vol(M^2_z)
			\;=\;2\pi{}e^g(M^3_z,M^2_z)\,z\;+\;C_3.
		\end{equation}
		Because $Vol(M^2_z)>0$ for all $z$, necessarily $e^g=0$.
		Thus $Vol(M^2_z)$ is constant.
		
		We use a proof by contradiction to prove strict inequality $\chi^g>0$.
		Assume to the contrary that $\chi^g=0$.
		Then by Proposition \ref{PropGrowthOfVSqVol} constants $C_0$ and $C_1$ exist so
		\begin{equation}
			\begin{aligned}
				\int_{M^2_z}\triangle_4z\,dVol_2
				&\;=\;C_2, \quad\text{and}\quad
				\int|V|^2\,dVol_2
				\;=\;C_2\cdot{}z\,+\,C_1.
			\end{aligned}
		\end{equation}
		Because $\int_{M^2_z}|V|^2dVol_2>0$ and $z$ ranges in $(-\infty,\infty)$, necessarily $C_2=0$ so
		\begin{equation}
			\begin{aligned}
				\int_{M^2_z}\triangle_4z\,dVol_2
				\;=\;0, \quad
				\int_{M^2_z}|V|^2\,dVol_2
				\;=\;C_1
			\end{aligned}
		\end{equation}
		Assume $\eta:M^4\rightarrow\mathbb{R}$ has compact support.
		By (\ref{EqnTransgressionOfRic}) we have $\int\eta|\Ric|^2dVol_4=-\int{}d\eta\wedge\mathcal{TP}{}_{\Ric}$.
		Using the expression for $\mathcal{TP}_{\Ric}$ in Lemma \ref{LemmaUsableCGB}, we have
		\begin{equation}
			2\int\eta|\Ric|^2
			\;=\;
			-\int{}d\eta\wedge\mathcal{L}_{\frac{\partial}{\partial{}z}}\left((\triangle_4z)^2\frac{1}{|V|}\,dVol_3\right).
		\end{equation}
		The Leibniz rule gives
		\begin{equation}
			2\int\eta|\Ric|^2
			\;=\;
			\int{}(\triangle_4z)^2\frac{1}{|V|}\;\mathcal{L}_{\frac{\partial}{\partial{}z}}\big(d\eta\big)\wedge{}dVol_3
		\end{equation}
		where we used the fact that $\mathcal{L}_{\frac{\partial}{\partial{}z}}=di_{\frac{\partial}{\partial{}z}}$ on 4-forms is a total derivative, which vanish under the integral.
		Now assume $\eta=\eta(z)$ is a function of one variable that has compact support.
		Then
		\begin{equation}
			\mathcal{L}_{\frac{\partial}{\partial{}z}}\big(d\eta\big)
			\;=\;\frac{\partial^2\eta}{\partial{}z^2}dz
		\end{equation}
		so abbreviating $\frac{\partial^2\eta}{\partial{}z^2}=\eta''$ we have
		\begin{equation}
			\begin{aligned}
			2\int\eta|\Ric|^2
			&\;=\;
			\int{}(\triangle_4z)^2\frac{1}{|V|}\;\eta''dz\wedge{}dVol_3
			\;=\;
			\int\eta''(\triangle_4z)^2\,dVol_4.
			\end{aligned}
		\end{equation}
		Using Lemma \ref{LemmaRicEllipticEst}, with $\varphi=\eta''$ and the assumption that $\chi^g=0$, we obtain
		\begin{equation}
			\begin{aligned}
			2\int\eta|\Ric|^2
			&\;\le\;
			C_1\int_{z'}^{z''}|\eta'''|\left(2\Lambda^2|z-z_0|+C_0\right)
			+2\Lambda^2C_1\int_{z'}^{z''}|\eta''|dz
			\end{aligned} \label{EqnMainRicciEst}
		\end{equation}
		where $\eta$ has support in $[z',z'']$, $z_0$ may be chosen arbitrarily, and $C_0=\sup_{M^2_{z_0}}|\triangle_4z|$.
		We select $z_0=0$ and choosing $\delta>0$ we select
		\begin{equation}
			\eta=
			\begin{cases}
				0, & z\notin[-\delta^{-1},\,\delta^{-1}] \\
				(1-(\delta\,z)^2)^4, & z\in[-\delta^{-1},\,\delta^{-1}].
			\end{cases}
		\end{equation}
		Then $\eta\in{}C^3_c$ so we may use the estimate (\ref{EqnMainRicciEst}).
		The integration in (\ref{EqnMainRicciEst}) may be carried out explicitly. We obtain
		\begin{equation}
			2\int\eta|\Ric|^2
			\;\le\;
			\frac{4608(3\sqrt{7}+4\sqrt{21})\Lambda^2C_1}{2401}\Lambda^2C_1\delta
			+\frac{1808}{49}C_0C_1\delta^2. \label{EqnEstIntCarriedOut}
		\end{equation}
		Noting that $\lim_{\delta\rightarrow0^+}\eta(z)\equiv1$, Eq. (\ref{EqnEstIntCarriedOut}) gives in the limit
		\begin{equation}
			\int|\Ric|^2\;=\;0
		\end{equation}
		and we have proven that if $\chi^g(M^2_z)=0$, then $M^4$ is Ricci-flat.
		But $M^4$ is 2-ended so it has a line, so by the Splitting Theorem \cite{CG71} the manifold is flat, contradicting the assumption that $M^4$ is non-flat. 
		We conclude that $\chi^g>0$.
		
		We have proven that $e^g=0$ and $\chi^g>0$.
		We must prove $\int|\Ric|^2=\infty$.
		By Propositions \ref{PropConstantGrowthInText} and \ref{PropGrowthOfVSqVol} we have constants $C_3>0$ and $C_1$ so
		\begin{equation}
			\begin{aligned}
				C_3\;=\;\int_{M^2_z}\,dVol_2
				\quad\text{and}\quad
				4\pi\chi^g\cdot{}z+C_1\;=\;\int_{M^2_z}\triangle_4z\,dVol_2.
			\end{aligned}
		\end{equation}
		By H\"older's inequality therefore
		\begin{equation}
			(4\pi\chi^g\cdot{}z\,+\,C_1)^2
			\;\le\;C_3\int_{M^2_z}(\triangle_4z)^2\,dVol_2
		\end{equation}
		which means the convex function $\int_{M^2_z}(\triangle_4z)^2\,dVol_2$ has quadratic growth (or greater).
		In particular $\frac{d}{dz}\int_{M^2_z}(\triangle_4z)^2\,dVol_2$ has no upper bound.
		But
		\begin{equation}
			\small
			\begin{aligned}
				\frac{d}{dz}\int_{M^2_z}(\triangle_4z)^2dVol_2
				&\;=\;\left.\frac{d}{dz}\right|_{z=0}\int_{M^2_z}(\triangle_4z)^2dVol_2
				\,+\,\int_{0}^z\left(\int_{M^2_z}|\Ric|^2dVol_2\right)\,dz \\
				&\;=\;\left.\frac{d}{dz}\right|_{z=0}\int_{M^2_z}(\triangle_4z)^2\,dVol_2
				\,+\,\int_{M^4_{0,z}}|\Ric|^2\,dVol_4
			\end{aligned}
		\end{equation}
		where we used the Integration Lemma (Lemma \ref{LemmaIntegration}) in the second line.    
		Because $\frac{d}{dz}\int_{M^2_z}(\triangle_4z)^2\,dVol_2$ has no upper bound, so $\int_{M^4_{0,z}}|\Ric|^2\,dVol_4$ has no upper bound.
		We conclude that $\int_{M^4}|\Ric|^2dVol_4=\infty$.
		
		Now the Seifert fibration $M^3_z\rightarrow{}M^2_z$ has positive orbifold Euler number for its base and zero fibration Euler number.
		A glance at the classification of Seifert fibered spaces---see \cite{Tol74} or Theorem 5.3 of \cite{Scott83}---states that $M^3_z$ has Thurston covering geometry $\mathbb{S}^2\times\mathbb{R}$.
		Because $M^3_z$ is compact, orientable, has orientable fibers, and an oriented base, it has the topology of $\mathbb{S}^2\times\mathbb{S}^1$ and the fibration $M^3_z\rightarrow{}M^2_z$ is the projection onto the $\mathbb{S}^2$ fiber.
		By Lemma \ref{LemmaUnboundedZText} and the fact that $z$ is unbounded from above and below means $M^4\approx\mathbb{S}^2\times\mathbb{S}^1\times\mathbb{R}$ where the projection onto the last factor $M^4\rightarrow\mathbb{R}$ is the function $z$.
	\end{proof}

	%
	%
	%
	%
	%
	%
	%
	%
	
	\section{Examples} \label{SecExamples}
	
	We present some illustrating examples of scalar-flat metrics with a Killing field.
	We refer the reader also to Sections \ref{SubsecEuclExample} and \ref{SubSectCGBExamples} in the text, which presented flat examples and LeBrun instanton examples in detail.
	
	\subsection{Type I examples} \label{SubSecTypeIExamples}
	
	The easiest example is $\mathbb{S}^2\times\Sigma^2$ where $\mathbb{S}^2$ has its unit round metric and $\Sigma^2$ is a surface of negative genus.
	More examples come from blowing (inserting $\mathbb{P}^1$ divisors of negative self-intersection) up at zeros of the Killing field, see for example \cite{Szekl15} and references therein.
	Compact scalar-flat metrics have been studied extensively and there are too many references to mention; see for example \cite{KLP97}, \cite{CTF98}, \cite{Szekl09}, \cite{HS02}, \cite{LeB20}.
	
	\subsection{The flat and hyperbolic examples} \label{SubSecTypeIIExamples}
	
	The simplest example is $\mathbb{C}^2$ with two rotational fields generated by the diffeomorphism flow $\psi_t.(z_1,z_1)=(e^{2\pi\alpha}z_1,e^{2\pi\beta}z_2)$.
	If $\alpha$ and $\beta$ have the same sign this is a Type IIa example.
	If $\alpha$ or $\beta$ is zero, or if they have opposite signs, the levelsets are non-compact and this is a Type IIb example.
	
	A wider range of phenomena is captured in the $\mathbb{S}^2\times\mathbb{H}^2$ examples.
	With $\theta_1\in[0,2\pi)$ the flow parameter on $\mathbb{S}^2$ and $\theta_2\in[0,2\pi)$ the flow parameter on $\mathbb{H}^2$, the metric is
	\begin{equation}
		g=g_{\mathbb{S}^2}+g_{\mathbb{H}^2}
		=\Big((dr^1)^2+(\sin(r)d\theta_1)^2\Big)+\Big((dr^2)^2+(f(r)d\theta_2)^2\Big)
	\end{equation}
	where $f(r^2)$ satisfied $f''(r^2)=f(r^2)$; the three cases are $f(r^2)=\sinh(r^2)$, $r^2\in[0,\infty)$ where the Killing field is elliptic, $f(r^2)=e^{r^2}$, $r^2\in(-\infty,\infty)$ where the Killing field is parabolic, and $f(r^2)=\cosh(r^2)$, $r^2\in(-\infty,\infty$ where the Killing field is hyperbolic.
	On the $\mathbb{S}^1$ factor $r^1\in[-\pi,\pi]$.
	We have $\theta_1,\theta_2\in[0,2\pi)$.
	We find potential functions in the usual way by solving $\nabla{}z^i=J\frac{\partial}{\partial\theta^i}$, and arrive at $z^1=-\cos(r^1)$ and the three cases $z^2=\cosh(r^2)$, $z^2=e^{r^2}$, and $z^2=\sinh(r^2)$.
	Changing to momentum variables, we find in the three cases
	\begin{equation}
		\small
		g
		=
		\frac{1}{1-(z^1)^2}\left(dz^1\right)^2
		+\left(1-(z^1)^2\right)\left(d\theta_1\right)^2
		+
		\begin{cases}
			\frac{1}{(z^2)^2-1}\left(dz^2\right)^2+((z^2)^2-1)\left(d\theta_2\right)^2 & \\
			\frac{1}{(z^2)^2}\left(dz^2\right)^2+(z^2)^2\left(d\theta_2\right)^2 & \\
			\frac{1}{(z^2)^2+1}\left(dz^2\right)^2+((z^2)^2+1)\left(d\theta_2\right)^2
		\end{cases}
	\end{equation}
	where $z^1\in[-1,1]$ and the coordinate ranges for $z^2$ are $z^2\in(1,\infty)$, $z^2\in(0,\infty)$, $z^2\in(-\infty,\infty)$ for the three cases respectively.
	The momentum reduction is the map $\Phi:(z^1,z^2):M^4\rightarrow\mathbb{R}^2$.
	The images $\Sigma^2=\Phi(M^4)\subset\mathbb{R}^2$, in the three cases, are depicted in Figure \ref{FigThreeCases}.
	\begin{figure}[ht]
		\centering
		\includegraphics[scale=0.45]{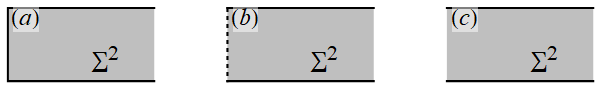}
		\caption{
			Momentum reductions for the $\mathbb{S}^2\times\mathbb{H}^2$ metrics:
			(a) the elliptic case, $\Sigma^2$ is a closed half-strip;
			(b) the parabolic case, $\Sigma^2$ is a semi-closed half-strip; and
			(c) the hyperbolic case, $\Sigma^2$ is a closed strip.
		}
		\label{FigThreeCases}
	\end{figure}
	The elliptic and hyperbolic cases are both Type IIa examples.
	In the parabolic case, the manifold has a cusp-like end as $z^2\searrow0$ an the norm of the field $\frac{\partial}{\partial\theta_2}$ proceeds to zero; therefore the parabolic case is a Type III example.
	
	\subsection{Type II examples}
	Type II examples are plentiful, not least because the Ricci-flat case has been so thoroughly studied.
	The examples of the LeBrun instantons of Section \ref{SubSectCGBExamples} are all Type IIa; see \cite{LeB89}.
	Other Type II examples include the Taub-NUT, multi-Taub-NUT, Eguchi-Hanson, and multi-Eguchi-Hanson examples \cite{GiHa78} \cite{EH79}, the Kronheimer examples \cite{Kro89}, the Pedersen-Poon examples \cite{PP91}, the Hwang-Singer examples \cite{HS02}, the Calderbank-Singer examples \cite{CS04}, the Calderbank-Pedersen examples \cite{CP02}, the Donaldson generalized Taub-NUTs \cite{Don09}, and the exceptional Taub-NUT of \cite{Weber23}.
	There is a general classification in the compact case in \cite{ACGT04}.
	The exceptional Taub-NUT is Type IIa, but neither ALE nor asymptotically conical, and has infinite $L^2(|\Ric|)$.
	A naive hope that Type IIa metrics might automatically have some asymptotic geometric rigidity seems to be false.

	\subsection{Type III examples} \label{SubSecTypeIIIExamples}
	
	Type III examples are characterized by no asymptotic lower bound on $|V|$, either because the zero-locus is non-compact (Type IIIb), or the zero-locus is compact but there is no asymptotic lower bound on $|V|$ (Type IIIa).
	Type IIIa is problematic mainly for the reason that level-sets can change between compact and non-compact.
	
	The ``exceptional half-plane'' metric is a toric scalar-flat manifold with $L^2(\Ric)=\infty$ that provides an example of both Type IIIa and IIIb.
	This manifold has underlying complex manifold $\mathbb{C}^2$ with one Killing field being translational and the other being rotational.
	The rotational field gives a Type IIIb example; its zero-locus is an unbounded totally geodesic copy of $\mathbb{C}$.
	The translational field gives a type IIIa example; its Killing field is nowhere-zero but has no lower-bound.
	
	\begin{figure}[ht]
		\centering
		\includegraphics[scale=0.5]{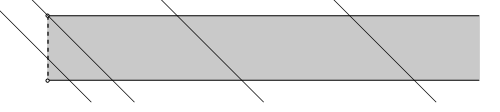} \\
		\hspace{0.2in}
		\includegraphics[scale=0.5]{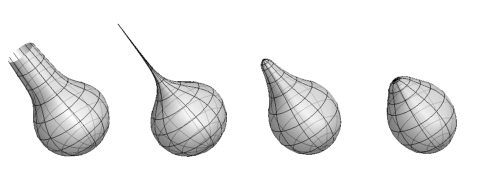}
		\caption{
			Topology change of the reduced $M^2_z$ in the $\mathbb{S}^2\times\mathbb{H}^2$ example, where the Killing field is non-zero but may approach zero asymptotically.
			Diagonal lines indicate various values of $z=z^1+z^2$ in the $z^1$-$z^2$ momentum plane.
			Below are representations of the corresponding reductions $M^2_z$.
		}
		\label{FigTopologyChange}
	\end{figure}
	For an example with an unbounded zero-locus that consists of isolated points, see \cite{AKL89} which has an example with infinite topological type.
	
	When Theorem \ref{ThmCptNonCpt} fails, it seems difficult to even begin an analysis.
	To see this actually occur, consider parabolic $\mathbb{S}^2\times\mathbb{H}^2$ case depicted in Figure~\ref{FigThreeCases}(b).
	Any constant-coefficient combination of the momentum functions $z^1$, $z^2$ gives another momentum function and corresponding Killing field.
	Taking $z=z^1+z^2$, we see that $M^3_z$ is compact when $z>0$ and non-compact when $0<z\le1$; soe Figure \ref{FigTopologyChange} for a depiction of this phenomenon.

\end{document}